\documentclass[11pt]{article}

\usepackage[T1]{fontenc}
\usepackage{lmodern}
\usepackage{amsmath,amssymb,amsthm}
\usepackage{mathrsfs}
\usepackage[margin=1in]{geometry}
\usepackage{enumitem}
\usepackage{xcolor}
\usepackage{tikz}
\usetikzlibrary{calc,positioning}
\definecolor{diagramred}{RGB}{196,45,52}
\definecolor{diagramblue}{RGB}{29,104,175}
\tikzset{
  s23edge/.style={line width=.85pt,line cap=round,line join=round,black!78},
  s23upper/.style={line width=1.2pt,line cap=round,line join=round,diagramred},
  s23lower/.style={line width=1.2pt,line cap=round,line join=round,diagramblue},
  s23upperfaint/.style={line width=.7pt,line cap=round,line join=round,diagramred!55!black},
  s23lowerfaint/.style={line width=.7pt,line cap=round,line join=round,diagramblue!55!black},
  s23return/.style={line width=1.35pt,line cap=round,line join=round,diagramblue},
  s23common/.style={line width=1.25pt,line cap=round,line join=round,black},
  s23vertex/.style={circle,fill=black,inner sep=1.35pt},
  s23label/.style={font=\small,fill=white,inner sep=1.1pt},
  s23tiny/.style={font=\footnotesize,fill=white,inner sep=.8pt},
  s23bracket/.style={line width=.65pt,black!65},
  s23cut/.style={line width=.65pt,densely dashed,black!60}
}
\usepackage[colorlinks=true,linkcolor=blue!45!black,citecolor=blue!45!black,urlcolor=blue!45!black]{hyperref}
\usepackage{microtype}

\newtheorem{theorem}{Theorem}[section]
\newtheorem{lemma}[theorem]{Lemma}
\newtheorem{proposition}[theorem]{Proposition}
\newtheorem{corollary}[theorem]{Corollary}
\theoremstyle{definition}
\newtheorem{definition}[theorem]{Definition}
\newtheorem*{theoremA}{Theorem A}
\newtheorem*{theoremB}{Theorem B}
\newtheorem*{theoremC}{Theorem C}

\newcommand{\cells}{\mathcal C}
\newcommand{\Path}{\mathcal P}
\newcommand{\K}{\mathcal K}
\newcommand{\Cl}{\operatorname{Cl}}
\newcommand{\D}{\mathcal D}
\newcommand{\topPath}{\operatorname{top}}
\newcommand{\botPath}{\operatorname{bot}}

\title{Finiteness properties of closed subgroups of Thompson's group \(F\)}
\author{Gili Golan}
\date{}

\begin{document}
\maketitle

\begin{abstract}
A subgroup \(H\) of Thompson's group \(F\) is \emph{closed} if every
piecewise-\(H\) function in \(F\) belongs to \(H\).  Prominent examples
of closed subgroups are the maximal subgroups of infinite index,
stabilizers and pointwise stabilizers of sets of points, and many of
Jones' subgroups.  A closed subgroup is
finitely generated if and only if its Stallings \(2\)-core is finite,
and it is then isomorphic to a diagram group over the core, in the
sense of Guba and Sapir.  We study the finiteness properties of closed
finitely generated subgroups of \(F\) with finitely many orbits on the
dyadic rationals.  For such a subgroup \(H\) the following are
equivalent: \(H\) is of type \(\mathrm{FP}_2\); \(H\) is finitely
presented; \(H\) is of type \(F_\infty\); \(H\) has finitely many
orbits on pairs of dyadic rationals; the action of \(H\) on the dyadic
rationals is oligomorphic.  If moreover the action of \(H\) on
\((0,1)\) is minimal, these conditions hold if and only if the image of
\(H\) in the abelianization \(\mathbb Z^2\) of \(F\) has rank two.  All
these conditions can be decided from the core of \(H\), and when they
hold a finite presentation of \(H\) can be computed.  As a consequence,
every finitely generated maximal subgroup of \(F\) is of type
\(F_\infty\).
\end{abstract}

\section{Introduction}
\label{sec:introduction}

Thompson's group \(F\) is the group of piecewise-linear
homeomorphisms of \([0,1]\) with finitely many breakpoints, all of them
dyadic rationals, and with slopes powers of \(2\).  Brown and Geoghegan
proved that \(F\) is of type \(F_\infty\); it was the first example of
a torsion-free group of type \(\mathrm{FP}_\infty\) with infinite
cohomological dimension \cite{BG84}.  The finiteness properties of
subgroups of \(F\) are much less understood.  Finitely generated
subgroups of \(F\) need not be finitely presented,
and no general method is known for deciding whether a given finitely
generated subgroup of \(F\) is finitely presented or of type
\(F_\infty\).  In this paper we answer these questions for a large
class of subgroups, the finitely generated closed subgroups with
finitely many orbits on the dyadic rationals.  For these subgroups we show that type
\(\mathrm{FP}_2\), finite presentability and type \(F_\infty\) are
equivalent to each other and to a simple dynamical condition, and that
they can be decided algorithmically; moreover, when they hold, results
of Guba and Sapir yield an algorithm for finding a finite presentation.

\subsection*{Closed subgroups and their cores}

Let \(\D=\mathbb Z[1/2]\cap(0,1)\) be the set of dyadic rationals in
\((0,1)\).  Let \(H\leq F\).  A function \(f\in F\) is
\emph{piecewise-\(H\)} if there is a finite subdivision of \([0,1]\)
such that on each interval of the subdivision \(f\) agrees with some
element of \(H\).  The set of all piecewise-\(H\) functions in \(F\) is
a subgroup of \(F\) containing \(H\), called the \emph{closure} of
\(H\) and denoted \(\Cl(H)\); the subgroup \(H\) is \emph{closed} if
\(H=\Cl(H)\).  Closed subgroups of \(F\) were introduced in
\cite{GS17} and studied further in
\cite{GolanMemo,GPSclosed,GolanMaximal}.
Many natural subgroups of \(F\) are closed.  Every maximal subgroup of
\(F\) of infinite index is closed \cite[Theorem~1.1]{GolanMaximal}.
The Jones subgroup \(\vec F\) and its generalizations \(\vec F_n\)
are closed \cite{GSjones,GPSclosed}, as are other Jones subgroups
studied in \cite{AielloNagnibeda1,AielloNagnibeda2}.
Stabilizers and pointwise stabilizers of sets of points in \((0,1)\)
are closed as well.  Finitely generated closed subgroups of \(F\) are undistorted
\cite{GPSclosed}.

Closed subgroups are best understood through the \emph{Stallings
\(2\)-core}, introduced by Guba and Sapir \cite{GS17} as a
two-dimensional analogue of the Stallings core graph of a subgroup of a
free group.  It is built in the framework of directed
\(2\)-complexes and diagram groups of Guba and Sapir
\cite{GubaSapir97,GS06}, which we now describe informally.  A
\emph{directed \(2\)-complex} consists of vertices, directed edges,
and \(2\)-cells; each \(2\)-cell is attached along two directed paths
with the same initial vertex and the same terminal vertex, its
\emph{top path} and its \emph{bottom path}.  The cells are divided
into positive cells and their inverses: the inverse of a cell with top
path \(P\) and bottom path \(Q\) is a cell with top path \(Q\) and
bottom path \(P\).  A \emph{diagram} over a directed \(2\)-complex is a
finite planar picture assembled from copies of the cells, similarly to
the way a van Kampen diagram is assembled from relator cells; it has a
top path and a bottom path, both connecting its initial vertex to its
terminal vertex.  Diagrams whose
top and bottom paths are labelled by the same path \(w\) can be
multiplied by stacking them vertically, and reduced by cancelling
adjacent mutually inverse cells.  The reduced diagrams with top and
bottom paths labelled by \(w\) form the \emph{diagram group} with base
\(w\).  Thompson's group \(F\) is itself a diagram group: it is
isomorphic to the diagram group with base \(x\) over the \emph{Dunce
hat}, the directed \(2\)-complex with one vertex, one directed edge
\(x\), and one positive cell with top path \(x\) and bottom path
\(xx\) \cite{GubaSapir97}.  Under
this isomorphism, the diagram of an element of \(F\) is obtained from a
pair of binary trees representing it by gluing the two trees along
their leaves and replacing every caret of the first tree by a positive
cell and every caret of the second tree by an inverse cell.

The Stallings \(2\)-core \(C(H)\) of a subgroup \(H\leq F\) is
obtained by taking the reduced diagrams of the elements of a generating
set of \(H\), identifying all their top and bottom edges to a single
edge \(p_H\), and then folding: whenever two cells have the same top
path or the same bottom path, they are identified.  The result is a
directed \(2\)-complex which does not depend on the chosen generating
set.  The construction shows that the core is finite whenever \(H\) is
finitely generated.  The core inherits the shape of
the cells of the diagrams of elements of \(F\): every positive cell of
\(C(H)\) has a single top edge and a bottom path of length two, and,
by foldedness, every edge is the top edge of at most one positive cell.
We call \(C(H)\) \emph{full} if every edge is the top edge of a
(necessarily unique) positive cell.  The edge \(p_H\) runs from the
\emph{initial vertex} of the core to its \emph{terminal vertex}; all
other vertices are called \emph{inner}, and an edge is \emph{inner} if
both its endpoints are inner vertices.

The core detects membership in the closure.  The diagram of an element
of \(F\) is \emph{accepted} by \(C(H)\) if its edges and cells can be
labelled by edges and cells of \(C(H)\), respecting incidences, so
that its top and bottom edges are labelled by \(p_H\).  An element of
\(F\) belongs to \(\Cl(H)\) if and only if its reduced diagram is
accepted by \(C(H)\) \cite{GolanMemo}.  Consequently \(\Cl(H)\) is
isomorphic to the diagram group over \(C(H)\) with base \(p_H\)
\cite{GolanMemo}.  Moreover, \(H\) and \(\Cl(H)\) have the same
core, and if the core of \(H\) is finite then \(\Cl(H)\) is finitely
generated \cite{GPSclosed}.  Thus a closed subgroup of \(F\) is
finitely generated if and only if its core is finite, and closed
finitely generated subgroups of \(F\) are exactly the diagram groups
over finite Stallings \(2\)-cores with base the distinguished edge.

This brings the finiteness properties of closed subgroups within reach
of the theory of diagram groups.  Two directed paths in a directed
\(2\)-complex \(\mathcal L\) are \emph{homotopic} if one is obtained
from the other by repeatedly replacing the top path of a cell by its
bottom path or vice versa.  We write \(\Path(\mathcal L)\) for the set
of homotopy classes of nonempty directed paths in \(\mathcal L\); with the
partial operation of concatenating composable paths, it is the
\emph{path semigroupoid} of \(\mathcal L\).  Guba and Sapir proved
that if \(\mathcal L\) is finite and \(\Path(\mathcal L)\) is finite,
then every diagram group over \(\mathcal L\) is of type \(F_\infty\)
\cite[Theorem~9.3]{GS06}.  Our main theorem shows that for a finite full
core \(\mathcal L=C(H)\) the diagram group over \(\mathcal L\) with
base \(p_H\) is of type \(\mathrm{FP}_2\) if and only if
\(\Path(\mathcal L)\) is finite, and characterizes the finiteness of \(\Path(\mathcal L)\)
in terms of the action of \(H\) on the dyadic rationals.

\subsection*{Main results}

Let \(H\) be a finitely generated closed subgroup of \(F\).  Then
\(C(H)\) is full if and only if \(H\) has finitely many orbits on
\(\D\) \cite{GolanMemo}; the alternative is that some open interval
contains at most one dyadic point from each \(H\)-orbit.  For a full
core, the inner vertices of \(C(H)\) correspond to the \(H\)-orbits on
\(\D\) \cite{GolanMemo}.  Write
\[
  \D^2_<=\{(a,b)\in\D^2:a<b\}
\]
for the set of increasing pairs of dyadic rationals.  Recall that an
action of a group on a set \(X\) is \emph{oligomorphic} if it has
finitely many orbits on \(X^n\) for every \(n\geq1\).

\begin{theoremA}[Theorem~\ref{thm:finite-full-core-equivalence} and
Proposition~\ref{prop:dyadic-tuple-orbits}]
Let \(H\leq F\) be a closed finitely generated subgroup with finitely
many orbits on \(\D\).  The following conditions are equivalent.
\begin{enumerate}[label=\textup{(\arabic*)}]
\item \(H\) is of type \(\mathrm{FP}_2\).
\item \(H\) is finitely presented.
\item \(H\) is of type \(F_\infty\).
\item The core \(C(H)\) has only finitely many homotopy classes of
      directed paths.
\item \(H\) has finitely many orbits on \(\D^2_<\).
\item The action of \(H\) on \(\D\) is oligomorphic.
\end{enumerate}
\end{theoremA}

The equivalence of \textup{(4)} and \textup{(5)} follows from a
correspondence, established in Section~\ref{sec:interval-paths},
between homotopy classes of directed paths in a full core and
\(H\)-orbits of increasing pairs of dyadic rationals
(Corollary~\ref{cor:interval-pair-dictionary}).  Similarly, the action
of \(H\) on increasing \(n\)-tuples of dyadic rationals can be
described in terms of homotopy classes of directed paths in the core
(Proposition~\ref{prop:tuple-transport}), which gives the equivalence
of \textup{(4)} and \textup{(6)}.  The implication from \textup{(4)}
to \textup{(3)} is the theorem of Guba and Sapir, and the implications
from \textup{(3)} to \textup{(2)} to \textup{(1)} are general.  The
heart of the paper is the implication from \textup{(1)} to
\textup{(4)}.

The conditions of Theorem~A can moreover be decided algorithmically.

\begin{theoremB}[Theorem~\ref{thm:finite-full-core-decision}]
There is an algorithm which, given the finite full core of a closed
subgroup \(H\leq F\), decides whether \(H\) is of type
\(\mathrm{FP}_2\), or equivalently whether it is of type \(F_\infty\).
\end{theoremB}

When the answer is positive, the core has finitely many homotopy
classes of paths, and the results of Guba and Sapir \cite{GS06} yield
an algorithm for finding a finite presentation of \(H\) and, for each
\(n\geq0\), computing the rank of \(H_n(H;\mathbb Z)\)
(Corollary~\ref{cor:effective-presentations-homology}).

An important application of Theorem~A concerns maximal subgroups.  The
maximal subgroups of \(F\) of infinite index were studied in
\cite{GolanMaximal}, where it was asked whether every finitely
generated maximal subgroup of \(F\) is finitely presented
\cite[Problem~8.8]{GolanMaximal}.  Maximal subgroups of infinite index
are closed and have full cores, so Theorem~A applies, and an analysis
of their possible dynamics on \((0,1)\) yields the following.

\begin{theoremC}[Theorem~\ref{thm:finitely-generated-maximal-subgroups-finfty}]
Every finitely generated maximal subgroup of Thompson's group \(F\) is
of type \(F_\infty\).
\end{theoremC}

Section~\ref{sec:applications} contains further consequences of
Theorem~A.  Let \(H\) be closed with finite full core and of type
\(\mathrm{FP}_2\).  Then for every closed interval \(J\subseteq[0,1]\)
with dyadic endpoints, the subgroup of \(H\) of all elements supported
in \(J\) is of type \(F_\infty\), and a finite presentation of it can
be computed (Proposition~\ref{prop:supported-subgroups}); every
nonempty open subinterval of \((0,1)\) supports a copy of \(F\) inside
\(H\) (Proposition~\ref{prop:locally-supported-copies-of-F}); and every
closed subgroup of \(F\) containing \(H\) is again of type
\(F_\infty\) (Proposition~\ref{prop:closed-overgroups-finfty}).

\subsection*{Outline of the proof}

Let \(H\leq F\) be closed with finite full core \(\mathcal L=C(H)\),
and let \(A\leq\mathbb Z^2\) be the image of \(H\) under the
abelianization map \(F\to\mathbb Z^2\), which records the logarithms of
the slopes at \(0\) and at \(1\).  It is not hard to see that both
coordinate projections of \(A\) are nonzero, so \(A\) has rank one or
two.  The proof of Theorem~A proceeds by a case analysis on the rank of
\(A\) and on the dynamics of \(H\) on \((0,1)\).

\emph{Rank one.}  If \(A\) has rank one, then \(H\) is not of type
\(\mathrm{FP}_2\) (Theorem~\ref{thm:rank-one-not-FP2}), and at the same
time \(\mathcal L\) has infinitely many homotopy classes of paths
(Corollary~\ref{cor:rank-one-infinite-paths}).  The first statement is
proved by combining the theorem of Bieri and Strebel on groups of type
\(\mathrm{FP}_2\) without free subgroups \cite{BieriStrebel} with the
fact that the breakpoints of the elements of a finitely generated
subgroup of \(F\) lie in finitely many orbits
\cite{BleakBroughHermiller21}.  It therefore remains to treat the case
in which \(A\) has rank two.

\emph{Rank two, minimal action.}  Suppose that \(A\) has rank two and
that the action of \(H\) on \((0,1)\) is minimal, that is, every orbit
is dense.  We prove that the homotopy classes of nonempty directed paths
between inner vertices of \(\mathcal L\) form a groupoid, and that the number
of such classes between any two inner vertices is exactly the index of
\(A\) in the smallest subgroup of the form
\(d_L\mathbb Z\times d_R\mathbb Z\) containing it
(Proposition~\ref{prop:exact-inner-path-count}).  In particular
\(\Path(\mathcal L)\) is finite and \(H\) is of type \(F_\infty\)
(Corollary~\ref{cor:minimal-rank-two-finfty}).  Thus, for minimal
actions, the finiteness properties of \(H\), and even the number of
\(H\)-orbits on \(\D^2_<\), depend only on the image of \(H\) in the
abelianization of \(F\) and on the number of \(H\)-orbits on \(\D\).

\emph{Rank two, nonminimal action.}  Suppose finally that \(A\) has
rank two and that the action is not minimal.  Then either \(H\) has
fixed points in \((0,1)\), or it has no fixed points and possesses a
unique exceptional minimal set, a perfect nowhere dense closed
invariant set (Proposition~\ref{prop:nonminimal-rank-two-alternative});
the rank assumption excludes a closed discrete orbit.  In both cases we
reduce the problem to closed subgroups with smaller cores.  If \(H\)
has fixed points, there are finitely many of them, all rational, and
the restriction of \(H\) to each of its orbitals, transported to
\((0,1)\) by a canonical chart, is again a closed
subgroup with finite full core, and one with fewer inner edges than
\(\mathcal L\).  We show that \(\Path(\mathcal L)\) is finite if and
only if the corresponding statement holds for all these restrictions
(Corollary~\ref{cor:fixed-points-orbital-reduction}).  If \(H\) has an
exceptional minimal set \(M\), then \(H\) permutes the gaps of \(M\),
finitely many gaps meet every orbit of gaps, and the stabilizer of a
gap restricted to that gap is again, after normalization, a closed
subgroup with finite full core and fewer inner edges.  Here the
homotopy classes represented by intervals meeting \(M\) are shown to be
finite in number by a variant of the argument in the minimal case,
while the remaining classes live inside the gaps, and again
\(\Path(\mathcal L)\) is finite if and only if the corresponding
statement holds for the restrictions of \(H\) to these finitely many
gaps (Corollary~\ref{cor:exceptional-gap-reduction}).

Since the number of inner edges decreases at each step, after finitely
many steps this reduction arrives at closed subgroups whose image in
the abelianization has rank one or whose action on \((0,1)\) is
minimal.  It follows that \(\Path(\mathcal L)\) is infinite if and only
if one of the subgroups obtained in the process has rank-one image
(Proposition~\ref{prop:rank-one-isolated-characterization}).  To
conclude that \(H\) itself is then not of type \(\mathrm{FP}_2\), we
show that such a subgroup is a quasi-retract of \(H\) in the sense of
Alonso, and use his theorem that type \(\mathrm{FP}_2\) passes to
quasi-retracts \cite{Alonso} (Theorem~\ref{thm:rank-one-isolated-not-FP2}).
If, on the other hand, \(\Path(\mathcal L)\) is finite, then \(H\) is
of type \(F_\infty\) by the theorem of Guba and Sapir.

Every step of this reduction is effective.  The core is conveniently
encoded as a finite automaton whose states are the edges of the core and
whose transitions lead from an edge to the two bottom edges of its cell.
The rank of \(A\), minimality of the action, the fixed points, a
finite automaton description of the exceptional minimal set,
representatives of its gap orbits, and the cores of the restrictions
can all be computed from this automaton
(Theorem~\ref{thm:minimality-child-automaton},
Propositions~\ref{prop:fixed-points-cyclic-components},
\ref{prop:gap-orbit-representatives}
and~\ref{prop:normalized-core-effective}), which gives Theorem~B.

\subsection*{The non-full case}

The hypothesis in Theorem~A that \(H\) has finitely many orbits on
\(\D\), that is, that the core of \(H\) is full, cannot simply be
dropped.  For instance, the cyclic subgroup generated by an element of
\(F\) with no fixed points in \((0,1)\) is closed and, being infinite
cyclic, is of type \(F_\infty\); but its core has infinitely many
homotopy classes of directed paths and it has infinitely many orbits on
\(\D\), so conditions \textup{(4)}, \textup{(5)} and \textup{(6)} of
Theorem~A fail for it.  However, we expect that type \(\mathrm{FP}_2\)
implies type \(F_\infty\) for every closed finitely generated subgroup
of \(F\), whether or not its core is full, and we are currently
working on this question.

\subsection*{Organization}

Section~\ref{sec:preliminaries} gives the necessary background on
\(F\), directed \(2\)-complexes and diagram groups, the Stallings
\(2\)-core and its automaton, and orbitals and minimal sets of group
actions on intervals.  Section~\ref{sec:interval-paths} establishes the
correspondence between homotopy classes of paths in a full core and
orbits of intervals with dyadic endpoints, and relates the image of \(H\) in the
abelianization of \(F\) to certain loops in the core.
Section~\ref{sec:rank-one-endpoint-image} treats the rank-one case and
Section~\ref{sec:minimal-rank-two} the minimal rank-two case.
Section~\ref{sec:rational-isolated-intervals} studies
\(H\)-isolated intervals, that is, open intervals whose translates
under \(H\) are pairwise equal or disjoint, and the restriction of
\(H\) to such an interval; it also proves the rank-one obstruction via
quasi-retracts.  Section~\ref{sec:nonminimal-rank-two} handles the
nonminimal rank-two case.
Section~\ref{sec:finiteness-properties-finite-full-core} proves
Theorems~A and~B, and Section~\ref{sec:applications} contains
Theorem~C and other applications.

\subsection*{Acknowledgements}

The author was partially supported by the Israel Science Foundation
(grant No.~2275/24).  The mathematical arguments in this paper were
developed in part through discussions with ChatGPT (OpenAI), which also
assisted with the preparation of the manuscript.

\section{Preliminaries}
\label{sec:preliminaries}

\subsection{Binary words, dyadic intervals, and Thompson's group
\texorpdfstring{\(F\)}{F}}
\label{subsec:F}

Let \(\{0,1\}^*\) denote the set of finite binary words, including the
empty word \(\varnothing\), and let \(\{0,1\}^{\mathbb N}\) denote the
set of infinite binary words.  We write \(|u|\) for the length of a
finite word \(u\).  For a finite or infinite binary word \(\eta\), let
\(.\eta\in[0,1]\) denote the number with binary expansion \(\eta\); for
a finite word, trailing zeros are understood.

For \(u\in\{0,1\}^*\), put
\begin{equation}
\label{eq:standard-interval}
  [u]=[.\!u,.\!u1^\infty].
\end{equation}
We call \([u]\) a \emph{standard dyadic interval}.  Two such intervals
have intersecting interiors precisely when one of their defining words
is a prefix of the other.

We write
\[
  \D=\mathbb Z[1/2]\cap(0,1)
\]
for the set of dyadic points in \((0,1)\).  Every \(a\in\D\) has a
unique representation \(a=.\!u\), where \(u\) is a finite binary word
ending in \(1\).

Throughout, a binary tree is a rooted plane tree in which every vertex
has either no children or exactly two children, designated left and
right.  A \emph{caret} is a vertex together with its two children.

Intervals used for dyadic subdivisions or interval classes are always
assumed to be nondegenerate.

In a finite rooted binary tree, label every left edge by \(0\) and every
right edge by \(1\).  The word read along the path from the root to a
vertex is the \emph{address} of that vertex.  If the leaf addresses of a
tree \(T\), listed from left to right, are \(u_1,\ldots,u_m\), then the
intervals
\[
  [u_1],\ldots,[u_m]
\]
have pairwise disjoint interiors and cover \([0,1]\).  They form the
\emph{dyadic subdivision} associated with \(T\).  Every finite set of
dyadic points is contained in the set of subdivision points of some
finite binary tree.  Every closed interval with dyadic endpoints is
also a finite union of consecutive standard dyadic intervals.

A \emph{finite complete prefix code} is a finite set of binary words
such that every infinite binary word has exactly one of them as a
prefix.  The leaf addresses of a finite binary tree form such a code.

A dyadic subdivision of a closed interval with dyadic endpoints is a
finite subdivision into standard dyadic intervals.  For \(I=[0,1]\),
these are precisely the dyadic subdivisions associated with finite
rooted binary trees.

Recall that Thompson's group \(F\) consists of the increasing
piecewise-linear homeomorphisms of \([0,1]\) having finitely many
breakpoints, all dyadic, and slopes integral powers of \(2\); see
\cite{CFP}.

Throughout, maps are composed from left to right. Thus, if
\(f\colon X\to Y\) and \(g\colon Y\to Z\), then
\[
  fg=f\circ g\colon X\to Z,
  \qquad (fg)(x)=g(f(x)).
\]
We retain the usual notation \(f(x)\) for evaluation. In particular,
the natural action of a group of homeomorphisms is a right action,
written \(x\cdot f=f(x)\). The same convention applies to the action
on intervals and on tuples of points. For a subgroup \(G\), we also
use \(Gx\) as shorthand for the orbit set
\(\{g(x):g\in G\}\).

Let \(T_+\) and \(T_-\) be finite rooted binary trees with the same
number of leaves, labelled respectively
\[
  u_1,\ldots,u_m
  \qquad\text{and}\qquad
  v_1,\ldots,v_m
\]
from left to right.  The tree pair \((T_+,T_-)\) represents the element
\(f\in F\) which maps each interval \([u_i]\) linearly onto \([v_i]\).
Thus \(T_+\) is the domain tree and \(T_-\) is the range tree.  The
corresponding pair of leaf addresses
\[
  u_i\longrightarrow v_i
\]
is called a \emph{branch pair} of the tree pair.

We say that \(f\) has the branch pair \(u\to v\) if \(u\to v\) occurs
in some tree pair representing \(f\).  Equivalently, \(f\) maps
\([u]\) linearly onto \([v]\).

Replacing the leaves with addresses \(u\) and \(v\) by the pairs of
leaves \(u0,u1\) and \(v0,v1\), respectively, is an \emph{elementary
expansion} of the tree pair.  It replaces the branch pair \(u\to v\)
by the two branch pairs
\[
  u0\longrightarrow v0,
  \qquad
  u1\longrightarrow v1,
\]
without changing the represented element.  Repeating this operation
shows that if \(f\) has the branch pair \(u\to v\), then it also has the
branch pair
\[
  uw\longrightarrow vw
\]
for every finite binary word \(w\).  Moreover,
\begin{equation}
\label{eq:prefix-substitution}
  f(.\!u\xi)=.\!v\xi
\end{equation}
for every finite or infinite binary word \(\xi\).  Every element of
\(F\) has a unique reduced tree pair (that is, one which is not an
elementary expansion of another tree-pair diagram), and all its
tree-pair representatives are obtained from it by elementary
expansions.

We use the endpoint-slope convention
\begin{equation}
\label{eq:pi-ab}
  \pi_{\mathrm{ab}}\colon F\longrightarrow\mathbb Z^2,
  \qquad
  \pi_{\mathrm{ab}}(f)
  =
  \bigl(\log_2 f'(0^+),\,\log_2 f'(1^-)\bigr).
\end{equation}
This is the standard abelianization map of \(F\).  If a tree pair for
\(f\) has leftmost and rightmost branch pairs
\[
  0^a\longrightarrow0^b,
  \qquad
  1^c\longrightarrow1^d,
\]
then
\[
  \pi_{\mathrm{ab}}(f)=(a-b,c-d).
\]
Its kernel is the derived subgroup \([F,F]\).  Thus \([F,F]\) consists
precisely of those elements of \(F\) which fix pointwise a neighborhood
of \(0\) and a neighborhood of \(1\).  Conjugation by the reflection
\(x\mapsto1-x\) defines an automorphism of \(F\) which interchanges the
two entries of \(\pi_{\mathrm{ab}}\).

An action on an interval is \emph{minimal} if every orbit is dense in
that interval.

For a closed interval \(J=[a,b]\subseteq[0,1]\) with dyadic endpoints
and \(a<b\), let
\begin{equation}
\label{eq:FJ}
  F[J]=\{f\in F:f(x)=x\text{ for every }x\notin(a,b)\}.
\end{equation}
There is an increasing piecewise-linear homeomorphism
\(\phi\colon[0,1]\to J\), with dyadic breakpoints and slopes powers of
\(2\), such that the map \(f\mapsto\phi^{-1}f\phi\), followed by extension
by the identity outside \(J\), identifies \(F\) with \(F[J]\).

The group \(F[J]\) is order-\(n\)-transitive on \(\D\cap(a,b)\) for
every \(n\geq1\):
given two increasing finite sequences
\[
  a<x_1<\cdots<x_n<b,
  \qquad
  a<y_1<\cdots<y_n<b
\]
of dyadic points, there is \(f\in F[J]\) such that
\(f(x_i)=y_i\) for every \(i\).  The action of \(F[J]\) on \((a,b)\)
is minimal.

\subsection{\texorpdfstring{Directed \(2\)-complexes, homotopy, and diagrams}
{Directed 2-complexes, homotopy, and diagrams}}

We use the formalism of directed \(2\)-complexes of Guba and
Sapir~\cite{GS06}.

A directed graph \(\Gamma\) consists of a vertex set \(V(\Gamma)\), an edge
set \(E(\Gamma)\), and maps
\(\iota,\tau\colon E(\Gamma)\to V(\Gamma)\) assigning to each edge its
initial and terminal vertex.  A nonempty \(1\)-path in \(\Gamma\) is a
finite sequence \(P=e_1e_2\cdots e_n\), where \(n\geq1\), of edges with
\(\tau(e_k)=\iota(e_{k+1})\) for \(1\leq k<n\); we set
\(\iota(P)=\iota(e_1)\) and \(\tau(P)=\tau(e_n)\), and call \(n\) the
length of \(P\).  For every vertex \(v\) there is also an empty
\(1\)-path \(\varepsilon_v\) of length zero, with
\(\iota(\varepsilon_v)=\tau(\varepsilon_v)=v\).  We refer to both kinds
simply as \(1\)-paths.  If
\(\tau(P)=\iota(Q)\), the concatenation \(PQ\) is defined by
juxtaposition; concatenation is associative, and
\(\varepsilon_{\iota(P)}P=P=P\varepsilon_{\tau(P)}\).

\begin{definition}[Directed \(2\)-complex; {\cite[Definition~2.1]{GS06}}]
A \emph{directed \(2\)-complex} \(L\) consists of:
\begin{enumerate}[label=\textup{(\arabic*)}]
\item a directed graph \(L^{(1)}\), the \(1\)-skeleton of \(L\), with
      vertex set \(V(L)\) and edge set \(E(L)\);
\item a set \(\cells(L)\) of positive \(2\)-cells; each
      \(f\in\cells(L)\) is equipped with two nonempty \(1\)-paths
      \(\topPath(f)\) and \(\botPath(f)\) in \(L^{(1)}\), its top and
      bottom paths, satisfying
      \[
      \iota(\topPath(f))=\iota(\botPath(f))
      \quad\text{and}\quad
      \tau(\topPath(f))=\tau(\botPath(f));
      \]
\item for every \(f\in\cells(L)\), a new formal symbol \(f^{-1}\), the
      inverse cell of \(f\), with
      \[
      \topPath(f^{-1})=\botPath(f),
      \qquad
      \botPath(f^{-1})=\topPath(f),
      \qquad
      (f^{-1})^{-1}=f.
      \]
      We write
      \(\cells(L)^{-1}=\{f^{-1}:f\in\cells(L)\}\); the sets
      \(\cells(L)\) and \(\cells(L)^{-1}\) are disjoint by construction,
      and a cell of \(L\) is an element of
      \(\cells(L)\sqcup\cells(L)^{-1}\).
\end{enumerate}
The complex \(L\) is \emph{finite} if \(V(L)\), \(E(L)\), and
\(\cells(L)\) are all finite.  A \(1\)-path in \(L\) means a \(1\)-path
in \(L^{(1)}\).
\end{definition}

When specifying a directed \(2\)-complex, we shall normally list only its
positive cells; for every listed positive cell, the corresponding formal
inverse cell is included without being listed separately.

When all vertices and endpoint maps are understood (for instance, when
the \(1\)-skeleton has been described explicitly or displayed in a
figure), we use the semigroup-presentation notation
\[
L=
\left\langle
E(L)\ \middle|\
\topPath(f)=\botPath(f),\quad f\in\cells(L)
\right\rangle .
\]
Thus the edges are listed as generators, and each positive cell contributes
the relation equating its top and bottom paths.

\begin{definition}[Homotopy of \(1\)-paths]
Let \(L\) be a directed \(2\)-complex.  An \emph{elementary transformation}
applied to a \(1\)-path \(P\) consists in writing
\[
P=A\,\topPath(f)\,B
\]
for some cell \(f\in\cells(L)\sqcup\cells(L)^{-1}\) and \(1\)-paths
\(A,B\), and replacing \(P\) by \(A\,\botPath(f)\,B\).  Two \(1\)-paths
\(P,Q\) are \emph{homotopic}, written \(P\simeq Q\), if \(Q\) is obtained
from \(P\) by a finite sequence of elementary transformations.
\end{definition}

Homotopy is an equivalence relation, since inverse cells undo elementary
transformations.  Since the top and bottom paths of a cell share their
endpoints, homotopic paths have the same initial and terminal vertices; and
if \(P\simeq Q\), then \(APB\simeq AQB\) for all \(1\)-paths \(A,B\)
such that these concatenations are defined.  Since the top and bottom paths
of cells are nonempty, an elementary transformation never applies to
\(\varepsilon_v\); thus
\[
\varepsilon_v\simeq P\quad\Longleftrightarrow\quad P=\varepsilon_v.
\]

\begin{definition}[Diagrams; {\cite[Definition~2.4]{GS06}}]
Let \(L\) be a directed \(2\)-complex.  A \emph{diagram over \(L\)} is a
finite connected plane directed graph \(\Delta\), with every edge labelled
by an edge of \(L\) and every bounded face labelled by a cell of \(L\).
The edge labelling respects incidence: the label of every \(1\)-path in
\(\Delta\) is a \(1\)-path in \(L\).  We require further that:
\begin{enumerate}[label=\textup{(\arabic*)}]
\item \(\Delta\) has exactly one source \(\iota(\Delta)\), a vertex with
      no incoming edges, and exactly one sink \(\tau(\Delta)\), a vertex
      with no outgoing edges;
\item every \(1\)-path in \(\Delta\) is simple, that is, it visits no
      vertex twice; in particular, \(\Delta\) has no directed cycles;
\item the boundary of each bounded face labelled by a cell \(f\) consists
      of two \(1\)-paths in \(\Delta\) with common endpoints, the top and
      bottom paths of the face, whose edge labels read \(\topPath(f)\) and
      \(\botPath(f)\), respectively, and of which the top path lies above
      the bottom path in the plane.
\end{enumerate}

By \cite[Section~2]{GS06}, conditions \textup{(1)}--\textup{(3)} imply that
the boundary of the unbounded face consists of two \(1\)-paths from
\(\iota(\Delta)\) to \(\tau(\Delta)\), the top path
\(\topPath(\Delta)\) and the bottom path \(\botPath(\Delta)\) of the
diagram.  If their edge labels read \(P\) and \(Q\), respectively, then
\(\Delta\) is a \((P,Q)\)-diagram; it is \emph{spherical} if \(P=Q\).
For every nonempty \(1\)-path \(P\) there is a diagram
\(\varepsilon(P)\) with no cells whose top and bottom paths coincide and
have label \(P\); it is called the \emph{trivial diagram} on \(P\).
\end{definition}

We write
\[
  \Delta\colon P\longrightarrow_L Q
\]
to indicate that \(\Delta\) is a \((P,Q)\)-diagram over \(L\), omitting
the subscript when the directed \(2\)-complex is clear.

If \(\Delta_1\) is a \((P,Q)\)-diagram and \(\Delta_2\) is a
\((Q,R)\)-diagram, their \emph{vertical concatenation}
\(\Delta_1\circ\Delta_2\), a \((P,R)\)-diagram, is obtained by identifying
\(\botPath(\Delta_1)\) with \(\topPath(\Delta_2)\).  If
\(\Delta_i\) is a \((P_i,Q_i)\)-diagram for \(i=1,2\), and the
concatenations \(P_1P_2\) and \(Q_1Q_2\) are defined, then the
\emph{horizontal sum} \(\Delta_1+\Delta_2\) is the
\((P_1P_2,Q_1Q_2)\)-diagram obtained by identifying
\(\tau(\Delta_1)\) with \(\iota(\Delta_2)\).  Reflecting a
\((P,Q)\)-diagram \(\Delta\) in a horizontal line and inverting all cell
labels yields the \emph{inverse diagram} \(\Delta^{-1}\), a
\((Q,P)\)-diagram.

\begin{lemma}[{\cite[Lemma~2.5]{GS06}}]
Two nonempty \(1\)-paths \(P,Q\) of a directed \(2\)-complex \(L\) are homotopic if
and only if there exists a \((P,Q)\)-diagram over \(L\).
\end{lemma}

Indeed, a single elementary transformation corresponds to a diagram with
one cell, and a sequence of elementary transformations to a vertical stack
of such diagrams; conversely every diagram decomposes into such a stack.

Two cells of a diagram form a \emph{dipole} if the bottom path of the first
is the top path of the second and their labels are mutually inverse cells of
\(L\).  Removing the two cells and identifying the top path of the first
with the bottom path of the second produces a new diagram.  A diagram is
\emph{reduced} if it contains no dipole, and two diagrams are
\emph{equivalent} if one is obtained from the other by finitely many
insertions and removals of dipoles.

\begin{theorem}[{\cite[Theorem~2.6]{GS06}}]
Every equivalence class of diagrams over a directed \(2\)-complex contains
exactly one reduced diagram.
\end{theorem}

\begin{definition}[Diagram groups]
Let \(L\) be a directed \(2\)-complex and \(P\) a nonempty \(1\)-path in
\(L\).  The \emph{diagram group} \(D(L,P)\) is the set of reduced
spherical \((P,P)\)-diagrams over \(L\), with multiplication
\[
\Delta_1\cdot\Delta_2
   :=\text{the reduced form of }\Delta_1\circ\Delta_2,
\]
identity \(\varepsilon(P)\), and inversion
\(\Delta\mapsto\Delta^{-1}\).  It is a group
\cite[Section~2]{GS06}.
\end{definition}

\begin{definition}[The path semigroupoid]
For vertices \(v,w\) of a directed \(2\)-complex \(L\), put
\[
\Path(L)(v,w)
   =\{\text{nonempty \(1\)-paths from \(v\) to \(w\) in \(L\)}\}/\simeq ,
\qquad
\Path(L)=\coprod_{v,w\in V(L)}\Path(L)(v,w),
\]
where \(\simeq\) is homotopy of \(1\)-paths.  For
\(u,v,w\in V(L)\), concatenation induces associative composition maps
\[
\Path(L)(u,v)\times\Path(L)(v,w)\longrightarrow\Path(L)(u,w),
\qquad ([P],[Q])\longmapsto[PQ].
\]
With these partially defined products, \(\Path(L)\) is the
\emph{path semigroupoid} of \(L\).  We say that \(\Path(L)\) is
\emph{finite} if the displayed disjoint union is finite.
\end{definition}

Recall that a group \(G\) is of type \(F_\infty\) if there is a
\(K(G,1)\) CW-complex with finitely many cells in every dimension.

\begin{theorem}[{\cite[Theorem~9.3]{GS06}}]
\label{thm:finite-path-semigroupoid}
Let \(L\) be a finite directed \(2\)-complex with finitely many homotopy
classes of nonempty \(1\)-paths.  Then every diagram group over \(L\) is
of type \(F_\infty\).
\end{theorem}

\subsection{The Dunce hat, positive diagrams, and common expansions}
\label{subsec:dunce}

The \emph{Dunce hat} is the directed \(2\)-complex
\begin{equation}\label{eq:dunce}
  \K=\langle x\mid x=x^2\rangle ,
\end{equation}
with one vertex, one edge \(x\), and one positive cell with top path
\(x\) and bottom path \(xx\).

A diagram is \emph{positive} if all of its cells are labelled by
positive cells.

\begin{definition}[Tree-like complexes and positive expansions]
\label{def:tree-like}
A directed \(2\)-complex \(L\) is \emph{tree-like} if, for every positive
cell \(f\), the path \(\topPath(f)\) is a single edge, the path
\(\botPath(f)\) has length at least two, and distinct positive cells
have distinct top edges.  If there exists a positive \((P,Q)\)-diagram
over \(L\), then \(Q\) is a \emph{positive expansion} of \(P\).  A
\emph{common positive expansion} of nonempty \(1\)-paths \(P_1,P_2\) is a
\(1\)-path which is a positive expansion of both.
\end{definition}

The Dunce hat is tree-like, as are the cores considered below.

\begin{lemma}[Common positive expansions]
\label{lem:common-expansion}
Let \(L\) be a tree-like directed \(2\)-complex, and let \(P_1\simeq P_2\)
be homotopic nonempty \(1\)-paths.  Then \(P_1\) and \(P_2\) have a
common positive expansion.  More precisely, if \(\Delta\) is a reduced
\((P_1,P_2)\)-diagram, then the \(1\)-skeleton of \(\Delta\) contains a
\(1\)-path \(\gamma\) from its source to its sink which passes through
every vertex of \(\Delta\).  Let \(R\) be the label of \(\gamma\) in
\(L\).  The part of \(\Delta\) above \(\gamma\) is a positive
\((P_1,R)\)-diagram, and the inverse of the subdiagram below \(\gamma\)
is a positive \((P_2,R)\)-diagram.
\end{lemma}

This is the horizontal-path observation of Guba and Sapir
\cite[Section~5, remarks preceding Theorem~26]{GubaSapir99}.  It is
stated there for semigroup presentations; forgetting the vertices of
\(L\) and then restoring them gives the formulation above.

For the Dunce hat, positive diagrams with top path \(x\) are naturally
identified with finite rooted binary trees: each cell subdivides its top
edge into the two edges of its bottom path, and the bottom edges of the
diagram, read from left to right, correspond to the leaves.  For a
finite rooted binary tree \(T\), let \(\Delta(T)\) denote the
corresponding positive diagram.

By \cite[Example~6.4]{GubaSapir97}, there is an isomorphism
\begin{equation}\label{eq:F-diagram-model}
  F\cong D(\K,x).
\end{equation}
Under this isomorphism, the element represented by a reduced tree pair
\((T_+,T_-)\) corresponds to the reduced spherical diagram
\[
  \Delta(T_+)\circ\Delta(T_-)^{-1}.
\]
Conversely, applying Lemma~\ref{lem:common-expansion} to a reduced
spherical diagram over \(\K\) decomposes it into two positive tree
diagrams.  Thus every reduced spherical diagram over \(\K\) arises from
a reduced tree pair.

\subsection{\texorpdfstring{The Stallings \(2\)-core of a subgroup of \(F\)}
{The Stallings 2-core of a subgroup of F}}
\label{subsec:core}

We recall the Stallings \(2\)-core of Guba and Sapir.  The construction
first appeared in print in \cite[Section~3.2]{GS17}.

\begin{definition}[The Stallings \(2\)-core]
\label{def:core}
Let \(H\leq F\), and choose a generating set \(X\) for \(H\).  For each
\(g\in X\), let
\(\Delta_g\) be the reduced spherical diagram over \(\K\) representing
\(g\).

Identify the top and bottom edges of all the diagrams \(\Delta_g\),
\(g\in X\), with one another, and denote their common image by \(p_H\).
If \(X\) is empty, take instead a single edge \(p_H\).
For every cell \(\pi\) occurring in one of the diagrams, designate its
one-edge boundary as \(\topPath(\pi)\) and its two-edge boundary as
\(\botPath(\pi)\), regardless of which of these two paths lay
geometrically above the other in the original diagram \(\Delta_g\).

Starting with the resulting bouquet, repeatedly perform the following
folding operation.  If two distinct cells have the same top path or the
same bottom path, identify the two cells and identify their other boundary
paths edge by edge.

Continue performing foldings until no folding is possible.  If infinitely many foldings are required, take the direct limit of a
sequence in which no possible folding is postponed indefinitely.  The resulting folded directed \(2\)-complex, together with
the distinguished edge \(p_H\), is the \emph{Stallings \(2\)-core} of
\(H\), denoted \(C(H)\).  The quotient cells produced by this construction are declared to be
the positive cells of \(C(H)\).  We then adjoin to \(C(H)\) the formal
inverse of each such cell.
\end{definition}

The based isomorphism type of \((C(H),p_H)\) is independent of the
chosen generating set and of the order in which the foldings are
performed.  If \(H\) is finitely generated, take \(X\) to be finite.
Then \(C(H)\) is finite.

For the remainder of this subsection, fix a subgroup \(H\leq F\), and
write
\[
\mathcal L=C(H),\qquad p=p_H.
\]

Every positive cell of \(\mathcal L\) has a single top edge and a two-edge bottom
path.  If its top edge is \(e\) and its bottom path is \(e_0e_1\), we
write the cell as
\begin{equation}
\label{eq:core-cell}
  e=e_0e_1.
\end{equation}
The edges \(e_0\) and \(e_1\) are called the \emph{left child} and
\emph{right child} of \(e\), respectively.  Foldedness implies that,
whenever they exist, they are uniquely determined by \(e\).  An edge
which is the top of no positive cell is a \emph{leaf} of the core, and
the core is \emph{full} if it has no leaves.  In particular, every core
is tree-like.

\begin{definition}[Core labellings, acceptance, and closure]
\label{def:acceptance}
Let \(H\leq F\), let \(\mathcal L=C(H)\), and let \(p=p_H\).  Let
\(\Delta\) be a diagram over the Dunce hat \(\K\).  A \emph{core labelling} of \(\Delta\)
consists of labelling every edge of \(\Delta\) by an edge of \(\mathcal L\) and
every cell of \(\Delta\) by a cell of \(\mathcal L\), in such a way that, with
these labels, \(\Delta\) is a diagram over \(\mathcal L\).  Signs are preserved:
a cell labelled by the positive cell of \(\K\) is labelled by a
positive cell of \(\mathcal L\), whereas a cell labelled by the inverse cell of
\(\K\) is labelled by an inverse cell of \(\mathcal L\).

A spherical \((x,x)\)-diagram over \(\K\) is \emph{accepted} by
\((\mathcal L,p)\) if it admits a core labelling in which both its top and bottom
edges are labelled by \(p\).

The \emph{closure} of \(H\) is
\[
\Cl(H)=
\left\{
h\in F\ \middle|\
\text{the reduced diagram of \(h\) is accepted by }(C(H),p_H)
\right\}.
\]
The subgroup \(H\) is \emph{closed} if \(H=\Cl(H)\).
\end{definition}

By foldedness, a core labelling, if it exists, is uniquely determined
by the labels on the top path.  Indeed, reading the diagram from top to
bottom, a positive cell is determined by its top edge, while a negative
cell is determined by its two-edge top path.  Likewise, the core
labelling is determined by the labels on the bottom path.

\begin{theorem}[The core and the closure]
\label{thm:core-closure}
Let \(H\leq F\), let \(\mathcal L=C(H)\), and let \(p=p_H\).
\begin{enumerate}[label=\textup{(\roman*)}]
\item The set \(\Cl(H)\) is a subgroup of \(F\) containing \(H\), and
      forgetting the core labels induces a canonical isomorphism
      \[
      D(\mathcal L,p)\cong\Cl(H).
      \]
\item If \(H\) is closed, then \(H\) is finitely generated if and only
      if \(C(H)\) is finite.
\end{enumerate}
\end{theorem}

\begin{proof}
Part \textup{(i)} is
\cite[Definition~3.7 and Lemma~3.8]{GolanMemo}; see also
\cite[Section~2.4]{GPSclosed}.  Part \textup{(ii)} is
\cite[Corollary~5.14]{GPSclosed}.
\end{proof}

Let \(H\leq F\). An element \(f\in F\) is a
\emph{piecewise-\(H\) function} if there are points
\[
0=c_0<c_1<\cdots<c_n=1
\]
and elements \(h_1,\ldots,h_n\in H\) such that \(f\) and \(h_i\)
agree on \([c_{i-1},c_i]\) for every \(i\). Since
\(f,h_1,\ldots,h_n\) belong to \(F\), the subdivision may be chosen
with dyadic subdivision points. After refining it to a dyadic
subdivision, this is equivalent to requiring that \(f\) have a
tree-pair diagram in which every branch pair \(u\to v\) is a branch
pair of some element of \(H\).

The closure \(\Cl(H)\) is precisely the subgroup of all piecewise-\(H\)
functions in \(F\) \cite[Theorem~5.6]{GolanMemo}. Consequently, \(H\)
is closed if and only if every piecewise-\(H\) function belonging to
\(F\) lies in \(H\).

\begin{definition}[Tree automata, readable words, and acceptance]
\label{def:readable}
We follow \cite[Section~2]{GolanDown}, specialized to binary
automata.  A \emph{rooted deterministic binary automaton} is a triple
\(\mathcal A=(Q,\delta,p_{\mathcal A})\), where \(Q=Q(\mathcal A)\) is a
set of \emph{states}, \(p_{\mathcal A}\in Q\) is the \emph{root}, and
\[
  \delta\colon Q\times\{0,1\}\dashrightarrow Q
\]
is a partial \emph{transition function}.  It is a \emph{tree
automaton} if
\begin{enumerate}[label=\textup{(\arabic*)}]
\item every state is reachable from \(p_{\mathcal A}\); and
\item for every state \(q\), either \(\delta(q,i)\) is defined for
      both \(i\in\{0,1\}\) or for neither.
\end{enumerate}
A state at which neither transition is defined is a \emph{leaf}, and
\(\mathcal A\) is \emph{full} if it has no leaves.  The tree automaton
\(\mathcal A\) is \emph{folded} if no two distinct states have the
same ordered pair of children.  A \emph{morphism} of tree automata is
a map of state sets which carries root to root and satisfies
\(\varphi(\delta(q,i))=\delta(\varphi(q),i)\) whenever the left-hand
side is defined, and an \emph{isomorphism} is a morphism whose inverse
is also a morphism.  A
morphism carries the state reached by a word to the state reached by
the same word, and every state is reached from the root by some word;
hence there is at most one morphism between two given tree automata.
A tree-automaton in the sense of
\cite[Definition~2.4]{GolanMaximal} is a folded tree automaton in the
sense above.

We extend \(\delta\) recursively to finite binary words by
\[
  \delta(q,\varnothing)=q,
  \qquad
  \delta(q,ui)=\delta\bigl(\delta(q,u),i\bigr)
  \qquad
  \bigl(u\in\{0,1\}^*,\ i\in\{0,1\}\bigr),
\]
whenever the right-hand side is defined; consequently
\[
  \delta(q,uw)=\delta\bigl(\delta(q,u),w\bigr)
\]
whenever either side is defined.  A finite binary word \(u\) is
\emph{readable in \(\mathcal A\)} if \(\delta(p_{\mathcal A},u)\) is defined,
and
in that case we put
\[
  u_{\mathcal A}^{+}=\delta(p_{\mathcal A},u).
\]

A tree-pair diagram with branch pairs
\[
  u_1\longrightarrow v_1,\ldots,u_m\longrightarrow v_m
\]
is \emph{accepted by \(\mathcal A\)} if all the words \(u_i,v_i\) are
readable in \(\mathcal A\) and
\[
  (u_i)_{\mathcal A}^{+}=(v_i)_{\mathcal A}^{+}
  \qquad(1\leq i\leq m).
\]
Let \(\mathcal D(\mathcal A)\) be the set of elements of \(F\)
admitting at least one accepted tree-pair diagram.  It is a subgroup
of \(F\) \cite[Lemmas~2.5 and~2.11]{GolanDown}, called the
\emph{subgroup accepted by \(\mathcal A\)}.  If \(\mathcal A\) is
folded, then an element of \(F\) belongs to \(\mathcal D(\mathcal A)\)
if and only if its reduced tree-pair diagram is accepted by
\(\mathcal A\) \cite[Lemma~2.5]{GolanDown}.
\end{definition}

We now realize the core \(\mathcal L=C(H)\) as a tree automaton.
Consider the state set \(E(\mathcal L)\), with root state \(p\), and
the partial transition function
\[
  \delta\colon E(\mathcal L)\times\{0,1\}
  \dashrightarrow E(\mathcal L)
\]
defined as follows.  If
\[
  e=e_0e_1
\]
is a positive cell of \(\mathcal L\), then
\[
  \delta(e,0)=e_0,
  \qquad
  \delta(e,1)=e_1.
\]
If \(e\) is a leaf of \(\mathcal L\), neither transition from \(e\) is
defined.  We extend \(\delta\) to finite binary words by the recursion
above.

This transition function is well defined.  Indeed, two positive cells
with top edge \(e\) would have the same top path and hence would
coincide by foldedness.  Since an edge is either the top edge of a
positive cell or a leaf of \(\mathcal L\), at every state either both
transitions are defined or neither transition is defined; in
particular the leaves of the automaton are exactly the leaves of the
core, and the automaton is full if and only if the core is full.  If
two states \(e\) and \(e'\) have the same ordered pair of children,
then the positive cells with top edges \(e\) and \(e'\) have the same
bottom path; foldedness therefore forces these cells to coincide, and
hence \(e=e'\).  Finally, every edge of \(C(H)\) is reached from \(p\)
by some finite binary word \cite[Section~4]{GolanMemo}, so every state
is reachable from the root state.

Thus \(E(\mathcal L)\), with root state \(p\) and transition function
\(\delta\), is a folded tree automaton.  We call it the \emph{child
automaton} of \(\mathcal L\) and denote it by
\(\mathcal A(\mathcal L)\).  For the child automaton we suppress the
subscript \(\mathcal A(\mathcal L)\) and write
\[
  u^+=\delta(p,u).
\]
Thus \(\varnothing\) is readable and \(\varnothing^+=p\).  If \(u\) is
readable and \(u^+=e\) is the top edge of the positive cell
\[
  e=e_0e_1,
\]
then \(u0\) and \(u1\) are readable, and
\[
  (u0)^+=e_0,
  \qquad
  (u1)^+=e_1.
\]

If the core is full, every binary word is readable.  Reading a word
\(u\) in \(\mathcal A(\mathcal L)\) ends at the state \(u^+\), which is
an edge of \(\mathcal L\).  The following lemma describes exactly when
\(\Delta(T)\) admits a core labelling and identifies the \(1\)-path in
\(\mathcal L\) read along its bottom path.

\begin{lemma}[Trees over the core]
\label{lem:trees-over-core}
Let \(T\) be a finite rooted binary tree whose leaves, from left to
right, have addresses
\[
  u_1,\ldots,u_m.
\]
The positive diagram \(\Delta(T)\) admits a core labelling with top edge
\(p\) if and only if every \(u_i\) is readable.  In that case the
labelling is unique, and its bottom path is
\[
  u_1^+\cdots u_m^+.
\]
In particular, this word in the edges of \(\mathcal L\) is a
\(1\)-path.
\end{lemma}

\begin{proof}
If such a labelling exists, reading the edge labels along each
root-to-leaf branch shows that every \(u_i\) is readable.

Conversely, suppose that all the \(u_i\) are readable.  We construct
the labelling from the root downward.  Under the identification of
Subsection~\ref{subsec:dunce}, the edges of \(\Delta(T)\) correspond
to the vertices of \(T\), the top edge corresponding to the root.
Label the top edge by \(p=\varnothing^+\), and suppose that the edge
corresponding to the vertex with address \(w\) has already been
labelled by \(w^+\).  If that vertex is not a leaf of \(T\), then
\(w0\) and \(w1\) are readable, and the positive cell
\[
  w^+=(w0)^+(w1)^+
\]
labels its caret and determines the labels of its two child edges.
Continuing downward labels every edge and cell of \(\Delta(T)\).  The
labels of the bottom edges, read from left to right, are
\(u_1^+,\ldots,u_m^+\).  Uniqueness follows from foldedness.
\end{proof}

Whenever the equivalent conditions of
Lemma~\ref{lem:trees-over-core} hold, we denote the resulting positive
diagram over \(\mathcal L\) by
\[
  \Psi_T\colon p\longrightarrow u_1^+\cdots u_m^+.
\]

\begin{lemma}[Acceptance in terms of branch pairs]
\label{lem:acceptance-branches}
Let
\[
  \Delta=\Delta(T_+)\circ\Delta(T_-)^{-1}
\]
be a tree-pair diagram with branch pairs
\[
  u_1\longrightarrow v_1,\ldots,u_m\longrightarrow v_m,
\]
listed from left to right.  Then \(\Delta\) is accepted by
\((\mathcal L,p)\) if and only if all the words \(u_i,v_i\) are
readable and
\[
  u_i^+=v_i^+
  \qquad(1\leq i\leq m).
\]
\end{lemma}

\begin{proof}
Suppose first that \(\Delta\) is accepted.  Restricting its core
labelling to the two tree diagrams and applying
Lemma~\ref{lem:trees-over-core} shows that all the words \(u_i,v_i\)
are readable.  The labels on the two copies of the middle path agree,
so \(u_i^+=v_i^+\) for every \(i\).

Conversely, suppose that all the words \(u_i,v_i\) are readable and
that \(u_i^+=v_i^+\) for every \(i\).  By
Lemma~\ref{lem:trees-over-core}, the two tree diagrams admit unique core
labellings with top edge \(p\), and their bottom paths are
\[
  u_1^+\cdots u_m^+
  \qquad\text{and}\qquad
  v_1^+\cdots v_m^+,
\]
respectively.  These labellings glue across the middle path of
\(\Delta\) precisely when the corresponding leaf-edge labels agree.
\end{proof}

Thus a tree-pair diagram is accepted by the child automaton
\(\mathcal A(\mathcal L)\) in the sense of
Definition~\ref{def:readable} if and only if it is accepted by
\((\mathcal L,p)\) in the sense of Definition~\ref{def:acceptance}.  In
particular, the subgroup accepted by \(\mathcal A(C(H))\) is
\(\Cl(H)\).

\begin{lemma}[Acceptance and reduction]
\label{lem:acceptance-reduction}
Let \(H\leq F\), let \(\mathcal L=C(H)\), and let \(p=p_H\).  If a
diagram \(\Delta\) over \(\K\) is accepted by \((\mathcal L,p)\), then
every reduction of \(\Delta\) is accepted.  In particular, the reduced
form of \(\Delta\) is accepted.

If, in addition, \(\mathcal L\) is full and \(\Delta\) is a tree-pair
diagram, then \(\Delta\) is accepted if and only if its reduced
tree-pair diagram is accepted.
\end{lemma}

\begin{proof}
Removing a dipole from a diagram over \(\mathcal L\) leaves a diagram
over \(\mathcal L\), so it preserves the core labelling.  This proves
the first assertion.

For the second assertion, one implication follows from the first.
Conversely, every tree-pair diagram is obtained from its reduced
tree-pair diagram by finitely many common expansions.  Each common
expansion inserts a positive--inverse dipole along a single edge.
Since \(\mathcal L\) is full, that edge is the top edge of a positive
cell of \(\mathcal L\), so the expanded tree-pair diagram remains
accepted.
\end{proof}

In particular, if a tree-pair diagram representing \(f\in F\) is
accepted by \((\mathcal L,p)\), then its reduced tree-pair diagram is
accepted by Lemma~\ref{lem:acceptance-reduction}.  Hence
\(f\in\Cl(H)\) by Theorem~\ref{thm:core-closure}\textup{(i)}.

\begin{lemma}[Branch-pair criterion]
\label{lem:branch-criterion}
Let \(H\leq F\) be closed, and let \(u,v\) be words readable in
\(C(H)\).  Then
\[
  u^+=v^+
  \quad\Longleftrightarrow\quad
  \text{some \(h\in H\) has the branch pair \(u\to v\).}
\]
\end{lemma}

\begin{proof}
This is \cite[Lemma~2.19]{GolanMaximal}; the element obtained there
belongs to \(\Cl(H)\), which equals \(H\) by closedness.
\end{proof}

We shall also need to fold automata.  By
\cite[Definition~2.9]{GolanDown}, every tree automaton \(\mathcal A\)
has a \emph{folded quotient} \(\overline{\mathcal A}\), the quotient of
\(\mathcal A\) by the smallest congruence for which the quotient
automaton is folded; concretely, it is obtained by repeatedly
identifying two states which have the same ordered pair of children.
It is finite if \(\mathcal A\) is finite, and full if \(\mathcal A\)
is full.  Moreover
\[
  \mathcal D(\mathcal A)=\mathcal D(\overline{\mathcal A})
\]
by \cite[Lemma~2.11]{GolanDown}.

We call a tree automaton \(\mathcal A\) a \emph{core automaton} if
\[
  \mathcal A\cong\mathcal A(C(G))
\]
for some subgroup \(G\leq F\).  The following proposition recognizes
the core automata among the finite full folded tree automata, and
reconstructs the directed \(2\)-core from the automaton.

\begin{proposition}[Finite full core automata and directed
\texorpdfstring{\(2\)}{2}-cores]
\label{prop:finite-full-core-automaton}
Let \(\mathcal A\) be a finite full folded tree automaton, and let
\(G\leq F\) be the subgroup accepted by \(\mathcal A\).  Then
\(\mathcal A\) is a core automaton if and only if, for every pair of
finite binary words \(u,v\) satisfying
\[
  u_{\mathcal A}^{+}=v_{\mathcal A}^{+},
\]
some element of \(G\) has the branch pair \(u\to v\).  In that case
\[
  \mathcal A\cong\mathcal A(C(G)).
\]

Moreover, the based directed \(2\)-complex \(C(G)\) is recovered from
\(\mathcal A\) as follows.  For every state \(q\in Q(\mathcal A)\),
introduce two formal symbols \(\iota_q\) and \(\tau_q\).  Put
\[
  q_0=\delta_{\mathcal A}(q,0),
  \qquad
  q_1=\delta_{\mathcal A}(q,1),
\]
and let \(\sim\) be the equivalence relation generated by
\[
  \iota_q\sim\iota_{q_0},
  \qquad
  \tau_{q_0}\sim\iota_{q_1},
  \qquad
  \tau_{q_1}\sim\tau_q
\]
for every state \(q\).  The vertices of \(C(G)\) are the resulting
equivalence classes.  The state \(q\) gives an edge from
\([\iota_q]\) to \([\tau_q]\), and the two transitions from \(q\)
give the positive cell
\[
  q=q_0q_1.
\]
The root state \(p_{\mathcal A}\) gives the distinguished edge.
\end{proposition}

\begin{proof}
We first note that \(\mathcal A\) is reduced in the sense of
\cite[Definition~4.1]{GolanMaximal}.  Indeed, suppose that
\(\mathcal A\) were a nontrivial extension of a tree automaton
\(\mathcal B\).  Then a nonempty binary tree \(T\) would have been
attached at some leaf of \(\mathcal B\).  The leaves of \(T\) are
leaves of \(\mathcal A\); since \(\mathcal A\) is full, \(T\) has no
leaves and is therefore infinite, contrary to the finiteness of
\(\mathcal A\).

The asserted equivalence therefore follows from
\cite[Lemma~6.1]{GolanMaximal}.  Moreover, in the converse part of the
proof of that lemma, the subgroup is taken to be precisely the subgroup
\(G\) accepted by \(\mathcal A\), and a bijective morphism from the
core automaton of \(G\) to \(\mathcal A\) is constructed.  Consequently,
\[
  \mathcal A\cong\mathcal A(C(G)).
\]

Use this isomorphism to identify the states of \(\mathcal A\) with the
edges of \(C(G)\), and let \(\widetilde{\mathcal L}\) be the based
directed \(2\)-complex constructed from \(\mathcal A\) in the statement.
For every positive cell
\[
  q=q_0q_1
\]
of \(C(G)\), its boundary satisfies
\[
  \iota(q)=\iota(q_0),
  \qquad
  \tau(q_0)=\iota(q_1),
  \qquad
  \tau(q_1)=\tau(q).
\]
Thus the assignment \(\iota_q\mapsto\iota(q)\),
\(\tau_q\mapsto\tau(q)\) on formal endpoints is constant on the
equivalence classes of \(\sim\).  Together with the map sending the
edge and positive cell associated with \(q\) to the corresponding edge
and positive cell of \(C(G)\), it therefore gives a well-defined
morphism
\[
  \Phi\colon\widetilde{\mathcal L}\longrightarrow C(G),
  \qquad
  \Phi([\iota_q])=\iota(q),
  \qquad
  \Phi([\tau_q])=\tau(q).
\]
It is bijective on edges and on positive cells.  It is also surjective
on vertices, since by the construction in Definition~\ref{def:core}
every vertex of \(C(G)\) is an endpoint of an edge.

It remains to prove injectivity on vertices.  Recall the construction
of \(C(G)\) from Definition~\ref{def:core}.  By
Subsection~\ref{subsec:dunce}, every reduced spherical diagram used in
forming the initial bouquet is a tree-pair diagram.  For a positive
tree diagram, the identifications among the formal endpoints of its
edges are generated precisely by the three incidence relations
\[
  \iota(q)=\iota(q_0),
  \qquad
  \tau(q_0)=\iota(q_1),
  \qquad
  \tau(q_1)=\tau(q)
\]
coming from its positive cells.  This follows by induction on the
number of carets, since attaching a caret to an edge \(q\) adds the two
edges \(q_0,q_1\) and exactly one new vertex, namely the common
endpoint \(\tau(q_0)=\iota(q_1)\), and contributes exactly the three
displayed relations.  The same statement holds for the inverse of a
positive tree diagram.

Gluing the two tree diagrams along their common leaf path identifies
the corresponding middle edges edge by edge.  Forming the initial
bouquet likewise identifies whole directed edges, and every subsequent
folding identifies two cells together with their remaining boundary
paths edge by edge.  Each of these operations is a quotient by
identifications of edges, and identifying two edges forces exactly the
identification of their initial vertices and of their terminal
vertices.  Consequently, these operations introduce no vertex
identifications other than identifications of corresponding endpoints
of edges that are themselves identified.

After passage to the final core, all edge identifications have already
been recorded by using a single pair of symbols
\(\iota_q,\tau_q\) for each edge \(q\) of \(C(G)\), and every positive
cell of \(C(G)\) is the image of a positive cell of one of the
diagrams, no cells being created along the way.  The remaining vertex
identifications are therefore exactly the three incidence relations
contributed by the positive cells of \(C(G)\), namely the relations
defining \(\sim\).  Hence \(\Phi\) is injective on vertices.

It follows that
\[
  \widetilde{\mathcal L}\cong C(G)
\]
as based directed \(2\)-complexes.
\end{proof}

\subsection{Inner vertices, dyadic points, and boundary paths}
\label{subsec:inner-paths}

Throughout this subsection, let \(H\leq F\) be closed, and write
\[
\mathcal L=C(H),\qquad p=p_H,\qquad
\iota_{\mathcal L}=\iota(p),\qquad \tau_{\mathcal L}=\tau(p).
\]
The vertices \(\iota_{\mathcal L}\) and \(\tau_{\mathcal L}\) are called the initial and
terminal vertices of \(\mathcal L\); all other vertices are called
\emph{inner}.  We write \(V_{\mathrm{in}}(\mathcal L)\) for the set of inner
vertices of \(\mathcal L\).  An edge is inner if both its endpoints are inner,
and a \(1\)-path is inner if all its edges are inner.

\begin{lemma}[Boundary incidence; {\cite[Corollary~4.10]{GolanMemo}}]
\label{lem:boundary-incidence}
The vertex \(\iota_{\mathcal L}\) has no incoming edges, and \(\tau_{\mathcal L}\) has no
outgoing edges.  Every inner vertex has at least one incoming and one
outgoing edge.  If \(u\) is readable, then:
\begin{enumerate}[label=\textup{(\arabic*)}]
\item \(u\) contains the digit \(0\) if and only if \(u^+\) is not
      incident to \(\tau_{\mathcal L}\);
\item \(u\) contains the digit \(1\) if and only if \(u^+\) is not
      incident to \(\iota_{\mathcal L}\).
\end{enumerate}
Consequently, \(u^+\) is an inner edge if and only if \(u\) contains
both digits.
\end{lemma}

Thus every non-inner edge either begins at \(\iota_{\mathcal L}\) or ends at
\(\tau_{\mathcal L}\).  In particular, every \(1\)-path with inner endpoints is
inner.  Since homotopy preserves endpoints, every path occurring in a
homotopy between paths with inner endpoints is also inner.

\begin{lemma}[Accessibility]
\label{lem:accessibility}
Every vertex of \(\mathcal L\) lies on a \(1\)-path from \(\iota_{\mathcal L}\) to
\(\tau_{\mathcal L}\).
\end{lemma}

\begin{proof}
Every vertex in the initial bouquet lies on a directed source-to-sink
path in one of the generator diagrams, and the quotient defining
\(C(H)\) preserves such paths.  If the generating set is empty, the
assertion is immediate because \(\mathcal L\) consists of the single edge \(p\).
\end{proof}

\begin{lemma}[Leaf intervals]
\label{lem:leaf-intervals}
Suppose that \(u\) is readable and that \(u^+\) is a leaf of \(\mathcal L\).
Then no two distinct points of \(\operatorname{int}[u]\) belong to the
same \(H\)-orbit.
\end{lemma}

\begin{proof}
Let \(a=.\!uw_1\) and \(b=.\!uw_2\) be dyadic points in
\(\operatorname{int}[u]\), with their finite binary expansions chosen to
end in \(1\).  Since \(u^+\) is a leaf, \(u\) is the longest readable
prefix of both \(uw_1\) and \(uw_2\).  Hence
\cite[Lemma~6.2]{GolanMemo} implies that \(a\) and \(b\) can belong to
the same \(H\)-orbit only if \(w_1=w_2\), and therefore only if \(a=b\).

This proves the assertion for dyadic points.  In general, if
\(h(x)=y\neq x\) for
\(x,y\in\operatorname{int}[u]\), continuity allows us to choose a
dyadic point \(a\) sufficiently close to \(x\) that
\(a,h(a)\in\operatorname{int}[u]\) and \(h(a)\neq a\).  Since \(F\)
preserves \(\D\), this contradicts the dyadic case.
\end{proof}

\begin{corollary}
\label{cor:minimal-full}
If \(H\) acts minimally on \((0,1)\), then \(\mathcal L\) is full.
\end{corollary}

Suppose that \(\mathcal L\) is full, so every finite binary word is
readable. For \(a\in\D\), write \(a=.\!u\), where \(u\) ends in \(1\), and define
\[
v(a)=\iota(u^+).
\]
This is an inner vertex, called the \emph{type} of \(a\).  We also set
\[
v(0)=\iota_{\mathcal L},\qquad v(1)=\tau_{\mathcal L}.
\]
Since the left child of an edge has the same initial vertex as the edge
itself,
\begin{equation}
\label{eq:zero-ray-initial}
\iota\bigl((u0^j)^+\bigr)=\iota(u^+)
\qquad (j\geq0).
\end{equation}

Let \(a=.\!u\) and \(b=.\!v\), where \(u\) and \(v\) end in \(1\).  By
the tree-pair description, \(a\) and \(b\) belong to the same
\(H\)-orbit if and only if, for some \(m,n\geq0\), an element of \(H\)
has the branch pair
\[
u0^m\longrightarrow v0^n.
\]
By Lemma~\ref{lem:branch-criterion}, this is equivalent to
\[
(u0^m)^+=(v0^n)^+.
\]
By \cite[Remark~6.4 and Proposition~6.6]{GolanMemo}, such \(m,n\)
exist if and only if
\[
\iota(u^+)=\iota(v^+),
\]
or equivalently \(v(a)=v(b)\).  Finally, let \(z\) be an inner vertex.
By the construction of the core, \(z\) is the initial vertex of \(u^+\)
for some readable word \(u\).  Since \(z\) is inner,
Lemma~\ref{lem:boundary-incidence} implies that \(u\) contains a \(1\).
Write \(u=w0^j\), where \(w\) ends in \(1\).  Then
\eqref{eq:zero-ray-initial} yields \(z=v(.\!w)\).

\begin{proposition}[The orbit--core correspondence;
{\cite[Section~6]{GolanMemo}}]
\label{prop:orbit-core}
Assume that \(\mathcal L\) is full. For \(a,b\in\D\),
\[
a\text{ and }b\text{ belong to the same }H\text{-orbit}
\quad\Longleftrightarrow\quad
v(a)=v(b).
\]
Consequently, \(a\mapsto v(a)\) induces a bijection
\[
\D/H\longrightarrow V_{\mathrm{in}}(\mathcal L).
\]
\end{proposition}

\begin{proposition}[Characterization of finite full cores]
\label{prop:finite-full-core}
Suppose that \(H\leq F\) is closed and that
\(\mathcal L=C(H)\) is finite.  The following are equivalent:
\begin{enumerate}[label=\textup{(\roman*)}]
\item \(\mathcal L\) is full;
\item \(H\) has finitely many orbits on \(\D\);
\item there is no nonempty open interval \(U\subseteq(0,1)\) such that
      every two distinct dyadic points of \(U\) belong to distinct
      \(H\)-orbits.
\end{enumerate}
\end{proposition}

\begin{proof}
If \(\mathcal L\) is full, Proposition~\ref{prop:orbit-core} identifies
the \(H\)-orbits on \(\D\) with the inner vertices of \(\mathcal L\).
Thus \textup{(i)} implies \textup{(ii)}.  Since every nonempty open
interval contains infinitely many dyadic points, \textup{(ii)} implies
\textup{(iii)}.

Finally, if \(\mathcal L\) is not full, then some readable word \(u\)
has \(u^+\) a leaf.  Lemma~\ref{lem:leaf-intervals}, applied to
\(U=\operatorname{int}[u]\), shows that \textup{(iii)} fails.  Hence
\textup{(iii)} implies \textup{(i)}.
\end{proof}

\begin{lemma}[Vertices of a dyadic subdivision]
\label{lem:subdivision-vertices}
Assume that \(\mathcal L\) is full. Let \(T\) be a finite rooted
binary tree with leaf addresses
\[
u_1,\ldots,u_m
\]
from left to right, and let
\[
0=a_0<a_1<\cdots<a_m=1
\]
be its subdivision points, so that
\[
[u_k]=[a_{k-1},a_k]\qquad(1\leq k\leq m).
\]
By Lemma~\ref{lem:trees-over-core},
\[
u_1^+\cdots u_m^+
\]
is a \(1\)-path in \(\mathcal L\).  If \(z_0,z_1,\ldots,z_m\) are its successive
vertices, then
\[
z_k=v(a_k)\qquad(0\leq k\leq m).
\]
\end{lemma}

\begin{proof}
The assertion is immediate for \(k=0,m\).  Let \(1\leq k<m\), and
write \(a_k=.\!w\), where \(w\) ends in \(1\).  Since
\([u_{k+1}]\) begins at \(a_k\), we have \(u_{k+1}=w0^j\) for some
\(j\geq0\).  Hence
\[
z_k=\iota(u_{k+1}^+)
   =\iota\bigl((w0^j)^+\bigr)
   =\iota(w^+)
   =v(a_k)
\]
by \eqref{eq:zero-ray-initial}.
\end{proof}

We finish with facts valid for every core \(\mathcal L=C(H)\); fullness is no
longer assumed.

The following is \cite[Propositions~10.1 and~10.2]{GolanMemo}.

\begin{lemma}[Paths meeting the boundary]
\label{lem:boundary-paths}
\begin{enumerate}[label=\textup{(\arabic*)}]
\item If \(P,Q\) are nonempty \(1\)-paths beginning at \(\iota_{\mathcal L}\),
      then
      \[
      P\simeq Q\quad\Longleftrightarrow\quad
      \tau(P)=\tau(Q).
      \]
\item If \(P,Q\) are nonempty \(1\)-paths ending at \(\tau_{\mathcal L}\), then
      \[
      P\simeq Q\quad\Longleftrightarrow\quad
      \iota(P)=\iota(Q).
      \]
\end{enumerate}
\end{lemma}

By Lemma~\ref{lem:accessibility}, a nonempty path beginning at
\(\iota_{\mathcal L}\) can be extended on the right to a path from \(\iota_{\mathcal L}\)
to \(\tau_{\mathcal L}\), which is homotopic to \(p\) by
Lemma~\ref{lem:boundary-paths}.  The dual statement holds for paths
ending at \(\tau_{\mathcal L}\).  More generally, every \(1\)-path \(W\) can be
completed to a path
\[
AWB
\]
from \(\iota_{\mathcal L}\) to \(\tau_{\mathcal L}\), where \(A\) or \(B\) may be empty.

\begin{definition}[Thinness]
\label{def:thinness}
The path semigroupoid \(\Path(\mathcal L)\) is \emph{thin} if
\[
\bigl|\Path(\mathcal L)(v,w)\bigr|\leq1
\]
for every pair of vertices \(v,w\) of \(\mathcal L\).
\end{definition}

\begin{lemma}[Reduction to inner vertices]
\label{lem:inner-reduction}
\begin{enumerate}[label=\textup{(\arabic*)}]
\item The path semigroupoid \(\Path(\mathcal L)\) is thin if and only if any two
      nonempty \(1\)-paths with the same inner endpoints are homotopic.
\item Suppose that \(V(\mathcal L)\) is finite (in particular, this holds if
      \(\mathcal L\) is finite) and that
      \[
      \bigl|\Path(\mathcal L)(v,w)\bigr|<\infty
      \]
      for every pair of inner vertices \(v,w\).  Then \(\Path(\mathcal L)\) is
      finite.
\end{enumerate}
\end{lemma}

\begin{proof}
Every nonempty \(1\)-path either begins at \(\iota_{\mathcal L}\), ends at
\(\tau_{\mathcal L}\), or has two inner endpoints.  In the first two cases, its
homotopy class is determined by its other endpoint by
Lemma~\ref{lem:boundary-paths}.  Both assertions now follow.
\end{proof}

\subsection{Orbitals and minimal sets}
\label{subsec:minimal-sets}

Let \(J\) be a nonempty open interval and let
\(G\leq\operatorname{Homeo}_+(J)\).  Put
\[
\operatorname{Fix}_J(G)
 =
\{x\in J:g(x)=x\text{ for every }g\in G\},
\qquad
\operatorname{Supp}_J(G)=J\setminus\operatorname{Fix}_J(G).
\]
An \emph{orbital} of \(G\) in \(J\) is a connected component of
\(\operatorname{Supp}_J(G)\).  A nonempty subset \(M\subseteq J\) is a
\emph{minimal set} for the action of \(G\) if it is closed in \(J\),
\(G\)-invariant, and contains no proper nonempty closed
\(G\)-invariant subset.  A perfect nowhere-dense minimal set is an
\emph{exceptional minimal set}; the connected components of its
complement in \(J\) are its \emph{gaps}.

For \(g\in\operatorname{Homeo}_+(J)\), write
\[
  \operatorname{Supp}_J(g)=\operatorname{Supp}_J(\langle g\rangle),
\]
and call the orbitals of \(\langle g\rangle\) the orbitals of \(g\).

If \((a,b)\) is an orbital of \(G\) and \(x\in(a,b)\), then
\[
\inf Gx=a,
\qquad
\sup Gx=b.
\]
Indeed, an infimum or supremum lying in \((a,b)\) would be fixed by
every element of \(G\).

\begin{lemma}[Finite fixed set]
\label{lem:finite-fixed-set}
Let \(H\leq F\) be closed, and suppose that its core is finite and full.
Then
\[
\operatorname{Fix}_{(0,1)}(H)
\]
is finite and consists of rational points.
\end{lemma}

\begin{proof}
By Theorem~\ref{thm:core-closure}\textup{(ii)}, \(H\) is finitely
generated.  It is standard that a finitely generated subgroup of \(F\)
has only finitely many orbitals and that every endpoint of an orbital
is rational; see, for example,
\cite[Sections~2.2 and~4]{BleakBroughHermiller21}.  Hence
\(\operatorname{Fix}_{(0,1)}(H)\) is a finite union of intervals and
rational points.  Proposition~\ref{prop:finite-full-core} rules out a
nonempty fixed interval. Therefore
\(\operatorname{Fix}_{(0,1)}(H)\) is finite and consists of rational
points.
\end{proof}

\begin{lemma}[The fixed-point-free minimal-set alternative]
\label{lem:minimal-set-alternative}
Let \(J\) be a nonempty open interval, and let
\(G\leq\operatorname{Homeo}_+(J)\) be finitely generated.  Suppose that
\(\operatorname{Fix}_J(G)=\varnothing\).  Then \(G\) has a minimal set
in \(J\), and every minimal set \(M\subseteq J\) is of exactly one of
the following types:
\begin{enumerate}[label=\textup{(\arabic*)}]
\item \(M=J\), that is, the action of \(G\) on \(J\) is minimal;
\item \(M\) is a closed discrete orbit;
\item \(M\) is an exceptional minimal set.
\end{enumerate}
In case \textup{(2)}, the action on the ordered orbit \(M\) induces a
nonzero homomorphism \(G\to\mathbb Z\).  In case \textup{(3)}, \(M\) is
the unique minimal set of \(G\), the infimum and supremum of \(M\) are
the two endpoints of \(J\), and \(G\) permutes the gaps of \(M\).
\end{lemma}

This is the standard minimal-set alternative for finitely generated
groups of orientation-preserving homeomorphisms of an interval; see
\cite[Lemma~3.5.18 and the discussion following it]{DNR}.

\section{Intervals and paths in a full core}
\label{sec:interval-paths}

Throughout this section, \(H\leq F\) is closed, \(\mathcal L=C(H)\) is full, and
\(p=p_H\) is the distinguished edge of \(\mathcal L\).  These assumptions
remain in force throughout the section.

\subsection{Subdivision paths}
\label{subsec:subdivision-paths}

\begin{definition}
\label{def:subdivision-word}
Let \(I=[a,b]\subset(0,1)\), where \(a,b\in\D\) and \(a<b\), and let
\[
  \mathscr S=([u_1],\ldots,[u_m])
\]
be a dyadic subdivision of \(I\).  Since \(\mathcal L\) is full, every \(u_i\)
is readable.  The edge word
\[
  W_{\mathscr S}=u_1^+\cdots u_m^+
\]
is the \emph{subdivision word} associated with \(\mathscr S\).
\end{definition}

\begin{lemma}
\label{lem:subdivision-words}
Let \(I=[a,b]\subset(0,1)\), where \(a,b\in\D\) and \(a<b\), and let
\(\mathscr S\) be a dyadic subdivision of \(I\).  Then
\(W_{\mathscr S}\) is a \(1\)-path from \(v(a)\) to \(v(b)\).  If
\(\mathscr T\) is another dyadic subdivision of \(I\), then
\[
  W_{\mathscr S}\simeq W_{\mathscr T}.
\]
Moreover, if
\[
  \Delta\colon W_{\mathscr S}\longrightarrow_{\mathcal L} Q
\]
is positive, then
\[
  Q=W_{\mathscr S'}
\]
for some dyadic refinement \(\mathscr S'\) of \(\mathscr S\).
\end{lemma}

\begin{proof}
Extend \(\mathscr S\) to a dyadic subdivision of \([0,1]\), and let
\(T\) be the corresponding binary tree.  By
Lemma~\ref{lem:trees-over-core}, \(W_{\mathscr S}\) is a consecutive
subpath of the bottom \(1\)-path of \(\Psi_T\).  Hence it is a
\(1\)-path in \(\mathcal L\), and
Lemma~\ref{lem:subdivision-vertices} shows that its initial and terminal
vertices are \(v(a)\) and \(v(b)\).

Any two dyadic subdivisions of \(I\) have a common refinement.  An
elementary refinement replaces an interval \([u]\) by \([u0]\) and
\([u1]\), and the corresponding positive cell replaces \(u^+\) by
\[
  (u0)^+(u1)^+.
\]
This proves that \(W_{\mathscr S}\simeq W_{\mathscr T}\).

Finally, every cell of a positive diagram replaces one edge in exactly
this way.  Thus the bottom \(1\)-path of \(\Delta\) is
\(W_{\mathscr S'}\) for a dyadic refinement \(\mathscr S'\) of
\(\mathscr S\).
\end{proof}

\begin{definition}
\label{def:interval-class}
For dyadic \(0<a<b<1\), the \emph{interval class} associated with
\([a,b]\) is
\[
  \omega[a,b]=[W_{\mathscr S}]
  \in\Path(\mathcal L)(v(a),v(b)),
\]
where \(\mathscr S\) is any dyadic subdivision of \([a,b]\).  We also
write \(\omega(I)\) when \(I=[a,b]\).
\end{definition}

\begin{lemma}[Cutting]
\label{lem:cutting}
Let \(a,b,c\in\D\), with \(a<b<c\).  Then
\[
  \omega[a,c]=\omega[a,b]\,\omega[b,c].
\]
\end{lemma}

\begin{proof}
Let \(\mathscr S_1\) and \(\mathscr S_2\) be dyadic subdivisions of
\([a,b]\) and \([b,c]\), respectively.  Their union is a dyadic
subdivision \(\mathscr S\) of \([a,c]\), and
\[
  W_{\mathscr S}=W_{\mathscr S_1}W_{\mathscr S_2}.
\]
Passing to homotopy classes proves the result.
\end{proof}

\subsection{Transport, realization, and placement}
\label{subsec:transport-realization-placement}

\begin{proposition}[Transport]
\label{prop:transport}
Let \(h\in H\), and let \(a,b\in\D\), with \(a<b\).  Then
\[
  \omega[a,b]=\omega[h(a),h(b)].
\]
\end{proposition}

\begin{proof}
Choose an accepted tree pair for \(h\), expanded so that \(a\) and
\(b\) are subdivision points of its domain tree.  By
Lemma~\ref{lem:acceptance-branches}, \(u^+=v^+\) for every branch pair
\(u\to v\).  Hence the subdivision words associated with
\([a,b]\) and \([h(a),h(b)]\) coincide.
\end{proof}

\begin{proposition}[Exact transport]
\label{prop:exact-transport}
Let \(I,J\subset(0,1)\) be closed intervals with dyadic endpoints.  Then
\[
  \omega(I)=\omega(J)
\]
if and only if some \(h\in H\) maps \(I\) onto \(J\).
\end{proposition}

\begin{proof}
One implication follows from Proposition~\ref{prop:transport}.
Conversely, suppose that \(\omega(I)=\omega(J)\).  Choose dyadic
subdivisions \(\mathscr S_I\) and \(\mathscr S_J\) of \(I\) and \(J\),
extend them to dyadic subdivisions of \([0,1]\), and let \(T_I,T_J\)
be the corresponding binary trees.  Write
\[
  \Psi_{T_I}\colon p\longrightarrow A_IW_IB_I,
  \qquad
  \Psi_{T_J}\colon p\longrightarrow A_JW_JB_J,
\]
where
\[
  W_I=W_{\mathscr S_I},
  \qquad
  W_J=W_{\mathscr S_J}.
\]

The hypothesis gives \(W_I\simeq W_J\), so these \(1\)-paths have the
same initial and terminal vertices.  Lemma~\ref{lem:boundary-paths}
therefore gives
\[
  A_I\simeq A_J,
  \qquad
  B_I\simeq B_J.
\]
By Lemma~\ref{lem:common-expansion}, choose common positive expansions
\[
  A_I,A_J\longrightarrow A,\qquad
  W_I,W_J\longrightarrow W,\qquad
  B_I,B_J\longrightarrow B.
\]
Taking their horizontal sums and attaching them below the corresponding
diagrams \(\Psi_{T_I}\) and \(\Psi_{T_J}\) gives positive diagrams
\[
  \Theta_I,\Theta_J\colon p\longrightarrow AWB
\]
with the same bottom \(1\)-path.  Thus
\[
  \Theta_I\circ\Theta_J^{-1}
\]
is a spherical diagram over \(\mathcal L\), and by
Theorem~\ref{thm:core-closure} it represents an element \(h\in H\).

By Lemma~\ref{lem:subdivision-words}, there are refinements
\(\mathscr S_I'\) of \(\mathscr S_I\) and \(\mathscr S_J'\) of
\(\mathscr S_J\) such that
\[
  W=W_{\mathscr S_I'}=W_{\mathscr S_J'}.
\]
In the resulting tree pair, the intervals of \(\mathscr S_I'\) are
paired in order with those of \(\mathscr S_J'\).  Hence \(h(I)=J\).
\end{proof}

\begin{proposition}[Realization by an interval]
\label{prop:interval-realization}
Let \(v_1,v_2\) be inner vertices of \(\mathcal L\), and let
\[
  \xi\in\Path(\mathcal L)(v_1,v_2).
\]
There are \(a,b\in\D\), with \(a<b\), such that
\[
  v(a)=v_1,\qquad v(b)=v_2,\qquad
  \xi=\omega[a,b].
\]
\end{proposition}

\begin{proof}
Choose a nonempty \(1\)-path
\[
  P\colon v_1\longrightarrow v_2
\]
such that \(\xi=[P]\).  By Lemma~\ref{lem:accessibility}, there are nonempty
\(1\)-paths
\[
  A\colon\iota_{\mathcal L}\longrightarrow v_1,
  \qquad
  B\colon v_2\longrightarrow\tau_{\mathcal L}.
\]
Lemma~\ref{lem:boundary-paths} gives
\[
  p\simeq APB.
\]
By Lemma~\ref{lem:common-expansion}, there are positive diagrams
\[
  p\longrightarrow R,
  \qquad
  APB\longrightarrow R.
\]
Since every positive cell has a one-edge top \(1\)-path, the second
diagram is a horizontal sum
\[
  \Delta_A+\Delta_P+\Delta_B.
\]
Accordingly,
\[
  R=A'P'B',
\]
where \(\Delta_P\colon P\to P'\) is positive.

After forgetting its core labels, the positive diagram \(p\to R\) is
the diagram of a finite binary tree.  The consecutive edges forming
\(P'\) therefore determine a dyadic subdivision of some interval
\([a,b]\).  Hence
\[
  P'=W_{\mathscr S}
\]
for a dyadic subdivision \(\mathscr S\) of \([a,b]\).  Since \(A'\),
\(P'\), and \(B'\) are nonempty,
\[
  0<a<b<1.
\]
Lemma~\ref{lem:subdivision-words} gives
\[
  v(a)=v_1,\qquad v(b)=v_2,
\]
and
\[
  \xi=[P]=[P']=\omega[a,b].
\]
\end{proof}

\begin{corollary}[The interval-pair dictionary]
\label{cor:interval-pair-dictionary}
Let \(v,w\) be inner vertices of \(\mathcal L\), and let
\[
  \Omega_{v,w}
  =
  \{(a,b)\in\D^2:a<b,\ v(a)=v,\ v(b)=w\}.
\]
The group \(H\) acts on \(\Omega_{v,w}\) on the right by
\[
  (a,b)\mathbin{\cdot}h=(h(a),h(b)).
\]
The assignment
\[
  (a,b)\mathbin{\cdot}H\longmapsto\omega[a,b]
\]
is a bijection from the set of \(H\)-orbits in \(\Omega_{v,w}\) to
\[
  \Path(\mathcal L)(v,w).
\]

Consequently, the \(H\)-orbits on
\[
  \{(a,b)\in\D^2:a<b\}
\]
are in bijection with
\[
  \Path_{\mathrm{in}}(\mathcal L)
  =
  \coprod_{v,w\in V_{\mathrm{in}}(\mathcal L)}
  \Path(\mathcal L)(v,w).
\]
In particular,
\[
  \Path_{\mathrm{in}}(\mathcal L)\text{ is finite}
  \quad\Longleftrightarrow\quad
  H\text{ has finitely many orbits on increasing pairs in }\D.
\]
\end{corollary}

\begin{proof}
Proposition~\ref{prop:transport} shows that the assignment is constant
on \(H\)-orbits.  Proposition~\ref{prop:interval-realization} gives
surjectivity, and Proposition~\ref{prop:exact-transport} gives
injectivity.

Every increasing pair \((a,b)\in\D^2\) belongs to exactly one set
\(\Omega_{v,w}\), namely the one for which \(v=v(a)\) and \(w=v(b)\).
Taking the disjoint union of the preceding bijections over all inner
vertices \(v,w\) proves the final assertion.
\end{proof}

\begin{proposition}[Transport of increasing tuples]
\label{prop:tuple-transport}
Let \(k\geq2\), and let
\[
  a_1<\cdots<a_k,
  \qquad
  b_1<\cdots<b_k
\]
be points of \(\D\).  There is an element \(h\in H\) such that
\[
  h(a_i)=b_i\qquad(1\leq i\leq k)
\]
if and only if
\[
  v(a_i)=v(b_i)\qquad(1\leq i\leq k)
\]
and
\[
  \omega[a_i,a_{i+1}]=\omega[b_i,b_{i+1}]
  \qquad(1\leq i<k).
\]
\end{proposition}

\begin{proof}
Suppose first that such an element \(h\) exists.
Proposition~\ref{prop:orbit-core} gives
\[
  v(a_i)=v(b_i)
  \qquad(1\leq i\leq k),
\]
and Proposition~\ref{prop:transport} gives
\[
  \omega[a_i,a_{i+1}]
  =
  \omega[b_i,b_{i+1}]
  \qquad(1\leq i<k).
\]

Conversely, choose dyadic subdivisions of \([0,1]\) containing
respectively all the points \(a_i\) and all the points \(b_i\).  The
corresponding bottom \(1\)-paths have decompositions
\[
  A_0W_1\cdots W_{k-1}A_k
\]
and
\[
  B_0V_1\cdots V_{k-1}B_k,
\]
where
\[
  [W_i]=\omega[a_i,a_{i+1}],
  \qquad
  [V_i]=\omega[b_i,b_{i+1}].
\]
The path \(A_0\) ends at \(v(a_1)\), while \(B_0\) ends at \(v(b_1)\).
Since these vertices are equal,
Lemma~\ref{lem:boundary-paths}\textup{(1)} gives
\[
  A_0\simeq B_0.
\]
Similarly, Lemma~\ref{lem:boundary-paths}\textup{(2)} gives
\[
  A_k\simeq B_k.
\]
By hypothesis,
\[
  W_i\simeq V_i
  \qquad(1\leq i<k).
\]

Apply Lemma~\ref{lem:common-expansion} separately to each corresponding
pair of blocks.  Taking the horizontal sums of the resulting positive
diagrams gives two positive diagrams over \(\mathcal L\) with the same
bottom \(1\)-path and with the same marked block boundaries.  Let
\[
  \Theta_{\mathbf a}\colon p\longrightarrow_{\mathcal L}R,
  \qquad
  \Theta_{\mathbf b}\colon p\longrightarrow_{\mathcal L}R
\]
be these two positive diagrams.  Then
\[
  \Theta_{\mathbf a}\circ\Theta_{\mathbf b}^{-1}
\]
is a spherical diagram over \(\mathcal L\).  By
Theorem~\ref{thm:core-closure}\textup{(i)} and the closedness of \(H\),
it represents an element \(h\in H\).  The marked block boundaries give
\[
  h(a_i)=b_i
  \qquad(1\leq i\leq k).
\]
\end{proof}

\begin{corollary}[Factor placement]
\label{cor:interval-placement}
Let \(I,J\subset(0,1)\) be closed intervals with endpoints in \(\D\).
Let \(\mathscr S_I\) and \(\mathscr S_J\) be dyadic subdivisions of
\(I\) and \(J\), respectively, and put
\[
  W_I=W_{\mathscr S_I},
  \qquad
  W_J=W_{\mathscr S_J}.
\]
Suppose that
\[
  W_J\simeq C\,W_I\,D,
\]
where \(C\) and \(D\) are \(1\)-paths and either may be empty.  Then
there is a closed subinterval \(J'\subseteq J\), with endpoints in
\(\D\), such that
\[
  \omega(J')=\omega(I).
\]
Consequently, some element of \(H\) maps \(I\) onto \(J'\).
\end{corollary}

\begin{proof}
By Lemma~\ref{lem:common-expansion}, choose a common positive expansion
\(Q\), with positive diagrams
\[
  W_J\longrightarrow Q,
  \qquad
  C\,W_I\,D\longrightarrow Q.
\]
By Lemma~\ref{lem:subdivision-words},
\[
  Q=W_{\mathscr T}
\]
for a refinement \(\mathscr T\) of the chosen subdivision of \(J\).

The second positive diagram decomposes as a horizontal sum along the
indicated top blocks, with the empty blocks omitted.  Its bottom
\(1\)-path therefore has the form
\[
  Q=C'W'D',
\]
where \(W_I\to W'\) is positive.  Again by
Lemma~\ref{lem:subdivision-words},
\[
  W'=W_{\mathscr S'}
\]
for a refinement \(\mathscr S'\) of the chosen subdivision of \(I\).

The consecutive standard intervals corresponding to \(W'\) in
\(\mathscr T\) form a closed subinterval \(J'\subseteq J\) with dyadic
endpoints.  Hence
\[
  \omega(J')=[W']=[W_I]=\omega(I).
\]
Proposition~\ref{prop:exact-transport} now gives an element of \(H\)
mapping \(I\) onto \(J'\).
\end{proof}

\subsection{Boundary periods and endpoint relations}
\label{subsec:boundary-periods}

Continue to assume that \(H\leq F\) is closed and that
\(\mathcal L=C(H)\) is full.  Put
\[
  A=\pi_{\mathrm{ab}}(H).
\]

Assume throughout the remainder of this subsection that there exists
\((m,n)\in A\) such that \(mn\neq0\); equivalently, that \(A\) is not
contained in the union of the two coordinate axes.

This assumption is automatic when \(\mathcal L\) is finite.  Indeed,
there are integers \(1\leq m<n\) such that
\[
  (0^m)^+=(0^n)^+.
\]
Lemma~\ref{lem:branch-criterion} gives an element of \(H\) with branch
pair
\[
  0^m\longrightarrow 0^n,
\]
and hence an element of \(A\) with nonzero first coordinate.  Applying
the same argument to the words \(1^j\) gives an element of \(A\) with
nonzero second coordinate.  Thus the standing assumption holds.

Consequently, if \(\mathcal L\) is finite and \(A\) has rank one, then
\[
  A=\mathbb Z(a_L,a_R)
\]
for some \(a_L,a_R\in\mathbb Z\) such that \(a_La_R\neq0\).

Let
\[
  d_L=\gcd\{a:(a,b)\in A\},
  \qquad
  d_R=\gcd\{b:(a,b)\in A\},
\]
where both gcds are taken to be positive.  For \(n\geq0\), put
\[
  \ell_n=(0^n)^+,
  \qquad
  r_n=(1^n)^+.
\]

\begin{lemma}[Boundary periodicity]
\label{lem:boundary-periodicity}
There are \(N_L,N_R\geq1\) such that
\[
  \ell_{n+d_L}=\ell_n
  \qquad(n\geq N_L),
\]
\[
  r_{n+d_R}=r_n
  \qquad(n\geq N_R),
\]
and
\[
  \ell_m=\ell_n
  \quad\Longleftrightarrow\quad
  m\equiv n\pmod{d_L}
  \qquad(m,n\geq N_L),
\]
\[
  r_m=r_n
  \quad\Longleftrightarrow\quad
  m\equiv n\pmod{d_R}
  \qquad(m,n\geq N_R).
\]
\end{lemma}

\begin{proof}
By the definition of \(d_L\), there is \(h\in H\) whose first
endpoint coordinate is \(-d_L\).  Let
\[
  0^m\longrightarrow0^n
\]
be the leftmost branch pair of the reduced tree pair for \(h\).  On
\([0^m]\), the slope of \(h\) is \(2^{m-n}\).  Hence
\[
  m-n=-d_L,
\]
so \(n=m+d_L\).  Since \(h\) is nontrivial, its reduced tree pair has at least two leaves, and hence \(m\geq1\).

The reduced tree pair is accepted because \(H\) is closed.
Lemma~\ref{lem:acceptance-branches} therefore gives
\[
  \ell_m=\ell_{m+d_L}.
\]
By the uniqueness of the children of a given top edge following
\eqref{eq:core-cell}, equal edges have equal left children.  Repeatedly
taking left children gives
\[
  \ell_{m+j}=\ell_{m+d_L+j}
  \qquad(j\geq0).
\]
Thus the left periodicity holds with \(N_L=m\).

Conversely, suppose that \(\ell_m=\ell_n\).
Lemma~\ref{lem:branch-criterion} supplies an element \(g\in H\) having
the branch pair
\[
  0^m\longrightarrow0^n.
\]
On \([0^m]\), the slope of \(g\) is \(2^{m-n}\); since this interval
contains \(0\), it follows that
\[
  \log_2 g'(0^+)=m-n.
\]
By the definition of \(d_L\), we have \(d_L\mid m-n\).  Combining this
divisibility with the established periodicity proves the left-hand
equivalence.  The right-hand assertions follow symmetrically.
\end{proof}

For every \(n\geq0\), define the companion edges
\[
  \lambda_n=(0^n1)^+,
  \qquad
  \rho_n=(1^n0)^+.
\]
The corresponding positive cells are
\[
  \ell_n=\ell_{n+1}\lambda_n,
  \qquad
  r_n=\rho_n r_{n+1}.
\]
For \(n\geq1\), both words \(0^n1\) and \(1^n0\) contain both digits,
so Lemma~\ref{lem:boundary-incidence} shows that \(\lambda_n\) and
\(\rho_n\) are inner edges.  Lemma~\ref{lem:boundary-periodicity} and
the uniqueness of the two children of a top edge give
\[
  \lambda_{n+d_L}=\lambda_n
  \qquad(n\geq N_L),
\]
and
\[
  \rho_{n+d_R}=\rho_n
  \qquad(n\geq N_R).
\]

Set
\[
  \ell=\ell_{N_L},
  \qquad
  r=r_{N_R},
  \qquad
  v_L=\tau(\ell),
  \qquad
  v_R=\iota(r).
\]
Successively applying the positive cells along the left boundary gives
\[
\begin{aligned}
  \ell_{N_L}
  &\simeq
  \ell_{N_L+1}\lambda_{N_L}\\
  &\simeq
  \ell_{N_L+2}\lambda_{N_L+1}\lambda_{N_L}\\
  &\ \ \vdots\\
  &\simeq
  \ell_{N_L+d_L}
  \lambda_{N_L+d_L-1}\cdots\lambda_{N_L}\\
  &=
  \ell_{N_L}
  \lambda_{N_L+d_L-1}\cdots\lambda_{N_L}.
\end{aligned}
\]
Define
\[
  B_L=
  \lambda_{N_L+d_L-1}\cdots\lambda_{N_L}.
\]
The final line is a \(1\)-path of the form \(\ell B_L\) with the same
endpoints as \(\ell\).  Consequently, \(B_L\) is an inner \(1\)-path
from \(v_L\) to \(v_L\).  Put
\[
  \Lambda_L=[B_L]\in\Path(\mathcal L)(v_L,v_L).
\]

At the right boundary, the corresponding development is
\[
\begin{aligned}
  r_{N_R}
  &\simeq
  \rho_{N_R}r_{N_R+1}\\
  &\simeq
  \rho_{N_R}\rho_{N_R+1}r_{N_R+2}\\
  &\ \ \vdots\\
  &\simeq
  \rho_{N_R}\cdots\rho_{N_R+d_R-1}
  r_{N_R+d_R}\\
  &=
  \rho_{N_R}\cdots\rho_{N_R+d_R-1}
  r_{N_R}.
\end{aligned}
\]
Thus
\[
  B_R=
  \rho_{N_R}\cdots\rho_{N_R+d_R-1}
\]
is an inner \(1\)-path from \(v_R\) to \(v_R\).  Put
\[
  \Lambda_R=[B_R]\in\Path(\mathcal L)(v_R,v_R).
\]

The displayed developments are realized by positive diagrams
\[
  \ell\longrightarrow_{\mathcal L} \ell B_L,
  \qquad
  r\longrightarrow_{\mathcal L} B_Rr.
\]
Iterating them gives positive diagrams
\[
  \ell\longrightarrow_{\mathcal L} \ell B_L^{\,i},
  \qquad
  r\longrightarrow_{\mathcal L} B_R^{\,i} r
  \qquad(i\geq1).
\]
The one-period development and its iteration are shown in
Figure~\ref{fig:left-boundary-periods}.  The right-boundary diagram is
obtained by reflecting the same construction.

\begin{figure}[htbp]
\centering
\begin{tikzpicture}[x=.82cm,y=.58cm]
  \begin{scope}[xshift=-6.2cm]
    \foreach \i in {0,...,5}
      \coordinate (a\i) at (\i,0);
    \draw[s23edge] (a0) .. controls (1.05,4.15) and (3.95,4.15) ..
      node[s23label,above=2pt] {$\ell_{N_L}$} (a5);
    \draw[s23edge] (a0) .. controls (.85,3.20) and (3.15,3.20) ..
      node[s23tiny,above=1pt] {$\ell_{N_L+1}$} (a4);
    \draw[s23edge] (a0) .. controls (.65,2.30) and (2.35,2.30) ..
      node[s23tiny,above=1pt] {$\ell_{N_L+2}$} (a3);
    \draw[s23edge] (a0) .. controls (.45,1.42) and (1.55,1.42) ..
      node[s23tiny,above=1pt] {$\ell_{N_L+3}$} (a2);
    \draw[s23return] (a0) .. controls +(.27,.52) and +(-.27,.52) ..
      node[s23tiny,above=1pt] {$\ell$} (a1);
    \draw[s23upper] (a1) .. controls +(.27,.52) and +(-.27,.52) ..
      (a2);
    \draw[s23upper] (a2) .. controls +(.27,.52) and +(-.27,.52) ..
      (a3);
    \draw[s23upper] (a3) .. controls +(.27,.52) and +(-.27,.52) ..
      (a4);
    \draw[s23upper] (a4) .. controls +(.27,.52) and +(-.27,.52) ..
      (a5);
    \foreach \i in {0,...,5}
      \node[s23vertex] at (a\i) {};
    \node[font=\scriptsize,fill=white,inner sep=.4pt,below=3pt]
      at ($(a1)!.5!(a2)$) {$\lambda_{N_L+3}$};
    \node[font=\scriptsize,fill=white,inner sep=.4pt,below=15pt]
      at ($(a2)!.5!(a3)$) {$\lambda_{N_L+2}$};
    \node[font=\scriptsize,fill=white,inner sep=.4pt,below=3pt]
      at ($(a3)!.5!(a4)$) {$\lambda_{N_L+1}$};
    \node[font=\scriptsize,fill=white,inner sep=.4pt,below=15pt]
      at ($(a4)!.5!(a5)$) {$\lambda_{N_L}$};
    \draw[s23bracket] ($(a1)+(0,-1.75)$) -- ++(0,-.22) -|
      ($(a5)+(0,-1.75)$);
    \node[s23label,anchor=north] at ($(a1)!.5!(a5)+(0,-2.02)$)
      {$B_L$};
    \node[font=\small] at (2.5,4.85) {(a) one period, \(d_L=4\)};
  \end{scope}

  \begin{scope}[xshift=1.0cm]
    \foreach \i in {0,...,7}
      \coordinate (b\i) at (\i,0);
    \draw[s23edge] (b0) .. controls (1.45,4.35) and (5.55,4.35) ..
      node[s23label,above=2pt] {$\ell$} (b7);
    \draw[s23edge] (b0) .. controls (1.25,3.70) and (4.75,3.70) .. (b6);
    \draw[s23edge] (b0) .. controls (1.05,3.02) and (3.95,3.02) .. (b5);
    \draw[s23edge] (b0) .. controls (.85,2.38) and (3.15,2.38) .. (b4);
    \draw[s23edge] (b0) .. controls (.65,1.75) and (2.35,1.75) .. (b3);
    \draw[s23edge] (b0) .. controls (.45,1.12) and (1.55,1.12) .. (b2);
    \draw[s23return] (b0) .. controls +(.27,.48) and +(-.27,.48) ..
      node[s23tiny,above=1pt] {$\ell$} (b1);
    \foreach \i/\j in {1/2,2/3,3/4,4/5,5/6,6/7}
      \draw[s23upper] (b\i) .. controls +(.27,.48) and +(-.27,.48) ..
        (b\j);
    \foreach \i in {0,...,7}
      \node[s23vertex] at (b\i) {};
    \draw[s23bracket] ($(b1)+(0,-.75)$) -- ++(0,-.22) -|
      ($(b4)+(0,-.75)$);
    \draw[s23bracket] ($(b4)+(0,-.75)$) -- ++(0,-.22) -|
      ($(b7)+(0,-.75)$);
    \node[s23label,below=9pt] at ($(b1)!.5!(b4)$) {$B_L$};
    \node[s23label,below=9pt] at ($(b4)!.5!(b7)$) {$B_L$};
    \node[font=\small] at (3.5,4.85)
      {(b) two periods, \(d_L=3\)};
  \end{scope}
\end{tikzpicture}
\caption{The left boundary comb.  In panel \textup{(a)}, successive
development through one period gives a positive
\((\ell,\ell B_L)\)-diagram.  Panel \textup{(b)} shows two consecutive
periods: the first bottom edge is again \(\ell\), and each following
block is one copy of \(B_L\).  All edges are oriented from left to
right.}
\label{fig:left-boundary-periods}
\end{figure}
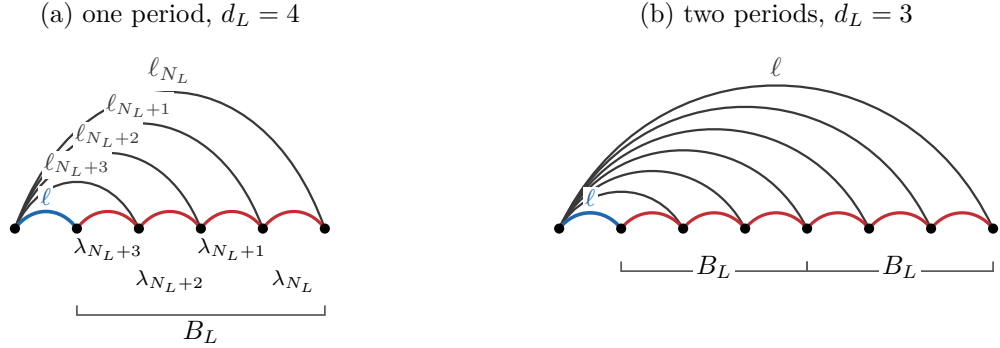

\begin{lemma}[Positive expansion of an extreme edge]
\label{lem:extreme-expansion}
\begin{enumerate}[label=\textup{(\arabic*)}]
\item Let \(P\) be a nonempty \(1\)-path beginning at \(v_L\), and let
      \[
        \Delta\colon\ell P\longrightarrow_{\mathcal L} R
      \]
      be positive.  Then there are an integer \(k\geq0\), a positive
      diagram
      \[
        P\longrightarrow_{\mathcal L} P',
      \]
      and, for \(0\leq j<k\), positive diagrams
      \[
        \lambda_{N_L+j}\longrightarrow_{\mathcal L} U_j
      \]
      such that
      \[
        R=
        \ell_{N_L+k}U_{k-1}\cdots U_0P'.
      \]
      When \(k=0\), the product of the \(U_j\) is absent.
\item Let \(P\) be a nonempty \(1\)-path ending at \(v_R\), and let
      \[
        \Delta\colon Pr\longrightarrow_{\mathcal L} R
      \]
      be positive.  Then there are an integer \(k\geq0\), a positive
      diagram
      \[
        P\longrightarrow_{\mathcal L} P',
      \]
      and, for \(0\leq j<k\), positive diagrams
      \[
        \rho_{N_R+j}\longrightarrow_{\mathcal L} V_j
      \]
      such that
      \[
        R=
        P'V_0\cdots V_{k-1}r_{N_R+k}.
      \]
      When \(k=0\), the product of the \(V_j\) is absent.
\end{enumerate}
\end{lemma}

\begin{proof}
Read the first diagram from top to bottom as a finite sequence of
replacements of one edge by the two children on the bottom of its
positive cell.  Beginning with the initial edge \(\ell\), follow the
edge obtained by choosing the left child at every such replacement.

Suppose that this edge is replaced \(k\) times.  At the \((j+1)\)-st
replacement, for \(0\leq j<k\), its left child is \(\ell_{N_L+j+1}\),
while its right child is \(\lambda_{N_L+j}\).  Thus the first edge of the
final bottom path is \(\ell_{N_L+k}\).

Any subsequent replacements beginning at the companion edge
\(\lambda_{N_L+j}\) produce a consecutive bottom path \(U_j\).  The
companion created last occurs first, so these paths appear in the order
\[
  U_{k-1},\ldots,U_0.
\]
The replacements beginning at the edges of \(P\) produce a positive
expansion \(P'\) of \(P\), and all these edges remain to the right of
the preceding blocks.  This gives the asserted decomposition of \(R\).

The proof of the second assertion is symmetric: follow the final edge
by choosing the right child, and record the left companion created at
each replacement.
\end{proof}

\begin{lemma}[Boundary comparison]
\label{lem:boundary-comparison}
\begin{enumerate}[label=\textup{(\arabic*)}]
\item If \(P,Q\colon v_L\to v\) are nonempty inner \(1\)-paths, then
      there are positive integers \(j_P,j_Q\) such that
      \[
        \Lambda_L^{j_P}[P]=\Lambda_L^{j_Q}[Q].
      \]
\item If \(P,Q\colon v\to v_R\) are nonempty inner \(1\)-paths, then
      there are positive integers \(j_P,j_Q\) such that
      \[
        [P]\Lambda_R^{j_P}=[Q]\Lambda_R^{j_Q}.
      \]
\end{enumerate}
\end{lemma}

\begin{proof}
The \(1\)-paths \(\ell P\) and \(\ell Q\) begin at \(\iota_{\mathcal L}\) and
have the same terminal vertex.  Lemma~\ref{lem:boundary-paths}\textup{(1)}
gives
\[
  \ell P\simeq\ell Q.
\]
By Lemma~\ref{lem:common-expansion}, there are positive diagrams
\[
  \ell P\longrightarrow_{\mathcal L} R,
  \qquad
  \ell Q\longrightarrow_{\mathcal L} R
\]
with the same bottom \(1\)-path \(R\).  Lemma~\ref{lem:extreme-expansion}
gives two decompositions
\[
  R=
  \ell_{N_L+k_P}U_{k_P-1}\cdots U_0P'
\]
and
\[
  R=
  \ell_{N_L+k_Q}V_{k_Q-1}\cdots V_0Q',
\]
where \(P'\) and \(Q'\) are positive expansions of \(P\) and \(Q\),
respectively.

The first edge of the same \(1\)-path \(R\) therefore has the two
labels
\[
  \ell_{N_L+k_P}=\ell_{N_L+k_Q}.
\]
Lemma~\ref{lem:boundary-periodicity} gives
\[
  k_P\equiv k_Q\pmod{d_L}.
\]
Expand this common first edge through the same number of additional
left-child cells, chosen so that the total number of left-child cells
in each diagram is a positive multiple of \(d_L\).  The resulting
common bottom \(1\)-path begins with \(\ell\) and has two decompositions
\[
  \ell S_PP'=\ell S_QQ',
\]
where, for some positive integers \(j_P,j_Q\),
\[
  [S_P]=\Lambda_L^{j_P},
  \qquad
  [S_Q]=\Lambda_L^{j_Q}.
\]
The two positive diagrams and the precise endpoints of these four
blocks are shown in Figure~\ref{fig:boundary-comparison}.

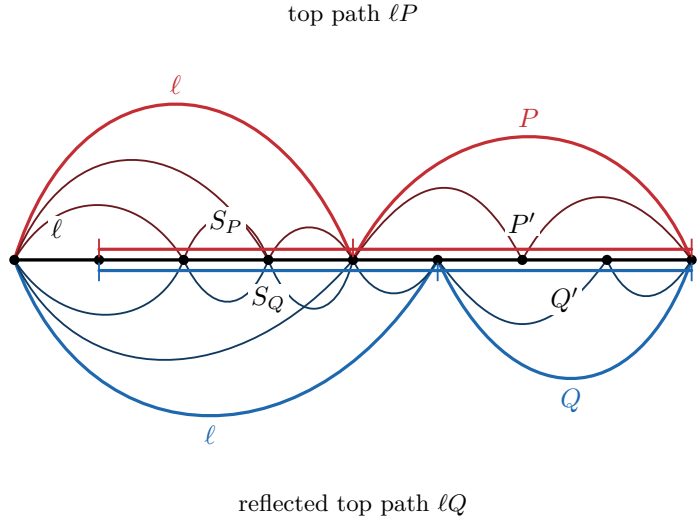
\begin{figure}[htbp]
\centering
\begin{tikzpicture}[x=1.12cm,y=.82cm]
  \foreach \i in {0,...,8}
    \coordinate (q\i) at (-4+\i,0);

  \draw[s23upper]
    (q0) .. controls (-3.25,3.35) and (-1.00,3.35) ..
    node[s23label,above=2pt] {$\ell$} (q4);
  \draw[s23upper]
    (q4) .. controls (1.00,2.65) and (3.20,2.65) ..
    node[s23label,above=2pt] {$P$} (q8);
  \draw[s23upperfaint]
    (q0) .. controls (-3.35,2.15) and (-1.95,2.15) .. (q3);
  \draw[s23upperfaint]
    (q0) .. controls (-3.50,1.18) and (-2.55,1.18) .. (q2);
  \draw[s23upperfaint]
    (q2) .. controls (-1.75,.92) and (-1.25,.92) .. (q3);
  \draw[s23upperfaint]
    (q3) .. controls (-.75,.70) and (-.30,.70) .. (q4);
  \draw[s23upperfaint]
    (q4) .. controls (.80,1.55) and (1.40,1.55) .. (q6);
  \draw[s23upperfaint]
    (q6) .. controls (2.55,1.35) and (3.25,1.35) .. (q8);

  \draw[s23lower]
    (q0) .. controls (-3.10,-3.35) and (-.40,-3.35) ..
    node[s23label,below=2pt] {$\ell$} (q5);
  \draw[s23lower]
    (q5) .. controls (1.85,-2.55) and (3.35,-2.55) ..
    node[s23label,below=2pt] {$Q$} (q8);
  \draw[s23lowerfaint]
    (q0) .. controls (-3.25,-2.15) and (-1.35,-2.15) .. (q4);
  \draw[s23lowerfaint]
    (q0) .. controls (-3.45,-1.18) and (-2.35,-1.18) .. (q2);
  \draw[s23lowerfaint]
    (q2) .. controls (-1.70,-.90) and (-1.25,-.90) .. (q3);
  \draw[s23lowerfaint]
    (q3) .. controls (-.70,-1.05) and (-.20,-1.05) .. (q4);
  \draw[s23lowerfaint]
    (q4) .. controls (.25,-.72) and (.70,-.72) .. (q5);
  \draw[s23lowerfaint]
    (q5) .. controls (1.75,-1.38) and (2.25,-1.38) .. (q7);
  \draw[s23lowerfaint]
    (q7) .. controls (3.25,-.78) and (3.65,-.78) .. (q8);

  \draw[s23common] (q0)--(q1)--(q2)--(q3)--(q4)--(q5)--(q6)--(q7)--(q8);
  \draw[s23upper] ([yshift=4pt]q1)--([yshift=4pt]q4);
  \draw[s23upper] ([yshift=4pt]q4)--([yshift=4pt]q8);
  \draw[s23lower] ([yshift=-4pt]q1)--([yshift=-4pt]q5);
  \draw[s23lower] ([yshift=-4pt]q5)--([yshift=-4pt]q8);
  \foreach \i in {0,...,8}
    \node[s23vertex] at (q\i) {};

  \node[s23label,above=7pt] at ($(q0)!.5!(q1)$) {$\ell$};
  \node[s23label,above=9pt] at ($(q1)!.5!(q4)$) {$S_P$};
  \node[s23label,above=9pt] at ($(q4)!.5!(q8)$) {$P'$};
  \node[s23label,below=9pt] at ($(q1)!.5!(q5)$) {$S_Q$};
  \node[s23label,below=9pt] at ($(q5)!.5!(q8)$) {$Q'$};
  \foreach \i in {1,4,8}
    \draw[diagramred,line width=.7pt]
      ([yshift=1pt]q\i)--([yshift=8pt]q\i);
  \foreach \i in {1,5,8}
    \draw[diagramblue,line width=.7pt]
      ([yshift=-1pt]q\i)--([yshift=-8pt]q\i);
  \node[s23tiny,align=center] at (0,3.95)
    {top path \(\ell P\)};
  \node[s23tiny,align=center] at (0,-3.95)
    {reflected top path \(\ell Q\)};
\end{tikzpicture}
\caption{The common positive expansion in
Lemma~\ref{lem:boundary-comparison}.  The upper top path has the two
blocks \(\ell\) and \(P\); the reflected lower top path has the two
blocks \(\ell\) and \(Q\).  Their common bottom \(1\)-path is
\(\ell S_PP'=\ell S_QQ'\).  The marked vertices specify the endpoints
of \(S_P,P',S_Q,Q'\); moreover
\([S_P]=\Lambda_L^{j_P}\), \([S_Q]=\Lambda_L^{j_Q}\),
\([P']=[P]\), and \([Q']=[Q]\).}
\label{fig:boundary-comparison}
\end{figure}

Removing their common first edge gives an equality of \(1\)-paths
\[
  S_PP'=S_QQ'.
\]
Passing to homotopy classes, and using \(P'\simeq P\) and
\(Q'\simeq Q\), gives
\[
  \Lambda_L^{j_P}[P]=\Lambda_L^{j_Q}[Q].
\]

For the right-hand assertion, Lemma~\ref{lem:boundary-paths}\textup{(2)}
gives
\[
  Pr\simeq Qr.
\]
Apply Lemma~\ref{lem:common-expansion}, follow the final edge using
Lemma~\ref{lem:extreme-expansion}\textup{(2)}, and develop the common
final edge far enough to complete whole right periods.  Removing that
final edge from the common bottom \(1\)-path gives the asserted
equality.
\end{proof}

For an integer \(t\), put
\[
  t_+=\max\{t,0\},
  \qquad
  t_-=\max\{-t,0\}.
\]
Thus \(t=t_+-t_-\).  A factor whose exponent is zero is omitted from
the formulas below.

\begin{proposition}[Endpoint vectors and boundary relations]
\label{prop:endpoint-relations}
Let \(x,y\in\mathbb Z\) and suppose that
\[
  (d_Lx,d_Ry)\in A.
\]
Then there are a nonempty inner \(1\)-path
\[
  Z\colon v_L\longrightarrow v_R
\]
and a positive diagram
\[
  \Theta\colon p\longrightarrow_{\mathcal L} \ell Zr
\]
such that
\[
  \Lambda_L^{x_+}[Z]\Lambda_R^{y_+}
  =
  \Lambda_L^{x_-}[Z]\Lambda_R^{y_-}.
\]

Conversely, suppose that \(Z\colon v_L\to v_R\) is a nonempty inner
\(1\)-path and that there is a positive diagram
\[
  \Theta\colon p\longrightarrow_{\mathcal L} \ell Zr.
\]
If \(\alpha,\beta,\gamma,\delta\) are nonnegative integers and
\[
  \Lambda_L^\alpha[Z]\Lambda_R^\beta
  =
  \Lambda_L^\gamma[Z]\Lambda_R^\delta,
\]
then
\[
  \bigl((\alpha-\gamma)d_L,\,
        (\beta-\delta)d_R\bigr)\in A.
\]
\end{proposition}

\begin{proof}
Choose \(h\in H\) with
\[
  \pi_{\mathrm{ab}}(h)=(d_Lx,d_Ry).
\]
Let \((T_+,T_-)\) be the reduced tree pair for \(h\), with \(T_+\) the
domain tree and \(T_-\) the range tree.  Since \(H\) is closed, its
tree-pair diagram is accepted.  By
Lemma~\ref{lem:acceptance-reduction}, every common expansion of this
tree pair is accepted.

If necessary, first expand a paired leaf so that both trees have at
least three leaves.  Then expand the paired extreme leaves
simultaneously until their leftmost depths are at least \(N_L\) and
congruent to \(N_L\) modulo \(d_L\), and their rightmost depths are at
least \(N_R\) and congruent to \(N_R\) modulo \(d_R\).

Lemma~\ref{lem:acceptance-branches} therefore gives the same bottom
\(1\)-path for the positive diagrams \(\Psi_{T_+}\) and
\(\Psi_{T_-}\).

Let \(m_L,m_R\) be the leftmost and rightmost depths of the domain
tree, and let \(n_L,n_R\) be those of the range tree.  Then
\[
  m_L-n_L=d_Lx,
  \qquad
  m_R-n_R=d_Ry.
\]
By Lemma~\ref{lem:boundary-periodicity}, the first edge in their common
bottom path is labelled \(\ell\), and the final edge is labelled \(r\).
Writing this common bottom path as
\[
  \ell Zr,
\]
the path \(Z\) is nonempty because the trees have at least three
leaves.  Every non-extreme leaf address contains both digits.
Lemma~\ref{lem:boundary-incidence} therefore shows that every edge of
\(Z\) is inner.  Either \(\Psi_{T_+}\) or \(\Psi_{T_-}\) may be taken
as the positive diagram \(\Theta\).

Let \(U\) be the least common expansion of the domain and range trees,
obtained by taking the union of their caret sets.  Its leftmost and
rightmost depths are
\[
  \max\{m_L,n_L\},
  \qquad
  \max\{m_R,n_R\}.
\]
Consequently, the domain tree acquires \(d_Lx_-\) additional left
levels and \(d_Ry_-\) additional right levels, whereas the range tree
acquires \(d_Lx_+\) additional left levels and \(d_Ry_+\) additional
right levels.

Both refinements have the same final bottom \(1\)-path.  Its extreme
edges are again labelled \(\ell\) and \(r\), so write it as
\[
  \ell Mr.
\]
Applying Lemma~\ref{lem:extreme-expansion} to the first and final parts
of the domain refinement gives a positive diagram
\[
  B_L^{x_-}ZB_R^{y_-}\longrightarrow_{\mathcal L} M.
\]
The range refinement similarly gives a positive diagram
\[
  B_L^{x_+}ZB_R^{y_+}\longrightarrow_{\mathcal L} M.
\]
Figure~\ref{fig:endpoint-common-refinement} displays these two middle
diagrams and the endpoints of all six top blocks.

\begin{figure}[htbp]
\centering
\begin{tikzpicture}[x=.67cm,y=.68cm]
  \begin{scope}[yshift=4.75cm]
    \foreach \i in {0,...,14}
      \coordinate (a\i) at (\i,0);

    \path[fill=diagramred!8]
      (a1) .. controls +(.48,1.02) and +(-.48,1.02) .. (a3)
      -- (a2) -- cycle;
    \path[fill=diagramred!8]
      (a3) .. controls +(.48,1.02) and +(-.48,1.02) .. (a5)
      -- (a4) -- cycle;
    \path[fill=diagramred!8]
      (a5) .. controls +(.48,1.02) and +(-.48,1.02) .. (a7)
      -- (a6) -- cycle;
    \path[fill=diagramred!8]
      (a7) .. controls +(.82,1.58) and +(-.82,1.58) .. (a10)
      .. controls +(-.48,.72) and +(.48,.72) .. (a8) -- cycle;
    \path[fill=diagramred!13]
      (a8) .. controls +(.48,.72) and +(-.48,.72) .. (a10)
      -- (a9) -- cycle;
    \path[fill=diagramred!8]
      (a10) .. controls +(.48,1.02) and +(-.48,1.02) .. (a12)
      -- (a11) -- cycle;

    \draw[s23common] (a0)--(a1)--(a2)--(a3)--(a4)--(a5)--(a6)--(a7)
      --(a8)--(a9)--(a10)--(a11)--(a12)--(a13)--(a14);
    \draw[s23upper] (a1) .. controls +(.48,1.02) and +(-.48,1.02) .. (a3);
    \draw[s23upper] (a3) .. controls +(.48,1.02) and +(-.48,1.02) .. (a5);
    \draw[s23upper] (a5) .. controls +(.48,1.02) and +(-.48,1.02) .. (a7);
    \draw[s23upper] (a7) .. controls +(.82,1.58) and +(-.82,1.58) .. (a10);
    \draw[s23upperfaint]
      (a8) .. controls +(.48,.72) and +(-.48,.72) .. (a10);
    \draw[s23upper] (a10) .. controls +(.48,1.02) and +(-.48,1.02) .. (a12);
    \draw[s23upper] (a12)--(a13);
    \draw[s23edge] (a0)--(a1) (a13)--(a14);

    \foreach \i in {0,...,14}
      \node[s23vertex] at (a\i) {};
    \node[s23tiny,anchor=north east] at ($(a0)+(0,-.12)$)
      {$\iota_{\mathcal L}$};
    \node[s23tiny,anchor=north] at ($(a1)+(0,-.12)$) {$v_L$};
    \node[s23tiny,anchor=north] at ($(a5)+(0,-.12)$) {$v_L$};
    \node[s23tiny,anchor=north] at ($(a10)+(0,-.12)$) {$v_R$};
    \node[s23tiny,anchor=north] at ($(a13)+(0,-.12)$) {$v_R$};
    \node[s23tiny,anchor=north west] at ($(a14)+(0,-.12)$)
      {$\tau_{\mathcal L}$};
    \node[s23label,above=4pt] at ($(a0)!.5!(a1)$) {$\ell$};
    \node[s23label,above=4pt] at ($(a13)!.5!(a14)$) {$r$};

    \draw[s23bracket] ($(a1)+(0,2.00)$) -- ++(0,.20) -|
      ($(a5)+(0,2.00)$);
    \draw[s23bracket] ($(a5)+(0,2.00)$) -- ++(0,.20) -|
      ($(a10)+(0,2.00)$);
    \draw[s23bracket] ($(a10)+(0,2.00)$) -- ++(0,.20) -|
      ($(a13)+(0,2.00)$);
    \node[s23tiny] at ($(a1)!.5!(a5)+(0,2.42)$) {$B_L^{x_-}$};
    \node[s23label] at ($(a5)!.5!(a10)+(0,2.42)$) {$Z$};
    \node[s23tiny] at ($(a10)!.5!(a13)+(0,2.42)$) {$B_R^{y_-}$};
    \draw[s23bracket] ($(a1)+(0,-.72)$) -- ++(0,-.20) -|
      ($(a13)+(0,-.72)$);
    \node[s23label,anchor=north] at ($(a1)!.5!(a13)+(0,-1.02)$) {$M$};
    \node[font=\small] at (7,2.95) {\textup{(a)} domain-side refinement};
  \end{scope}

  \begin{scope}
    \foreach \i in {0,...,14}
      \coordinate (b\i) at (\i,0);

    \path[fill=diagramblue!8]
      (b1) .. controls +(.48,1.02) and +(-.48,1.02) .. (b3)
      -- (b2) -- cycle;
    \path[fill=diagramblue!8]
      (b4) .. controls +(.48,1.02) and +(-.48,1.02) .. (b6)
      -- (b5) -- cycle;
    \path[fill=diagramblue!8]
      (b6) .. controls +(.82,1.58) and +(-.82,1.58) .. (b9)
      .. controls +(-.48,.72) and +(.48,.72) .. (b7) -- cycle;
    \path[fill=diagramblue!13]
      (b7) .. controls +(.48,.72) and +(-.48,.72) .. (b9)
      -- (b8) -- cycle;
    \path[fill=diagramblue!8]
      (b9) .. controls +(.48,1.02) and +(-.48,1.02) .. (b11)
      -- (b10) -- cycle;
    \path[fill=diagramblue!8]
      (b11) .. controls +(.48,1.02) and +(-.48,1.02) .. (b13)
      -- (b12) -- cycle;

    \draw[s23common] (b0)--(b1)--(b2)--(b3)--(b4)--(b5)--(b6)--(b7)
      --(b8)--(b9)--(b10)--(b11)--(b12)--(b13)--(b14);
    \draw[s23lower] (b1) .. controls +(.48,1.02) and +(-.48,1.02) .. (b3);
    \draw[s23lower] (b3)--(b4);
    \draw[s23lower] (b4) .. controls +(.48,1.02) and +(-.48,1.02) .. (b6);
    \draw[s23lower] (b6) .. controls +(.82,1.58) and +(-.82,1.58) .. (b9);
    \draw[s23lowerfaint]
      (b7) .. controls +(.48,.72) and +(-.48,.72) .. (b9);
    \draw[s23lower] (b9) .. controls +(.48,1.02) and +(-.48,1.02) .. (b11);
    \draw[s23lower] (b11) .. controls +(.48,1.02) and +(-.48,1.02) .. (b13);
    \draw[s23edge] (b0)--(b1) (b13)--(b14);

    \foreach \i in {0,...,14}
      \node[s23vertex] at (b\i) {};
    \node[s23tiny,anchor=north east] at ($(b0)+(0,-.12)$)
      {$\iota_{\mathcal L}$};
    \node[s23tiny,anchor=north] at ($(b1)+(0,-.12)$) {$v_L$};
    \node[s23tiny,anchor=north] at ($(b4)+(0,-.12)$) {$v_L$};
    \node[s23tiny,anchor=north] at ($(b9)+(0,-.12)$) {$v_R$};
    \node[s23tiny,anchor=north] at ($(b13)+(0,-.12)$) {$v_R$};
    \node[s23tiny,anchor=north west] at ($(b14)+(0,-.12)$)
      {$\tau_{\mathcal L}$};
    \node[s23label,above=4pt] at ($(b0)!.5!(b1)$) {$\ell$};
    \node[s23label,above=4pt] at ($(b13)!.5!(b14)$) {$r$};

    \draw[s23bracket] ($(b1)+(0,2.00)$) -- ++(0,.20) -|
      ($(b4)+(0,2.00)$);
    \draw[s23bracket] ($(b4)+(0,2.00)$) -- ++(0,.20) -|
      ($(b9)+(0,2.00)$);
    \draw[s23bracket] ($(b9)+(0,2.00)$) -- ++(0,.20) -|
      ($(b13)+(0,2.00)$);
    \node[s23tiny] at ($(b1)!.5!(b4)+(0,2.42)$) {$B_L^{x_+}$};
    \node[s23label] at ($(b4)!.5!(b9)+(0,2.42)$) {$Z$};
    \node[s23tiny] at ($(b9)!.5!(b13)+(0,2.42)$) {$B_R^{y_+}$};
    \draw[s23bracket] ($(b1)+(0,-.72)$) -- ++(0,-.20) -|
      ($(b13)+(0,-.72)$);
    \node[s23label,anchor=north] at ($(b1)!.5!(b13)+(0,-1.02)$) {$M$};
    \node[font=\small] at (7,2.95) {\textup{(b)} range-side refinement};
  \end{scope}
\end{tikzpicture}
\caption{The two positive refinements used in
Proposition~\ref{prop:endpoint-relations}, drawn separately.  In
panel~\textup{(a)}, the middle top \(1\)-path is
\(B_L^{x_-}ZB_R^{y_-}\); in panel~\textup{(b)}, it is
\(B_L^{x_+}ZB_R^{y_+}\).  In both panels the outer edges \(\ell\) and
\(r\) are unchanged, and the middle bottom \(1\)-path is the same
labelled path \(M\).  The marked junctions have labels
\(v_L,v_L,v_R,v_R\), so the three middle top blocks are composable.
Every displayed cell is positive, and no cell crosses a block
boundary.  A boundary block whose exponent is zero is omitted.}
\label{fig:endpoint-common-refinement}
\end{figure}
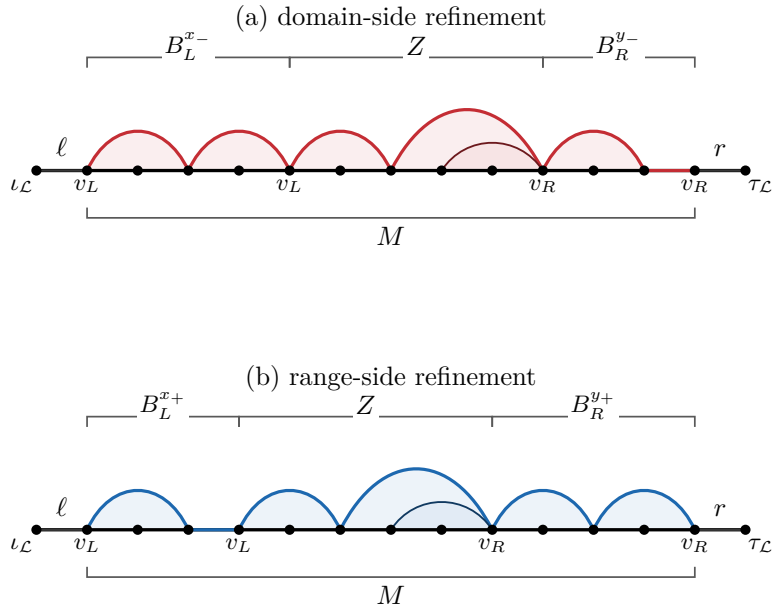

Passing to homotopy classes gives
\[
  \Lambda_L^{x_+}[Z]\Lambda_R^{y_+}
  =
  \Lambda_L^{x_-}[Z]\Lambda_R^{y_-}.
\]

For the converse, the assumed equality means that
\[
  B_L^\alpha ZB_R^\beta
  \simeq
  B_L^\gamma ZB_R^\delta.
\]
Lemma~\ref{lem:common-expansion} supplies positive diagrams from these
two \(1\)-paths to a common positive expansion \(M\).

Beginning with \(\Theta\), develop its first bottom edge through
\(\alpha\) complete left periods and its final bottom edge through
\(\beta\) complete right periods.  Attach the first middle expansion
between the resulting outer edges.  This produces a positive diagram
\[
  \Theta_{\alpha,\beta}\colon
  p\longrightarrow_{\mathcal L} \ell Mr.
\]
The same construction with \(\gamma,\delta\) gives
\[
  \Theta_{\gamma,\delta}\colon
  p\longrightarrow_{\mathcal L} \ell Mr.
\]

Suppose that the underlying tree of \(\Theta\) has leftmost and
rightmost depths \(s_L,s_R\).  The underlying trees of these two
diagrams have respective extreme depths
\[
  s_L+\alpha d_L,
  \qquad
  s_R+\beta d_R,
\]
and
\[
  s_L+\gamma d_L,
  \qquad
  s_R+\delta d_R.
\]
The spherical diagram
\[
  \Theta_{\alpha,\beta}
  \circ
  \Theta_{\gamma,\delta}^{-1}
\]
therefore has endpoint vector
\[
  \bigl((\alpha-\gamma)d_L,\,
        (\beta-\delta)d_R\bigr).
\]
It is a spherical \((p,p)\)-diagram over \(\mathcal L\).  By
Theorem~\ref{thm:core-closure}\textup{(i)} and the closedness of \(H\),
it represents an element of \(H\).  Hence its endpoint vector belongs
to \(A\).
\end{proof}

\begin{corollary}[Rectangular and coupled boundary relations]
\label{cor:boundary-relations}
\begin{enumerate}[label=\textup{(\arabic*)}]
\item If
      \[
        A=d_L\mathbb Z\times d_R\mathbb Z,
      \]
      then there are nonempty inner \(1\)-paths
      \[
        Z_L,Z_R\colon v_L\longrightarrow v_R
      \]
      and positive diagrams
      \[
        p\longrightarrow_{\mathcal L} \ell Z_Lr,
        \qquad
        p\longrightarrow_{\mathcal L} \ell Z_Rr
      \]
      such that
      \[
        \Lambda_L[Z_L]=[Z_L],
        \qquad
        [Z_R]\Lambda_R=[Z_R].
      \]
\item There are an integer \(t\), a nonempty inner \(1\)-path
      \[
        Z\colon v_L\longrightarrow v_R,
      \]
      and a positive diagram
      \[
        p\longrightarrow_{\mathcal L} \ell Zr
      \]
      such that
      \[
        (-d_L,t\,d_R)\in A
      \]
      and
      \[
        \Lambda_L[Z]=[Z]\Lambda_R^t
        \qquad\text{if }t\geq0,
      \]
      whereas
      \[
        \Lambda_L[Z]\Lambda_R^{-t}=[Z]
        \qquad\text{if }t<0.
      \]
      If \(A\neq d_L\mathbb Z\times d_R\mathbb Z\), then every integer
      \(t\) satisfying \((-d_L,t\,d_R)\in A\) is nonzero.
\end{enumerate}
\end{corollary}

\begin{proof}
For the first assertion, apply Proposition~\ref{prop:endpoint-relations}
to the vectors
\[
  (-d_L,0)
  \qquad\text{and}\qquad
  (0,-d_R).
\]

For the second, the definition of \(d_L\) gives a vector of \(A\) whose
first coordinate is \(-d_L\).  Its second coordinate belongs to
\(d_R\mathbb Z\), so it has the form
\[
  (-d_L,t\,d_R)
\]
for some \(t\in\mathbb Z\).  Proposition~\ref{prop:endpoint-relations}
gives the displayed relation according to the sign of \(t\).

Finally, suppose that \(t=0\).  By the definition of \(d_R\), there
is an element \((a,d_R)\in A\).  Since \(d_L\mid a\), write
\(a=s\,d_L\), where \(s\in\mathbb Z\).  Together with
\((-d_L,0)\in A\), this gives \((0,d_R)\in A\).  Thus \(A\) equals
\[
  d_L\mathbb Z\times d_R\mathbb Z.
\]
This proves the final assertion.
\end{proof}

\begin{corollary}[Rank-one endpoint image]
\label{cor:rank-one-infinite-paths}
Suppose that \(A\) has rank one, say
\[
  A=\mathbb Z(a_L,a_R),
\]
where \(a_La_R\neq0\).
Then there are infinitely many pairwise
nonhomotopic nonempty inner \(1\)-paths from \(v_L\) to \(v_R\).
Equivalently, \(\Path(\mathcal L)(v_L,v_R)\) is infinite.
\end{corollary}

\begin{proof}
Choose \(Z\colon v_L\to v_R\) and a positive diagram
\[
  \Theta\colon p\longrightarrow_{\mathcal L}\ell Zr
\]
as in Corollary~\ref{cor:boundary-relations}\textup{(2)}.  We claim that
the classes
\[
  \Lambda_L^n[Z]
  \qquad(n\geq1)
\]
are pairwise distinct.  Indeed, if
\[
  \Lambda_L^n[Z]=\Lambda_L^m[Z]
\]
for some \(n>m\geq1\), the converse direction of
Proposition~\ref{prop:endpoint-relations}, with \(\alpha=n\),
\(\gamma=m\), and \(\beta=\delta=0\), gives
\[
  ((n-m)d_L,0)\in A.
\]
This is impossible because every nonzero element of
\(\mathbb Z(a_L,a_R)\) has nonzero second coordinate.  Hence the
displayed classes are pairwise distinct.
\end{proof}

\section{The rank-one endpoint image}
\label{sec:rank-one-endpoint-image}

For \(n\geq0\), a group \(G\) is of type \(\mathrm{FP}_n\) if the
trivial \(\mathbb ZG\)-module \(\mathbb Z\) has a projective resolution
whose modules in dimensions \(0,\ldots,n\) are finitely generated.
It is of type \(\mathrm{FP}_\infty\) if it has such a resolution with
finitely generated modules in every dimension.

For \(f\in F\), let \(\operatorname{Br}(f)\) denote the finite set of
breakpoints of \(f\).  We shall use the following lemma; see
\cite[Lemma~2.2]{BleakBroughHermiller21}.

\begin{lemma}
\label{lem:breakpoint-orbits}
Let \(G\leq F\) be finitely generated.  Then
\[
  \bigcup_{g\in G}\operatorname{Br}(g)
\]
is contained in a finite union of \(G\)-orbits.
\end{lemma}

\begin{lemma}
\label{lem:nonkernel-no-dyadic-fixed-point}
Let \(H\leq F\) be closed, suppose that \(C(H)\) is finite and full, and
suppose that
\[
  A=\pi_{\mathrm{ab}}(H)
\]
has rank one.  If \(h\in H\) fixes a point of \(\D\), then
\[
  \pi_{\mathrm{ab}}(h)=(0,0).
\]
\end{lemma}

\begin{proof}
Since \(C(H)\) is finite, as noted above, there are nonzero integers
\(a_L,a_R\) such that
\[
  A=\mathbb Z(a_L,a_R).
\]
Suppose that \(h\) fixes \(c\in\D\) and that
\[
  \pi_{\mathrm{ab}}(h)=m(a_L,a_R),
  \qquad m\neq0.
\]
Define
\[
  h_c(t)=
  \begin{cases}
    h(t),&0\leq t\leq c,\\
    t,&c\leq t\leq1.
  \end{cases}
\]
The element \(h_c\) is a piecewise-\(H\) function belonging to \(F\).
Since \(H\) is closed, \(h_c\in H\).  On the other hand,
\[
  \pi_{\mathrm{ab}}(h_c)=(ma_L,0),
\]
which does not belong to \(\mathbb Z(a_L,a_R)\).  This contradiction
proves the lemma.
\end{proof}

We shall use the following consequence of the Bieri--Strebel splitting
theorem.

\begin{theorem}[Bieri--Strebel]
\label{thm:bieri-strebel-consequence}
Let \(G\) be of type \(\mathrm{FP}_2\), suppose that \(G\) contains no
nonabelian free subgroup, and let
\[
  \chi\colon G\longrightarrow\mathbb Z
\]
be an epimorphism with nontrivial kernel.  Then there are a nontrivial
finitely generated subgroup
\[
  K\leq\ker\chi
\]
and an element \(s\in G\) such that
\[
  \chi(s)\in\{1,-1\}
  \qquad\text{and}\qquad
  s^{-1}Ks\leq K.
\]
\end{theorem}

This follows from \cite[Theorem~A]{BieriStrebel}; see also
\cite[Theorem~2.7]{BGK}.

\begin{theorem}
\label{thm:rank-one-not-FP2}
Let \(H\leq F\) be closed, and suppose that \(C(H)\) is finite and full.
If
\[
  A=\pi_{\mathrm{ab}}(H)
\]
has rank one, then \(H\) is not of type \(\mathrm{FP}_2\).
\end{theorem}

\begin{proof}
Since \(C(H)\) is finite, as noted above, there are nonzero integers
\(a_L,a_R\) such that
\[
  A=\mathbb Z(a_L,a_R).
\]
Define an epimorphism
\[
  \chi\colon H\longrightarrow\mathbb Z
\]
by
\[
  \pi_{\mathrm{ab}}(h)=\chi(h)(a_L,a_R),
\]
and put \(N=\ker\chi\).

The subgroup \(N\) is nontrivial.  Otherwise \(H\cong\mathbb Z\), but a
fundamental domain for a generator on one of its orbitals contains
infinitely many dyadic points belonging to distinct \(H\)-orbits,
contradicting Proposition~\ref{prop:finite-full-core}.

Suppose, for a contradiction, that \(H\) is of type \(\mathrm{FP}_2\).
Since \(H\) contains no nonabelian free subgroup \cite{BrinSquier},
Theorem~\ref{thm:bieri-strebel-consequence} gives a nontrivial finitely
generated subgroup
\[
  K\leq N
\]
and an element \(s\in H\) such that
\[
  \chi(s)\in\{1,-1\}
  \qquad\text{and}\qquad
  s^{-1}Ks\leq K.
\]

Choose \(1\neq k\in K\), and let \((x,y)\) be the leftmost orbital of
\(k\).  Since \(k\in N\), it is the identity on neighborhoods of \(0\)
and \(1\).  Therefore \(x\in(0,1)\), and \(x\) is a dyadic breakpoint
of \(k\).

Since \(\chi(s)\neq0\),
Lemma~\ref{lem:nonkernel-no-dyadic-fixed-point} implies that \(s\) fixes
no point of \(\D\).  Hence
\[
  s(x)\neq x.
\]

For \(n\geq0\), put
\[
  k_n=s^{-n}ks^n,
  \qquad
  x_n=s^n(x).
\]
Then \(k_n\in K\), and \(x_n\) is the left endpoint of the leftmost
orbital of \(k_n\).  In particular, \(x_n\) is a breakpoint of \(k_n\).
The points \(x_n\) are pairwise distinct.

By Lemma~\ref{lem:breakpoint-orbits}, two of these breakpoints belong to
the same \(K\)-orbit.  Hence there are \(i<j\) and \(c\in K\) such that
\[
  c(x_i)=x_j=s^{j-i}(x_i).
\]
Thus
\[
  g=s^{j-i}c^{-1}
\]
fixes the dyadic point \(x_i\).  Since \(c\in K\leq N\),
\[
  \chi(g)=(j-i)\chi(s)\neq0.
\]
Therefore \(\pi_{\mathrm{ab}}(g)\neq(0,0)\), contradicting
Lemma~\ref{lem:nonkernel-no-dyadic-fixed-point}.
\end{proof}

\section{The minimal rank-two case}
\label{sec:minimal-rank-two}

Throughout Subsections~\ref{subsec:stabilizer-orbitals}--\ref{subsec:exact-inner-path-count},
let \(H\leq F\) be closed, suppose that the action of \(H\) on
\((0,1)\) is minimal and that
\[
  A=\pi_{\mathrm{ab}}(H)
\]
has rank two, and write
\[
  \mathcal L=C(H).
\]
By Corollary~\ref{cor:minimal-full}, the core \(\mathcal L\) is full.

\subsection{Orbitals and endpoint cofinality}
\label{subsec:stabilizer-orbitals}

For \(E\subseteq[0,1]\), write
\[
  H_E=\{h\in H:h|_E=\operatorname{id}_E\},
  \qquad
  H_x=H_{\{x\}}.
\]

\begin{lemma}[One-sided endpoint slopes]
\label{lem:one-sided-slopes}
Let \(H\leq F\) be closed, suppose that its action on \((0,1)\) is
minimal and that \(\pi_{\mathrm{ab}}(H)\) has rank two.
For every \(a\in(0,1)\), the subgroup \(H_{[a,1]}\) contains an
element \(g\) with \(g'(0^+)\neq1\), and \(H_{[0,a]}\) contains an
element \(g\) with \(g'(1^-)\neq1\).
\end{lemma}

\begin{proof}
Choose \(h\in H\) whose endpoint vector is \((k,0)\), where
\(k\neq0\).  It fixes \([d,1]\) pointwise for some \(d\in\D\).  Since
\((0,1)\) is an orbital of \(H\),
\[
  \inf Hd=0.
\]
Thus there is \(g_1\in H\) such that \(g_1(d)<a\).  Then
\(g_1^{-1}hg_1\) fixes \([g_1(d),1]\) pointwise, so it belongs to
\(H_{[a,1]}\), and its slope at \(0\) is nontrivial.  The other
assertion is symmetric.
\end{proof}

\begin{lemma}[Fixed-point transfer]
\label{lem:fixed-point-transfer}
Let \(H\leq F\) be closed, suppose that its action on \((0,1)\) is
minimal and that \(\pi_{\mathrm{ab}}(H)\) has rank two.
Let \(a\in(0,1)\).
If \(H_{[a,1]}\) fixes a point in \((0,a)\), then \(H_a\) fixes a
point in \((0,a)\).  The symmetric assertion holds at the other
endpoint.
\end{lemma}

\begin{proof}
Let \(G=H_{[a,1]}\), and let \(B\) be the set of fixed points of \(G\)
in \((0,a)\).  Assume that \(B\neq\varnothing\).  By
Lemma~\ref{lem:one-sided-slopes}, \(G\) contains an element \(g_0\) with
nontrivial slope at \(0\).  Hence \(g_0\) moves every point in some
interval \((0,\varepsilon)\), and therefore \(c=\inf B>0\).  Since the
fixed-point set of \(G\) is closed, \(c\in B\).

Every \(g\in H_a\) normalizes \(G\), because it preserves \([a,1]\)
setwise.  It therefore preserves \(B\) and fixes its infimum \(c\).
\end{proof}

The following proposition can be viewed as a generalization of
\cite[Lemma~3.5]{GenFn}.

\begin{proposition}[Stabilizer orbitals]
\label{prop:stabilizer-orbitals}
Let \(H\leq F\) be closed, suppose that its action on \((0,1)\) is
minimal and that \(\pi_{\mathrm{ab}}(H)\) has rank two.
For every \(b\in(0,1)\), the interval \((0,b)\) is an orbital of
\(H_{[b,1]}\).  For every \(a\in(0,1)\), the interval \((a,1)\) is an
orbital of \(H_{[0,a]}\).
\end{proposition}

\begin{proof}
We prove the first statement.  Fix \(a_0\in(0,1)\).  If
\(H_{[a_0,1]}\) has no fixed point in \((0,a_0)\), put \(c=a_0\);
then \((0,c)\) is already an orbital of \(H_{[c,1]}\).  Otherwise, let
\(B\) be the set of fixed points of \(H_{[a_0,1]}\) in \((0,a_0)\),
and let \(c=\inf B\).

In this case, the proof of Lemma~\ref{lem:fixed-point-transfer} gives
\(0<c<a_0\) and shows that \(H_{a_0}\) fixes \(c\).  Every \(x<c\) is
moved by some element of \(H_{[a_0,1]}\), while this subgroup fixes \(c\).
Consequently,
\(H_c\) has no fixed point in \((0,c)\).  The contrapositive of
Lemma~\ref{lem:fixed-point-transfer}, applied at \(c\), now shows that
\(H_{[c,1]}\) has no fixed point in \((0,c)\).  Thus \((0,c)\) is an
orbital of \(H_{[c,1]}\).

In either case, \((0,c)\) is an orbital of \(H_{[c,1]}\).

Let \(b\in(0,1)\) and \(y\in(0,b)\).  By minimality, there is
\(h\in H\) such that
\[
  y<h(c)<b.
\]
Conjugating the orbital obtained above, we find that \((0,h(c))\) is
an orbital of \(H_{[h(c),1]}\).  Since
\[
  H_{[h(c),1]}\leq H_{[b,1]},
\]
some element of \(H_{[b,1]}\) moves \(y\).  Since \(y\) was arbitrary,
\((0,b)\) is an orbital of \(H_{[b,1]}\).

The other assertion follows by reflection.
\end{proof}

\begin{corollary}[Endpoint cofinality]
\label{cor:endpoint-cofinality}
Let \(H\leq F\) be closed, suppose that its action on \((0,1)\) is
minimal and that \(\pi_{\mathrm{ab}}(H)\) has rank two.
If \(0<a<x<1\), then
\[
  \inf H_{[0,a]}x=a,
  \qquad
  \sup H_{[0,a]}x=1.
\]
If \(0<x<b<1\), then
\[
  \inf H_{[b,1]}x=0,
  \qquad
  \sup H_{[b,1]}x=b.
\]
\end{corollary}

\begin{proof}
By Proposition~\ref{prop:stabilizer-orbitals}, \((a,1)\) is an orbital
of \(H_{[0,a]}\) and \((0,b)\) is an orbital of \(H_{[b,1]}\).  The
assertions are the property of orbitals recorded in
Subsection~\ref{subsec:minimal-sets}.
\end{proof}

\begin{corollary}[Endpoint shrinking]
\label{cor:endpoint-shrinking}
Let \(H\leq F\) be closed, suppose that its action on \((0,1)\) is
minimal and that \(\pi_{\mathrm{ab}}(H)\) has rank two, and let
\(\mathcal L=C(H)\).  Let \(0<a<b<1\) be dyadic points and put
\[
  W=\omega[a,b].
\]
For every \(\varepsilon\) satisfying
\[
  0<\varepsilon<\frac{b-a}{2},
\]
there exist dyadic points \(c,d\) such that
\[
  a<c<a+\varepsilon<b-\varepsilon<d<b
\]
and
\[
  W=\omega[a,c]=\omega[d,b].
\]
\end{corollary}

\begin{proof}
By Corollary~\ref{cor:endpoint-cofinality}, choose
\(h_1\in H_{[0,a]}\) such that \(c=h_1(b)\) lies in
\((a,a+\varepsilon)\), and choose \(h_2\in H_{[b,1]}\) such that
\(d=h_2(a)\) lies in \((b-\varepsilon,b)\).  The points \(c,d\) are
dyadic, and Proposition~\ref{prop:transport} gives
\[
  W=\omega[a,c]=\omega[d,b].
\]
\end{proof}

\subsection{Regularity, local divisibility, and the groupoid structure}
\label{subsec:minimal-semigroupoid}

We use the following convention for semigroupoids; see
\cite[pp.~603--604]{Szendrei93}.  A semigroupoid \(\mathcal S\)
consists of a set \(\operatorname{Ob}(\mathcal S)\) of objects and, for
every pair
\[
  u,v\in\operatorname{Ob}(\mathcal S),
\]
a set \(\mathcal S(u,v)\) of arrows from \(u\) to \(v\).  These sets are
pairwise disjoint.

For every three objects \(u,v,w\), the structure includes a composition
rule which assigns to every pair
\[
  P\in\mathcal S(u,v),
  \qquad
  Q\in\mathcal S(v,w)
\]
an arrow
\[
  PQ\in\mathcal S(u,w).
\]
Composition is associative: if \(t,u,v,w\) are objects and
\[
  P\in\mathcal S(t,u),
  \qquad
  Q\in\mathcal S(u,v),
  \qquad
  R\in\mathcal S(v,w),
\]
then
\[
  (PQ)R=P(QR).
\]
Thus products are written in the order in which the arrows are
traversed.  We do not assume the existence of identity arrows.

Let \(u,v\) be objects.  An arrow \(W\in\mathcal S(u,v)\) is called
\emph{regular} if there exists \(X\in\mathcal S(v,u)\) such that
\[
  W=WXW.
\]
The semigroupoid \(\mathcal S\) is called \emph{regular} if every arrow
of \(\mathcal S\) is regular.

Let \(u\) be an object and let \(P,Q\in\mathcal S(u,u)\).  We say that
\(P\) is a \emph{left divisor} of \(Q\) if
\[
  Q=PA
\]
for some \(A\in\mathcal S(u,u)\), and that \(P\) is a \emph{right
divisor} of \(Q\) if
\[
  Q=BP
\]
for some \(B\in\mathcal S(u,u)\).

The following lemma is completely standard.  We include its short proof
for completeness.

\begin{lemma}[A groupoid criterion for regular semigroupoids]
\label{lem:local-cancellation-criterion}
Let \(\mathcal S\) be a regular semigroupoid.  Suppose that, for every
object \(u\),
\[
  \mathcal S(u,u)\neq\varnothing,
\]
and that every arrow in \(\mathcal S(u,u)\) is both a left divisor and a
right divisor of every arrow in \(\mathcal S(u,u)\).  Then, for every
object \(u\), there is a unique idempotent
\[
  e_u\in\mathcal S(u,u).
\]
Moreover, for every pair of objects \(u,v\), every
\[
  W\in\mathcal S(u,v),
\]
and every \(X\in\mathcal S(v,u)\) satisfying
\[
  W=WXW,
\]
one has
\[
  WX=e_u,
  \qquad
  XW=e_v,
\]
and
\[
  e_uW=W=We_v.
\]
Thus the \(e_u\) are identity arrows, \(X\) is the inverse of \(W\), and
\(\mathcal S\) is a groupoid.
\end{lemma}

\begin{proof}
Fix an object \(u\) and choose \(P\in\mathcal S(u,u)\).  Since
\(\mathcal S\) is regular, there exists \(X\in\mathcal S(u,u)\) such
that
\[
  P=PXP.
\]
Then
\[
  (PX)^2=(PXP)X=PX,
\]
so \(\mathcal S(u,u)\) contains an idempotent.

Let \(e,f\in\mathcal S(u,u)\) be idempotents.  Since \(e\) is a left
divisor of \(f\), and \(f\) is a right divisor of \(e\), there are
\(A,B\in\mathcal S(u,u)\) such that
\[
  f=eA
  \qquad\text{and}\qquad
  e=Bf.
\]
Therefore
\[
  ef=e(eA)=eA=f
\]
and
\[
  ef=(Bf)f=Bf=e.
\]
Hence \(e=f\).  Denote the unique idempotent in
\(\mathcal S(u,u)\) by \(e_u\).

Now let \(W\in\mathcal S(u,v)\), and let \(X\in\mathcal S(v,u)\)
satisfy \(W=WXW\).  Then
\[
  (WX)^2=(WXW)X=WX
\]
and
\[
  (XW)^2=X(WXW)=XW.
\]
By uniqueness of the local idempotents,
\[
  WX=e_u
  \qquad\text{and}\qquad
  XW=e_v.
\]
Consequently,
\[
  e_uW=(WX)W=W
  \qquad\text{and}\qquad
  We_v=W(XW)=W.
\]
Since every arrow is regular, these equations show that the \(e_u\) are
identity arrows.  In particular,
\[
  XWX=(XW)X=e_vX=X,
\]
so \(X\) is the inverse of \(W\).  Hence \(\mathcal S\) is a groupoid.
\end{proof}

We regard \(\Path_{\mathrm{in}}(\mathcal L)\), introduced in
Corollary~\ref{cor:interval-pair-dictionary}, as a semigroupoid in the
above sense: its objects are the inner vertices of \(\mathcal L\), its
arrows from \(u\) to \(v\) are the elements of
\(\Path(\mathcal L)(u,v)\), and composition is induced by concatenation
of \(1\)-paths.  This is well defined because every \(1\)-path with
inner endpoints is inner, and associativity is inherited from
\(\Path(\mathcal L)\).  We now verify the hypotheses of
Lemma~\ref{lem:local-cancellation-criterion} for this semigroupoid.

\begin{lemma}[Direct regularity]
\label{lem:direct-regularity}
Let \(H\leq F\) be closed, suppose that its action on \((0,1)\) is
minimal and that \(\pi_{\mathrm{ab}}(H)\) has rank two, and let
\(\mathcal L=C(H)\).  For every pair of inner vertices \(u,v\) of
\(\mathcal L\) and every
\[
  W\in\Path(\mathcal L)(u,v),
\]
there exists
\[
  X\in\Path(\mathcal L)(v,u)
\]
such that
\[
  W=WXW.
\]
\end{lemma}

\begin{proof}
By Proposition~\ref{prop:interval-realization}, choose dyadic points
\(0<a<b<1\) such that
\[
  W=\omega[a,b].
\]
Corollary~\ref{cor:endpoint-shrinking} gives dyadic points \(c,d\)
satisfying
\[
  a<c<d<b,
  \qquad
  W=\omega[a,c]=\omega[d,b].
\]
Put
\[
  X=\omega[c,d].
\]
The equality \(\omega[a,c]=W\) shows that \(c\) has type \(v\), while
the equality \(\omega[d,b]=W\) shows that \(d\) has type \(u\).  Thus
\(X\) is a path class from \(v\) to \(u\).  By
Lemma~\ref{lem:cutting},
\[
  W=\omega[a,b]
   =\omega[a,c]\omega[c,d]\omega[d,b]
   =WXW.
\]
The interval configuration is shown in
Figure~\ref{fig:inward-regularity}.
\end{proof}

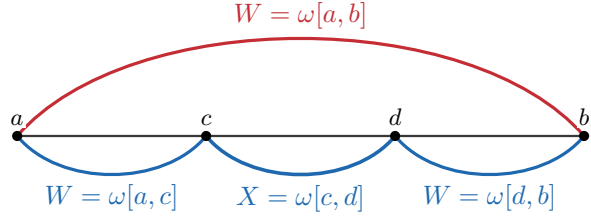
\begin{figure}[htbp]
\centering
\begin{tikzpicture}[x=1.25cm,y=.8cm]
  \coordinate (A) at (0,0);
  \coordinate (C) at (2,0);
  \coordinate (D) at (4,0);
  \coordinate (B) at (6,0);

  \draw[s23edge] (A)--(C)--(D)--(B);
  \draw[s23upper]
    (A) .. controls (1.2,2.15) and (4.8,2.15) ..
    node[s23label,above=2pt] {\(W=\omega[a,b]\)} (B);
  \draw[s23lower]
    (A) .. controls +(.45,-.85) and +(-.45,-.85) ..
    node[s23label,below=2pt] {\(W=\omega[a,c]\)} (C);
  \draw[s23return]
    (C) .. controls +(.45,-.85) and +(-.45,-.85) ..
    node[s23label,below=2pt] {\(X=\omega[c,d]\)} (D);
  \draw[s23lower]
    (D) .. controls +(.45,-.85) and +(-.45,-.85) ..
    node[s23label,below=2pt] {\(W=\omega[d,b]\)} (B);

  \foreach \P/\lab in {A/a,C/c,D/d,B/b} {
    \node[s23vertex] at (\P) {};
    \node[s23tiny,above=3pt] at (\P) {\(\lab\)};
  }
\end{tikzpicture}
\caption{The inward interval decomposition used in
Lemma~\ref{lem:direct-regularity}.  The whole interval \([a,b]\) has
class \(W\), while the consecutive subintervals \([a,c]\), \([c,d]\),
and \([d,b]\) have classes \(W\), \(X\), and \(W\), respectively.}
\label{fig:inward-regularity}
\end{figure}

\begin{lemma}[Local divisibility]
\label{lem:local-divisibility}
Let \(H\leq F\) be closed, suppose that its action on \((0,1)\) is
minimal and that \(\pi_{\mathrm{ab}}(H)\) has rank two, and let
\(\mathcal L=C(H)\).  For every inner vertex \(u\), the set
\[
  \Path(\mathcal L)(u,u)
\]
is nonempty.  Moreover, every arrow in \(\Path(\mathcal L)(u,u)\) is
both a left divisor and a right divisor of every arrow in
\(\Path(\mathcal L)(u,u)\).
\end{lemma}

\begin{proof}
By Proposition~\ref{prop:orbit-core}, choose a dyadic point \(a\) with
\(v(a)=u\).  Since the action of \(H\) is minimal, \(a\) is not fixed
by every element of \(H\).  Choose \(g\in H\) that moves \(a\),
replacing \(g\) by \(g^{-1}\) if necessary so that \(g(a)>a\).  Then
\[
  \omega[a,g(a)]\in\Path(\mathcal L)(u,u),
\]
so this set is nonempty.

Now let
\[
  U,V\in\Path(\mathcal L)(u,u).
\]
By Proposition~\ref{prop:interval-realization}, choose dyadic points
\[
  a<b,\qquad c<d
\]
such that
\[
  U=\omega[a,b],
  \qquad
  V=\omega[c,d].
\]

Since \(v(a)=v(c)=u\), Proposition~\ref{prop:orbit-core} gives
\(h\in H\) with \(h(a)=c\).  Proposition~\ref{prop:transport} then
gives
\[
  U=\omega[c,h(b)].
\]
Applying Corollary~\ref{cor:endpoint-shrinking}, choose a dyadic point
\(x\in(c,d)\) such that
\[
  U=\omega[c,x].
\]
By Lemma~\ref{lem:cutting},
\[
  V=\omega[c,d]
   =\omega[c,x]\omega[x,d]
   =U\omega[x,d].
\]
Thus \(U\) is a left divisor of \(V\).

Similarly, since \(v(b)=v(d)=u\), choose \(k\in H\) with \(k(b)=d\).
Then
\[
  U=\omega[k(a),d].
\]
Corollary~\ref{cor:endpoint-shrinking} gives a dyadic point
\(y\in(c,d)\) such that
\[
  U=\omega[y,d].
\]
Therefore
\[
  V=\omega[c,d]
   =\omega[c,y]\omega[y,d]
   =\omega[c,y]U.
\]
Thus \(U\) is also a right divisor of \(V\).
\end{proof}

\begin{corollary}
\label{cor:minimal-groupoid}
Let \(H\leq F\) be closed, suppose that its action on \((0,1)\) is
minimal and that \(\pi_{\mathrm{ab}}(H)\) has rank two, and let
\[
  \mathcal L=C(H).
\]
Then \(\Path_{\mathrm{in}}(\mathcal L)\) is a groupoid.  More
explicitly, for every inner vertex \(u\), there is a unique idempotent
\[
  e_u\in\Path(\mathcal L)(u,u),
\]
and these idempotents are the identity arrows.  For every pair of inner
vertices \(u,v\) and every
\[
  W\in\Path(\mathcal L)(u,v),
\]
there exists
\[
  X\in\Path(\mathcal L)(v,u)
\]
which is the inverse of \(W\); explicitly,
\[
  WX=e_u,
  \qquad
  XW=e_v.
\]
\end{corollary}

\begin{proof}
Lemma~\ref{lem:direct-regularity} shows that
\(\Path_{\mathrm{in}}(\mathcal L)\) is regular, and
Lemma~\ref{lem:local-divisibility} verifies the local nonemptiness and
divisibility hypotheses of
Lemma~\ref{lem:local-cancellation-criterion}.  The result follows from
Lemma~\ref{lem:local-cancellation-criterion}.
\end{proof}

\subsection{The exact number of inner path classes}
\label{subsec:exact-inner-path-count}
Let \(H\leq F\) be closed and act minimally on \((0,1)\), and suppose
that
\[
  A=\pi_{\mathrm{ab}}(H)
\]
has rank two.

Let \(d_L\) and \(d_R\) be the positive generators of the two
coordinate projections of \(A\).  Then
\[
  A\leq d_L\mathbb Z\times d_R\mathbb Z.
\]
Set
\[
  q=\bigl[d_L\mathbb Z\times d_R\mathbb Z:A\bigr].
\]
Since \(A\) has rank two, \(q\) is finite.  Moreover, the quotient
\[
  \bigl(d_L\mathbb Z\times d_R\mathbb Z\bigr)/A
\]
is cyclic and is generated by the coset of \((d_L,0)\).  Indeed, since
the second coordinate projection of \(A\) is \(d_R\mathbb Z\), there is
an integer \(k\) such that
\[
  (kd_L,d_R)\in A.
\]
Hence
\[
  (0,d_R)+A=-k\bigl((d_L,0)+A\bigr).
\]
Consequently, \(q\) is the least positive integer \(t\) such that
\[
  (td_L,0)\in A.
\]
In particular,
\[
  [\mathbb Z^2:A]=d_Ld_Rq.
\]

\begin{lemma}[The order of \(\Lambda_L\)]
\label{lem:left-boundary-order}
Let \(H\leq F\) be closed and act minimally on \((0,1)\), and suppose
that \(A=\pi_{\mathrm{ab}}(H)\) has rank two.  Let
\(\mathcal L=C(H)\), and retain the notation \(d_L\), \(v_L\), and
\(\Lambda_L\) from Subsection~\ref{subsec:boundary-periods}.  Then, for
every positive integer \(t\),
\[
  \Lambda_L^t=e_{v_L}
  \quad\Longleftrightarrow\quad
  (td_L,0)\in A.
\]
Thus the order of \(\Lambda_L\) in the group
\(\Path(\mathcal L)(v_L,v_L)\) is \(q\).
\end{lemma}

\begin{proof}
Suppose first that \((td_L,0)\in A\).  By
Proposition~\ref{prop:endpoint-relations}, there are a nonempty
\(1\)-path \(Z\) from \(v_L\) to \(v_R\) and a positive diagram from
\(p_H\) to \(\ell Zr\) such that
\[
  \Lambda_L^t[Z]=[Z].
\]
Since \(\Path_{\mathrm{in}}(\mathcal L)\) is a groupoid by
Corollary~\ref{cor:minimal-groupoid}, we may cancel \([Z]\) on the
right.  Hence
\[
  \Lambda_L^t=e_{v_L}.
\]

Conversely, Corollary~\ref{cor:boundary-relations}\textup{(2)} provides a
nonempty
\(1\)-path \(Y\) from \(v_L\) to \(v_R\) for which there is a positive
diagram from \(p_H\) to \(\ell Yr\).  If
\(\Lambda_L^t=e_{v_L}\), then
\[
  \Lambda_L^t[Y]=[Y].
\]
The converse implication in Proposition~\ref{prop:endpoint-relations}
now gives \((td_L,0)\in A\).  The final assertion follows from the
characterization of \(q\) above.
\end{proof}

\begin{proposition}[Exact path count]
\label{prop:exact-inner-path-count}
Let \(H\leq F\) be closed and act minimally on \((0,1)\), and suppose
that \(A=\pi_{\mathrm{ab}}(H)\) has rank two.  Let
\(\mathcal L=C(H)\), and set
\[
  q=\bigl[d_L\mathbb Z\times d_R\mathbb Z:A\bigr].
\]
Then, for every pair of inner vertices \(u,v\) of \(\mathcal L\),
\[
  \bigl|\Path(\mathcal L)(u,v)\bigr|=q.
\]
\end{proposition}

\begin{proof}
We first count the path classes from \(v_L\) to an arbitrary inner
vertex \(v\).  By Proposition~\ref{prop:orbit-core}, there are dyadic
points \(a,c\in(0,1)\) of types \(v_L\) and \(v\), respectively.  Since
the action of \(H\) is minimal, the \(H\)-orbit of \(c\) is dense.  We
may therefore choose a point \(b\) in this orbit such that \(b>a\).
Then \(b\) has type \(v\), so
\[
  \omega[a,b]\in\Path(\mathcal L)(v_L,v).
\]
Put
\[
  P=\omega[a,b].
\]

We use the convention \(\Lambda_L^0=e_{v_L}\).  The \(q\) classes
\[
  P,\ \Lambda_LP,\ \ldots,\ \Lambda_L^{q-1}P
\]
are distinct.  Indeed, if
\[
  \Lambda_L^iP=\Lambda_L^jP
\]
for some \(0\leq i<j<q\), then cancellation in the groupoid gives
\[
  \Lambda_L^{j-i}=e_{v_L},
\]
contrary to Lemma~\ref{lem:left-boundary-order}.

Now let \(Q\in\Path(\mathcal L)(v_L,v)\).  Represent \(P\) and \(Q\) by
nonempty inner \(1\)-paths from \(v_L\) to \(v\).
By Lemma~\ref{lem:boundary-comparison}, there are positive integers
\(m,n\) such that
\[
  \Lambda_L^{m}Q=\Lambda_L^{n}P.
\]
Since \(\Lambda_L\) is invertible and has order \(q\), it follows that
\[
  Q=\Lambda_L^jP
\]
for some \(j\in\{0,\ldots,q-1\}\).  Thus
\[
  \bigl|\Path(\mathcal L)(v_L,v)\bigr|=q.
\]

Finally, let \(u,v\) be arbitrary inner vertices.  By the first
paragraph of this proof, applied with \(u\) in place of \(v\), the set
\(\Path(\mathcal L)(v_L,u)\) is nonempty.  Choose
\[
  C\in\Path(\mathcal L)(v_L,u).
\]
By Corollary~\ref{cor:minimal-groupoid}, \(C\) has an inverse
\(C^{-1}\).  Left multiplication by \(C\) defines a bijection
\[
  \Path(\mathcal L)(u,v)\longrightarrow
  \Path(\mathcal L)(v_L,v),
  \qquad
  W\longmapsto CW,
\]
whose inverse is \(Z\mapsto C^{-1}Z\).  Hence
\[
  \bigl|\Path(\mathcal L)(u,v)\bigr|
  =
  \bigl|\Path(\mathcal L)(v_L,v)\bigr|
  =q.
\]
\end{proof}

\begin{corollary}[The finitely generated minimal rank-two case]
\label{cor:minimal-rank-two-finfty}
Let \(H\leq F\) be a closed finitely generated subgroup.  Suppose that
\(H\) acts minimally on \((0,1)\) and that
\(\pi_{\mathrm{ab}}(H)\) has rank two.  Let
\(\mathcal L=C(H)\).  Then \(\Path(\mathcal L)\) is finite, and \(H\)
is of type \(F_\infty\).
\end{corollary}

\begin{proof}
By Theorem~\ref{thm:core-closure}\textup{(ii)}, the core
\(\mathcal L=C(H)\) is finite.
Proposition~\ref{prop:exact-inner-path-count} shows that there are only
finitely many path classes between any two inner vertices of
\(\mathcal L\).  Lemma~\ref{lem:inner-reduction}\textup{(2)} therefore
implies that \(\Path(\mathcal L)\) is finite.  Finally,
Theorem~\ref{thm:core-closure}\textup{(i)} identifies \(H\) with the
diagram group \(D(\mathcal L,p_H)\), and the result follows from
Theorem~\ref{thm:finite-path-semigroupoid}.
\end{proof}

\subsection{Characterizations}
\label{subsec:characterizations}

We call the inner path semigroupoid
\(\Path_{\mathrm{in}}(\mathcal L)\) \emph{connected} if
\[
  \Path(\mathcal L)(u,v)\neq\varnothing
\]
for every ordered pair of inner vertices \(u,v\).  We call it
\emph{locally finite} if every set \(\Path(\mathcal L)(u,v)\) is
finite.  In particular, a locally finite groupoid with finitely many
objects is finite.

\begin{lemma}[Connectedness and fixed points]
\label{lem:connected-fixed-points}
Let \(H\leq F\) be closed, and suppose that \(\mathcal L=C(H)\) is full.
Then \(\Path_{\mathrm{in}}(\mathcal L)\) is connected if and only if
\(H\) has no fixed point in \((0,1)\).
\end{lemma}

\begin{proof}
Suppose first that \(H\) has no fixed point in \((0,1)\).  Then
\((0,1)\) is an orbital of \(H\).  Consequently, for every
\(x\in(0,1)\),
\[
  \inf Hx=0,
  \qquad
  \sup Hx=1.
\]

Let \(u,v\) be inner vertices.  By Proposition~\ref{prop:orbit-core},
choose \(a,b\in\D\) such that
\[
  v(a)=u,
  \qquad
  v(b)=v.
\]
Since \(\sup Hb=1\), there exists \(h\in H\) such that \(h(b)>a\).
Therefore
\[
  \omega[a,h(b)]\in\Path(\mathcal L)(u,v),
\]
and \(\Path_{\mathrm{in}}(\mathcal L)\) is connected.

Conversely, suppose that \(c\in(0,1)\) is fixed by \(H\).  Choose dyadic
points
\[
  a<c<b,
\]
and put \(u=v(a)\) and \(v=v(b)\).  If the inner path semigroupoid were
connected, there would be a class in \(\Path(\mathcal L)(v,u)\).  By
Proposition~\ref{prop:interval-realization}, this class would have the
form
\[
  \omega[x,y]
\]
for some dyadic \(x<y\) satisfying
\[
  v(x)=v,
  \qquad
  v(y)=u.
\]
Proposition~\ref{prop:orbit-core} would then imply that \(x\) belongs
to the orbit of \(b\) and \(y\) belongs to the orbit of \(a\).  Thus
\(x>c>y\), contradicting \(x<y\).
\end{proof}

\begin{theorem}[Characterization of the minimal rank-two case]
\label{thm:minimal-rank-two-characterization}
Let \(H\leq F\) be closed, put
\[
  \mathcal L=C(H),
  \qquad
  A=\pi_{\mathrm{ab}}(H),
\]
and suppose that \(A\) contains an element \((m,n)\) with \(mn\neq0\).
The following conditions are equivalent.
\begin{enumerate}[label=\textup{(\arabic*)}]
\item The action of \(H\) on \((0,1)\) is minimal and \(A\) has rank
      two.
\item The core \(\mathcal L\) is full, and
      \(\Path_{\mathrm{in}}(\mathcal L)\) is a connected, locally
      finite groupoid.
\end{enumerate}
If \(\mathcal L\) is finite, these conditions are also equivalent to:
\begin{enumerate}[label=\textup{(\arabic*)},start=3]
\item The action of \(H\) on \((0,1)\) is minimal, and \(H\) has
      finitely many orbits on
      \[
        \D_{<}^{2}=\{(a,b)\in\D^2:a<b\}.
      \]
\end{enumerate}
\end{theorem}

\begin{proof}
Assume \textup{(1)}.  Corollary~\ref{cor:minimal-full} shows that
\(\mathcal L\) is full, and Corollary~\ref{cor:minimal-groupoid} shows
that \(\Path_{\mathrm{in}}(\mathcal L)\) is a groupoid.  Since a
minimal action has no fixed point,
Lemma~\ref{lem:connected-fixed-points} gives connectedness.
Proposition~\ref{prop:exact-inner-path-count} shows that every set
\[
  \Path(\mathcal L)(u,v)
\]
is finite.  Thus \textup{(2)} holds.

Conversely, assume \textup{(2)}.  Let \(x\in(0,1)\), and let
\(U\subseteq(0,1)\) be a nonempty open interval.  Choose closed
intervals \(I=[a,b]\) and \(J=[c,d]\), with dyadic endpoints, such that
\[
  x\in(a,b),
  \qquad
  J\subseteq U.
\]
Put
\[
  W=\omega(I)\in\Path(\mathcal L)(u,v),
  \qquad
  V=\omega(J)\in\Path(\mathcal L)(u',v').
\]
Connectedness provides an arrow
\[
  C\in\Path(\mathcal L)(u',u).
\]
Let \(X\in\Path(\mathcal L)(v,u')\) be the inverse of \(CW\), and set
\(D=XV\).  Then
\[
  CWD=(CW)XV=e_{u'}V=V.
\]
Corollary~\ref{cor:interval-placement} gives an element of \(H\)
mapping \(I\) onto a closed subinterval of \(J\).  In particular, this
element maps \(x\) into \(U\).  Hence the action is minimal.

The hypothesis on \(A\) implies that its two coordinate directions are
nonzero.  Thus \(A\) has rank one or two.  If \(A\) had rank one,
Corollary~\ref{cor:rank-one-infinite-paths}
would imply that
\[
  \Path(\mathcal L)(v_L,v_R)
\]
is infinite, contrary to local finiteness.  Therefore \(A\) has rank
two, proving \textup{(1)}.

Now suppose that \(\mathcal L\) is finite.  Since a locally finite
groupoid with finitely many objects is finite,
Corollary~\ref{cor:interval-pair-dictionary} gives the equivalence with
\textup{(3)}.  Conversely, if \textup{(3)} holds, that corollary shows
that \(\Path_{\mathrm{in}}(\mathcal L)\) is finite.
Corollary~\ref{cor:rank-one-infinite-paths} again excludes rank one, so
\(A\) has rank two.
\end{proof}

\begin{definition}
\label{def:relative-order-transitivity}
Let \(k\geq2\).  The action of \(H\) on \(\D\) is \emph{relatively
order-\(k\)-transitive} if, whenever
\[
  a_1<\cdots<a_k,
  \qquad
  b_1<\cdots<b_k
\]
are dyadic points such that \(a_i\) and \(b_i\) belong to the same
\(H\)-orbit for every \(i\), there exists \(h\in H\) such that
\[
  h(a_i)=b_i
  \qquad(1\leq i\leq k).
\]
If \(H\) is transitive on \(\D\), this is the usual
order-\(k\)-transitivity of its action on \(\D\).
\end{definition}

\begin{proposition}[Thinness and relative order transitivity]
\label{prop:thin-relative-order}
Let \(H\leq F\) be closed, and suppose that \(\mathcal L=C(H)\) is full.
The following conditions are equivalent.
\begin{enumerate}[label=\textup{(\arabic*)}]
\item The path semigroupoid \(\Path(\mathcal L)\) is thin.
\item The action of \(H\) on \(\D\) is relatively
      order-\(2\)-transitive.
\item The action of \(H\) on \(\D\) is relatively
      order-\(k\)-transitive for every \(k\geq2\).
\end{enumerate}
\end{proposition}

\begin{proof}
Corollary~\ref{cor:interval-pair-dictionary} and
Proposition~\ref{prop:orbit-core} show that relative
order-\(2\)-transitivity is equivalent to having at most one path
class between every ordered pair of inner vertices.  By
Lemma~\ref{lem:inner-reduction}\textup{(1)}, this is equivalent to
thinness of \(\Path(\mathcal L)\).  Thus \textup{(1)} and \textup{(2)}
are equivalent.

Assume \textup{(1)}, and let
\[
  a_1<\cdots<a_k,
  \qquad
  b_1<\cdots<b_k
\]
satisfy the assumptions of Definition~\ref{def:relative-order-transitivity}.
Proposition~\ref{prop:orbit-core} gives
\[
  v(a_i)=v(b_i)
  \qquad(1\leq i\leq k).
\]
For every \(i<k\), the two interval classes
\[
  \omega[a_i,a_{i+1}]
  \quad\text{and}\quad
  \omega[b_i,b_{i+1}]
\]
therefore have the same initial and terminal vertices.  Thinness gives
\[
  \omega[a_i,a_{i+1}]
  =
  \omega[b_i,b_{i+1}].
\]
Since \(H\) is closed, Proposition~\ref{prop:tuple-transport} gives a
single element \(h\in H\) such that
\[
  h(a_i)=b_i
  \qquad(1\leq i\leq k).
\]
Thus \textup{(1)} implies \textup{(3)}.  Taking \(k=2\) shows that
\textup{(3)} implies \textup{(2)}.
\end{proof}

\begin{theorem}[The rectangular characterization]
\label{thm:rectangular-characterization}
Let \(H\leq F\) be closed, put
\[
  \mathcal L=C(H),
  \qquad
  A=\pi_{\mathrm{ab}}(H),
\]
and suppose that \(A\) contains an element \((m,n)\) with \(mn\neq0\).
The following conditions are equivalent.
\begin{enumerate}[label=\textup{(\arabic*)}]
\item The action of \(H\) on \((0,1)\) is minimal, and
      \[
        A=m\mathbb Z\times n\mathbb Z
      \]
      for some positive integers \(m,n\).
\item The core \(\mathcal L\) is full and
      \[
        \bigl|\Path(\mathcal L)(u,v)\bigr|=1
      \]
      for every ordered pair of inner vertices \(u,v\).
\item The core \(\mathcal L\) is full, the inner path semigroupoid
      \(\Path_{\mathrm{in}}(\mathcal L)\) is connected, and
      \(\Path(\mathcal L)\) is thin.
\item The core \(\mathcal L\) is full, \(H\) has no fixed point in
      \((0,1)\), and the action on \(\D\) is relatively
      order-\(2\)-transitive.
\item The core \(\mathcal L\) is full, \(H\) has no fixed point in
      \((0,1)\), and the action on \(\D\) is relatively
      order-\(k\)-transitive for every \(k\geq2\).
\end{enumerate}
\end{theorem}

\begin{proof}
Assume \textup{(1)}.  Corollary~\ref{cor:minimal-full} shows that
\(\mathcal L\) is full.  Moreover, \(d_L=m\), \(d_R=n\), and
\[
  \bigl[d_L\mathbb Z\times d_R\mathbb Z:A\bigr]=1.
\]
Proposition~\ref{prop:exact-inner-path-count} therefore gives
\[
  \bigl|\Path(\mathcal L)(u,v)\bigr|=1
\]
for every ordered pair of inner vertices.  Thus \textup{(1)} implies
\textup{(2)}.

Suppose \textup{(2)} holds.  Connectedness is immediate, and
Lemma~\ref{lem:inner-reduction}\textup{(1)} shows that the uniqueness
of all inner path classes implies thinness of \(\Path(\mathcal L)\).
Hence \textup{(2)} implies \textup{(3)}.  Conversely, connectedness
supplies at least one inner path class between every ordered pair,
while thinness supplies at most one.  Thus \textup{(3)} implies
\textup{(2)}.

Under \textup{(2)}, uniqueness makes the unique local loops identity
arrows and the unique arrows in the reverse direction inverses.  Indeed,
uniqueness gives
\[
  e_uW=W=We_v,
  \qquad
  WX=e_u,
  \qquad
  XW=e_v.
\]
Thus \(\Path_{\mathrm{in}}(\mathcal L)\) is a connected, locally finite
groupoid.  Theorem~\ref{thm:minimal-rank-two-characterization} shows
that the action is minimal and that \(A\) has rank two.
Proposition~\ref{prop:exact-inner-path-count} now gives
\[
  1=\bigl[d_L\mathbb Z\times d_R\mathbb Z:A\bigr].
\]
Consequently,
\[
  A=d_L\mathbb Z\times d_R\mathbb Z,
\]
which proves \textup{(1)}.

Finally, Lemma~\ref{lem:connected-fixed-points} identifies connectedness
with the absence of fixed points, and
Proposition~\ref{prop:thin-relative-order} identifies thinness with
each of the two relative order-transitivity conditions.  Hence
\textup{(3)}, \textup{(4)}, and \textup{(5)} are equivalent.
\end{proof}

\medskip\noindent\textbf{Remark.}
In either Theorem~\ref{thm:minimal-rank-two-characterization} or
Theorem~\ref{thm:rectangular-characterization}, the hypothesis that
\(A\) contain an element \((m,n)\) with \(mn\neq0\) may instead be
replaced by the assumption that \(\mathcal L\) is finite.  Indeed, each
condition in the two theorems includes or implies fullness, and the
required element exists whenever the core is finite and full.

The hypothesis cannot be omitted without any replacement.  The
subgroup \([F,F]\) is closed, acts minimally on \((0,1)\), and is
order-\(k\)-transitive on \(\D\) for every \(k\geq2\), so that
conditions \textup{(2)}--\textup{(5)} of
Theorem~\ref{thm:rectangular-characterization} hold for it; but
\[
  \pi_{\mathrm{ab}}([F,F])=0,
\]
so \textup{(1)} fails.

\section{Rational \texorpdfstring{\(H\)}{H}-isolated intervals}
\label{sec:rational-isolated-intervals}

From now on, let \(H\leq F\) be closed and suppose that
\[
  \mathcal L=C(H)
\]
is finite and full.  By Theorem~\ref{thm:core-closure}\textup{(ii)},
\(H\) is finitely generated.  Theorem~\ref{thm:rank-one-not-FP2}
treats the rank-one case, while
Corollary~\ref{cor:minimal-rank-two-finfty} treats the case in which
\(\pi_{\mathrm{ab}}(H)\) has rank two and the action of \(H\) on
\((0,1)\) is minimal.  We now prepare for the remaining rank-two
nonminimal case.

\subsection{The normalized restriction}
\label{subsec:normalized-restriction}

\begin{definition}
\label{def:rational-isolated}
Let \(H\leq F\), and let
\[
  I=(a,b),
  \qquad
  0\leq a<b\leq1.
\]
We call \(I\) an \emph{\(H\)-isolated interval} if
\[
  h(I)\cap I\neq\varnothing
  \quad\Longrightarrow\quad
  h(I)=I
\]
for every \(h\in H\).  If, in addition, \(a\) and \(b\) are rational,
we call \(I\) a \emph{rational \(H\)-isolated interval}.
\end{definition}

An interval \(I\) is \(H\)-isolated if and only if any two intervals in
\[
  \{h(I):h\in H\}
\]
are either equal or disjoint.  Write
\[
  \operatorname{Stab}_H(I)=\{h\in H:h(I)=I\}
\]
for the setwise stabilizer of \(I\) in \(H\).  Moreover, if \(x,y\in I\)
and \(h(x)=y\) for some \(h\in H\), then
\(h\in\operatorname{Stab}_H(I)\).  Since the elements of \(F\) are
increasing, every element of \(\operatorname{Stab}_H(I)\) fixes both
endpoints of \(I\).  In the next section, we shall see that if
\(H\leq F\) is closed, \(C(H)\) is finite and full,
\(\pi_{\mathrm{ab}}(H)\) has rank two, and the action of \(H\) on
\((0,1)\) is not minimal, then \(H\) has a proper rational
\(H\)-isolated interval.

We now associate a coordinate map with every rational interval
\(I=(a,b)\).  Choose binary expansions
\[
  a=.\!\alpha,
  \qquad
  b=.\!\beta
\]
as follows: if \(a\) is dyadic, choose its expansion ending in
\(0^\infty\), and if \(b\) is dyadic, choose its expansion ending in
\(1^\infty\).  For nondyadic endpoints, use their unique binary
expansions.

Let \(u\) be the longest common prefix of \(\alpha\) and \(\beta\), and
write
\[
  \alpha=u0\xi_1\xi_2\cdots,
  \qquad
  \beta=u1\eta_1\eta_2\cdots .
\]
Let
\[
  i_1<i_2<\cdots
\]
be the positions for which \(\xi_{i_n}=0\), and let
\[
  j_1<j_2<\cdots
\]
be the positions for which \(\eta_{j_n}=1\).  Define
\[
  s_n=u0\xi_1\cdots\xi_{i_n-1}1,
  \qquad
  t_n=u1\eta_1\cdots\eta_{j_n-1}0
  \qquad(n\geq1).
\]
The intervals
\[
  [s_n],\ [t_n],
  \qquad n\geq1,
\]
are precisely the maximal standard dyadic intervals contained in
\(I\).

Define
\[
  \nu_I:(0,1)\longrightarrow I
\]
by
\[
  \nu_I(.0^n1\zeta)=.s_n\zeta,
  \qquad
  \nu_I(.1^n0\zeta)=.t_n\zeta
\]
for every \(n\geq1\) and every infinite binary word \(\zeta\).  Thus
\(\nu_I\) maps each interval \([0^n1]\) linearly onto \([s_n]\), and
each interval \([1^n0]\) linearly onto \([t_n]\).  The correspondence
is
\[
\begin{array}{ccccccccc}
\cdots &[0^31]&[0^21]&[01]&[10]&[1^20]&[1^30]&\cdots\\
&\downarrow&\downarrow&\downarrow&\downarrow&\downarrow&\downarrow\\
\cdots &[s_3]&[s_2]&[s_1]&[t_1]&[t_2]&[t_3]&\cdots .
\end{array}
\]
These formulas define an increasing homeomorphism, which we call the
\emph{canonical chart} associated with \(I\).  We extend it
continuously by
\[
  \nu_I(0)=a,
  \qquad
  \nu_I(1)=b.
\]
Since \(\nu_I\) only replaces a finite prefix and leaves the remaining
binary tail unchanged, it maps \(\D\) bijectively onto \(\D\cap I\),
and it maps rational points bijectively onto the rational points of
\(I\).

\begin{definition}
\label{def:normalized-restriction}
Let \(H\leq F\), and let \(I\) be a rational \(H\)-isolated interval.
The \emph{normalized restriction of \(H\) to \(I\)} is
\[
  K=K_I(H)
  =\nu_I
  \bigl(\operatorname{Stab}_H(I)|_I\bigr)
  \nu_I^{-1}.
\]
Thus the map \(k\mapsto\nu_I^{-1}k\nu_I\) identifies \(K\) with the
restriction of \(\operatorname{Stab}_H(I)\) to \(I\).
\end{definition}

\begin{lemma}
\label{lem:normalized-restriction}
Let \(H\leq F\) be closed, let \(I=(a,b)\) be a rational
\(H\)-isolated interval, and let \(K=K_I(H)\).  Then \(K\) is a closed
subgroup of \(F\).  Let \(\sigma_a\) and
\(\sigma_b\) be least eventual periods of the chosen expansions
\(\alpha\) and \(\beta\) of \(a\) and \(b\), respectively.  Thus
\(\sigma_a=0\) if \(a\) is dyadic, and \(\sigma_b=1\) if \(b\) is
dyadic.  Put
\[
  L_a=|\sigma_a|,
  \qquad
  z_a=\text{the number of zeros in }\sigma_a,
\]
and
\[
  L_b=|\sigma_b|,
  \qquad
  o_b=\text{the number of ones in }\sigma_b.
\]
If \(h\in\operatorname{Stab}_H(I)\) and
\[
  k=\nu_I(h|_I)\nu_I^{-1},
\]
then
\[
  \log_2 k'(0^+)
  =
  \frac{z_a}{L_a}\log_2 h'(a^+)
\]
and
\[
  \log_2 k'(1^-)
  =
  \frac{o_b}{L_b}\log_2 h'(b^-).
\]

If, in addition, \(C(H)\) is finite and full, then \(K\) contains an
element with nontrivial slope at \(0\) and an element with nontrivial
slope at \(1\).
\end{lemma}

\begin{proof}
Let \(h\in\operatorname{Stab}_H(I)\), and consider its action
immediately to the right of \(a\).  If \(a\) is nondyadic, take the
branch \(y\to y'\) whose domain contains \(a\).  If \(a\) is dyadic,
take the branch whose domain interval has \(a\) as its left endpoint;
this agrees with the expansion \(\alpha\) ending in \(0^\infty\).
Refining the branch if necessary, we may assume that \(y\) and \(y'\)
extend beyond the preperiod of \(\alpha\).

Since \(h(a)=a\), there is an infinite word \(\gamma\) such that
\[
  \alpha=y\gamma=y'\gamma.
\]
Therefore
\[
  \log_2 h'(a^+)=|y|-|y'|.
\]
The equality of the two tails implies that
\[
  |y|-|y'|=jL_a
\]
for some \(j\in\mathbb Z\).

Each copy of \(\sigma_a\) contains \(z_a\) zeros.  Since the intervals
\([s_n]\) are indexed by the successive zeros of \(\alpha\), for every
sufficiently large \(n\), the map \(h\) sends \([s_n]\) linearly onto
\[
  [s_{n-jz_a}].
\]
Hence \(k\) sends \([0^n1]\) linearly onto
\[
  [0^{\,n-jz_a}1]
\]
for all sufficiently large \(n\).  It follows that, for some
sufficiently large \(N\), \(k\) has the branch
\[
  0^N\longrightarrow0^{\,N-jz_a}.
\]
Consequently,
\[
  \log_2 k'(0^+)
  =
  jz_a
  =
  \frac{z_a}{L_a}\log_2 h'(a^+).
\]

The proof at \(b\) is symmetric.  If \(b\) is dyadic, we use the branch
whose domain interval has \(b\) as its right endpoint, in accordance
with the expansion \(\beta\) ending in \(1^\infty\).  Counting the
successive ones in \(\beta\) gives
\[
  \log_2 k'(1^-)
  =
  \frac{o_b}{L_b}\log_2 h'(b^-).
\]

The branches obtained at the two endpoints show that \(k\) is affine
on neighborhoods of \(0\) and \(1\).  Between these neighborhoods,
only finitely many intervals used in the definition of \(\nu_I\)
occur.  Hence \(k\) has finitely many dyadic breakpoints and belongs
to \(F\).

We now prove that \(K\) is closed.  Let \(f\in F\) be piecewise \(K\)
on a dyadic subdivision
\[
  0=x_0<x_1<\cdots<x_m=1.
\]
Choose \(k_i\in K\) agreeing with \(f\) on
\([x_{i-1},x_i]\), and choose
\[
  h_i\in\operatorname{Stab}_H(I)
\]
whose normalized restriction is \(k_i\).  The points
\[
  \nu_I(x_1),\ldots,\nu_I(x_{m-1})
\]
are dyadic.  Paste the \(h_i\) at these points, using \(h_1\) on all of
\([0,\nu_I(x_1)]\) and \(h_m\) on all of \([\nu_I(x_{m-1}),1]\).  The
resulting function
belongs to \(F\), is piecewise \(H\), and therefore belongs to \(H\).
Its restriction to \(I\) is
\[
  \nu_I^{-1}f\nu_I,
\]
so it stabilizes \(I\).  Hence \(f\in K\), proving that \(K\) is
closed.

Finally, suppose that \(\mathcal L=C(H)\) is finite and full.  For the
left endpoint, recall that \(\xi_{i_n}=0\), and write
\[
  c_n=u0\xi_1\cdots\xi_{i_n-1},
  \qquad
  s_n=c_n1,
  \qquad
  a=.c_n0\tau_n.
\]
Past the preperiod of \(\alpha\), the tails \(\tau_n\) take only
finitely many values, and \(E(\mathcal L)\) is finite.  Hence there
exist \(m<n\), both beyond the preperiod, such that
\[
  c_m^+=c_n^+
  \qquad\text{and}\qquad
  \tau_m=\tau_n.
\]
By Lemma~\ref{lem:branch-criterion}, some \(h\in H\) has the branch
\[
  c_m\longrightarrow c_n.
\]
The equality \(\tau_m=\tau_n\) implies that \(h(a)=a\), while expansion
to the right child gives the branch
\[
  s_m\longrightarrow s_n.
\]
Since \([s_m]\) and \([s_n]\) are contained in \(I\), isolation implies
that \(h\in\operatorname{Stab}_H(I)\).  Moreover,
\[
  \log_2 h'(a^+)=|c_m|-|c_n|\neq0.
\]
The first formula of the lemma shows that the corresponding element of
\(K\) has nontrivial slope at \(0\).

The symmetric argument at \(b\) produces an element of \(K\) with
nontrivial slope at \(1\).
\end{proof}

For any core \(\mathcal M\), let
\[
  E_{\mathrm{in}}(\mathcal M)
\]
denote its set of inner edges.  For a rational \(H\)-isolated interval
\(I\), define
\[
  E_H(I)=
  \left\{
    w^+\in E(\mathcal L):
    [w]\subseteq I
  \right\}.
\]
Every edge in \(E_H(I)\) is inner by
Lemma~\ref{lem:boundary-incidence}.

If \(w\) is a finite binary word containing both digits \(0\) and
\(1\), it has a unique expression of one of the forms
\[
  w=0^n1z
  \qquad\text{or}\qquad
  w=1^n0z.
\]
Define the \emph{physical lift} of \(w\) by
\[
  \widehat w=
  \begin{cases}
    s_nz,&w=0^n1z,\\
    t_nz,&w=1^n0z.
  \end{cases}
\]
Then
\[
  \nu_I([w])=[\widehat w]\subseteq I.
\]

\begin{proposition}
\label{prop:normalized-core}
Let \(H\leq F\) be closed, suppose that \(\mathcal L=C(H)\) is finite
and full, and let \(I\) be a rational \(H\)-isolated interval.  Let
\(K=K_I(H)\).  Then \(C(K)\) is finite and full, and the canonical
chart induces a bijection
\[
  E_{\mathrm{in}}\bigl(C(K)\bigr)
  \longrightarrow
  E_H(I).
\]
More precisely, writing \(w_K^+\) for the state reached by \(w\) in
\(C(K)\) and \(v_H^+\) for the state reached by \(v\) in
\(\mathcal L\), this bijection is induced by
\[
  w_K^+\longmapsto\widehat w{}_H^+
\]
for words \(w\) containing both digits \(0\) and \(1\).
\end{proposition}

\begin{proof}
We first prove that \(C(K)\) is full.  Let \(x,y\in\D\).  Since
\(\nu_I\) maps \(\D\) bijectively onto \(\D\cap I\),
\[
  y\in Kx
  \quad\Longleftrightarrow\quad
  \nu_I(y)\in H\nu_I(x).
\]
The forward implication is immediate.  Conversely, if an element of
\(H\) maps \(\nu_I(x)\) to \(\nu_I(y)\), then its image of \(I\) meets
\(I\); isolation forces it to stabilize \(I\), and its normalized
restriction maps \(x\) to \(y\).

Since \(\mathcal L\) is finite and full,
Proposition~\ref{prop:finite-full-core} shows that \(H\) has finitely
many orbits on \(\D\).  Hence \(K\) also has finitely many orbits on
\(\D\).  If \(C(K)\) had a leaf, Lemma~\ref{lem:leaf-intervals} would
provide an open interval containing infinitely many dyadic points
belonging to distinct \(K\)-orbits.  Therefore \(C(K)\) is full.

We claim that, whenever \(w\) and \(z\) each contain both digits,
\[
  w_K^+=z_K^+
  \quad\Longleftrightarrow\quad
  \widehat w_H^+=\widehat z_H^+.
\]
Suppose first that \(w_K^+=z_K^+\).  Since \(K\) is closed,
Lemma~\ref{lem:branch-criterion} gives \(k\in K\) with the branch
\(w\to z\).  Lifting \(k\) to an element of
\(\operatorname{Stab}_H(I)\) gives the branch
\(\widehat w\to\widehat z\), so
Lemma~\ref{lem:branch-criterion} gives
\(\widehat w_H^+=\widehat z_H^+\).

Conversely, suppose that \(\widehat w_H^+=\widehat z_H^+\).  By
Lemma~\ref{lem:branch-criterion}, some \(h\in H\) has the branch
\(\widehat w\to\widehat z\).  Both associated standard dyadic
intervals are contained in \(I\), so \(h(I)\cap I\neq\varnothing\) and
\(h\) stabilizes \(I\).  Its normalized restriction belongs to \(K\)
and has the branch \(w\to z\).  Another application of
Lemma~\ref{lem:branch-criterion} gives \(w_K^+=z_K^+\).  This proves
the claim.

Every standard dyadic interval contained in \(I\) lies in one of the
maximal intervals \([s_n]\) or \([t_n]\).  The claim therefore gives
the asserted bijection; in particular \(C(K)\) has only finitely many
inner edges.

It remains to prove that \(C(K)\) is finite.  By
Lemma~\ref{lem:normalized-restriction}, \(K\) contains an element with
nontrivial slope at \(0\).  Such an element has a branch
\(0^m\to0^n\) with \(m\neq n\), so
Lemma~\ref{lem:branch-criterion} gives \((0^m)_K^+=(0^n)_K^+\), and
taking left children repeatedly shows that the sequence
\((0^r)_K^+\), \(r\geq1\), is eventually periodic.  Thus only finitely
many states occur along the left boundary ray, and symmetrically along
the right one.  Since every state of \(C(K)\) is reached by the empty
word, by a word \(0^n\), by a word \(1^n\), or by a word containing
both digits, \(C(K)\) is finite.
\end{proof}

\begin{lemma}[The endpoint invariants of the normalized restriction]
\label{lem:normalized-endpoint-invariants}
Let \(H\leq F\) be closed, suppose that \(\mathcal L=C(H)\) is finite
and full, let \(I\) be a rational \(H\)-isolated interval and let
\(K=K_I(H)\).  Write
\[
  \alpha=\mu_a\sigma_a^\infty,
  \qquad
  \beta=\mu_b\sigma_b^\infty,
\]
where \(\mu_a\) extends \(u0\), \(\mu_b\) extends \(u1\), and
\(\sigma_a,\sigma_b,z_a,o_b\) are as in
Lemma~\ref{lem:normalized-restriction}.  For \(j\geq0\) put
\[
  X_j=(\mu_a\sigma_a^{\,j})_H^+,
  \qquad
  Y_j=(\mu_b\sigma_b^{\,j})_H^+.
\]
Since \(\mathcal L\) is finite, both sequences are eventually
periodic; let \(m_a\) and \(m_b\) be their least eventual periods, and
choose \(J_a,J_b\) with
\[
  X_{j+m_a}=X_j
  \quad(j\geq J_a),
  \qquad
  Y_{j+m_b}=Y_j
  \quad(j\geq J_b).
\]
Let \(d_L^K\) and \(d_R^K\) be the positive generators of the two
coordinate projections of \(\pi_{\mathrm{ab}}(K)\), which exist by
Lemma~\ref{lem:normalized-restriction}.  Then
\[
  d_L^K=m_az_a,
  \qquad
  d_R^K=m_bo_b.
\]
\end{lemma}

\begin{proof}
The equality \(X_{J_a+m_a}=X_{J_a}\), together with
Lemma~\ref{lem:branch-criterion}, gives an element \(h\in H\) with the
branch
\[
  \mu_a\sigma_a^{J_a}
  \longrightarrow
  \mu_a\sigma_a^{J_a+m_a}.
\]
Evaluating this branch at the tail \(\sigma_a^\infty\) shows that
\(h\) fixes \(a\); being increasing, \(h\) then maps points of \(I\)
near \(a\) into \(I\), so isolation gives \(h(I)=I\).  Its slope at
\(a\) has logarithm \(-m_a|\sigma_a|\), so
Lemma~\ref{lem:normalized-restriction} shows that its normalized
restriction has first endpoint coordinate \(-m_az_a\).  Hence
\(d_L^K\mid m_az_a\).

Conversely, let \(k\in K\), and let
\(h\in\operatorname{Stab}_H(I)\) be the corresponding element.  Once
\(j\) is sufficiently large, the branch of \(h\) at \(a\) sends
\(\mu_a\sigma_a^{\,j}\) to \(\mu_a\sigma_a^{\,j+t}\) for some
\(t\in\mathbb Z\).  Thus \(X_{j+t}=X_j\) for all sufficiently large
\(j\).  If \(t\neq0\), then \(|t|\) is an eventual period of
\((X_j)\), and the minimality of \(m_a\) gives \(m_a\mid t\).
If \(t=0\), this divisibility is automatic.  By
Lemma~\ref{lem:normalized-restriction} the first endpoint coordinate
of \(k\) is \(-tz_a\), and is therefore divisible by \(m_az_a\).  It
follows that \(d_L^K=m_az_a\); the proof of \(d_R^K=m_bo_b\) is
symmetric.
\end{proof}

\begin{proposition}[Effective construction of the normalized core]
\label{prop:normalized-core-effective}
Let \(H\leq F\) be closed, suppose that \(\mathcal L=C(H)\) is finite
and full, let \(I\) be a rational \(H\)-isolated interval and let
\(K=K_I(H)\).  Then the child automaton \(\mathcal A(C(K))\), and
hence the based directed \(2\)-complex \(C(K)\), can be constructed
from the child automaton \(\mathcal A(\mathcal L)\) together with the
preperiods and the periods of the expansions \(\alpha\) and \(\beta\).
\end{proposition}

\begin{proof}
Write \(d_L=d_L^K\) and \(d_R=d_R^K\), and let \(m_a,m_b,J_a,J_b\) be
as in Lemma~\ref{lem:normalized-endpoint-invariants}.

Put
\[
\begin{aligned}
  N_L&=\min\{n\geq1:|u|+1+i_n>|\mu_a\sigma_a^{J_a}|\},\\
  N_R&=\min\{n\geq1:|u|+1+j_n>|\mu_b\sigma_b^{J_b}|\}.
\end{aligned}
\]

Let \(h\) be the element of \(\operatorname{Stab}_H(I)\) with the
branch \(\mu_a\sigma_a^{J_a}\to\mu_a\sigma_a^{J_a+m_a}\) produced in
the proof of Lemma~\ref{lem:normalized-endpoint-invariants}.  For
\(n\geq N_L\) the word \(s_n\) extends \(\mu_a\sigma_a^{J_a}\), and
replacing that prefix by \(\mu_a\sigma_a^{J_a+m_a}\) inserts \(m_a\)
copies of \(\sigma_a\), hence exactly \(m_az_a=d_L\) zeros of the
enumeration defining the intervals \([s_n]\); the resulting word is
therefore \(s_{\,n+d_L}\).  Expanding the branch of \(h\) accordingly,
we obtain the branches
\[
  s_n\longrightarrow s_{\,n+d_L}
  \qquad(n\geq N_L).
\]
Let \(k_L\in K\) be the normalized restriction of \(h\).  Since
\(\nu_I\) maps \([0^n1]\) onto \([s_n]\), the element \(k_L\) has the
branches
\[
  0^n1\longrightarrow0^{\,n+d_L}1
  \qquad(n\geq N_L).
\]
Moreover the intervals \([0^j1]\) with \(j\geq n\) cover \([0^n]\)
apart from the point \(0\), which \(k_L\) fixes, and \(k_L\) carries
each of them onto \([0^{\,j+d_L}1]\), with slope \(2^{-d_L}\).  Hence
\(k_L\) maps \([0^n]\) linearly onto \([0^{\,n+d_L}]\); that is, it
also has the branches
\[
  0^n\longrightarrow0^{\,n+d_L}
  \qquad(n\geq N_L).
\]
The symmetric argument at \(b\) gives an element \(k_R\in K\) with the
branches \(1^n0\to1^{\,n+d_R}0\) and \(1^n\to1^{\,n+d_R}\) for every
\(n\geq N_R\).

For \(n\geq1\) put
\[
  \ell_n=(0^n)_K^+,
  \quad
  r_n=(1^n)_K^+,
  \quad
  S_n=(s_n)_H^+,
  \quad
  T_n=(t_n)_H^+ .
\]
By Lemma~\ref{lem:branch-criterion}, the branches of \(k_L\) and of
\(h\) displayed above give
\[
  \ell_{\,n+d_L}=\ell_n,
  \qquad
  S_{\,n+d_L}=S_n
  \qquad(n\geq N_L),
\]
and symmetrically \(r_{\,n+d_R}=r_n\) and \(T_{\,n+d_R}=T_n\) for
\(n\geq N_R\).  Moreover \(E_H(I)\) is the closure, under the two
child transitions of \(\mathcal L\), of the finite set
\[
  \{S_1,\ldots,S_{\,N_L+d_L-1},\;
    T_1,\ldots,T_{\,N_R+d_R-1}\},
\]
since the intervals \([s_n]\) and \([t_n]\) are the maximal standard
dyadic intervals contained in \(I\), and every standard dyadic
interval contained in \(I\) is a descendant of one of them.

We now construct a tree automaton \(\mathcal A_I^0\).  Its states are
a root state \(p_I\), one state \(\overline e\) for each
\(e\in E_H(I)\), left boundary states
\(L_1,\ldots,L_{\,N_L+d_L-1}\) and right boundary states
\(R_1,\ldots,R_{\,N_R+d_R-1}\); in the rules below the left indices
are read periodically after \(N_L\) and the right indices
periodically after \(N_R\).  Define
\[
  p_I\xrightarrow{\ 0\ }L_1,
  \qquad
  p_I\xrightarrow{\ 1\ }R_1,
\]
\[
  L_n\xrightarrow{\ 0\ }L_{n+1},
  \qquad
  L_n\xrightarrow{\ 1\ }\overline{S_n},
  \qquad
  R_n\xrightarrow{\ 0\ }\overline{T_n},
  \qquad
  R_n\xrightarrow{\ 1\ }R_{n+1},
\]
and \(\overline e\xrightarrow{\ i\ }\overline{e_i}\) whenever the
children of \(e\) in \(\mathcal L\) are \(e_0\) and \(e_1\).  All
these transitions are defined, since the children of an edge of
\(E_H(I)\) again lie in \(E_H(I)\), and every state is reached from
\(p_I\), since \(L_n\), \(R_n\), \(\overline{S_n}\) and
\(\overline{T_n}\) are reached by \(0^n\), \(1^n\), \(0^n1\) and
\(1^n0\), and the remaining states by longer words.  Thus
\(\mathcal A_I^0\) is a finite full tree automaton, in which \(0^n\)
reaches \(L_n\), \(1^n\) reaches \(R_n\), and a word \(w\) containing
both digits reaches \(\overline{\widehat w{}_H^+}\).

Sending \(p_I\) to the root state of \(\mathcal A(C(K))\), \(L_n\) to
\(\ell_n\), \(R_n\) to \(r_n\), and \(\overline e\) to the inner edge
of \(C(K)\) corresponding to \(e\), defines a morphism
\[
  \Pi\colon\mathcal A_I^0\longrightarrow\mathcal A(C(K)).
\]
It is well defined by the periodicity relations above, and it commutes
with the transitions: the children of \(\ell_n\) are \(\ell_{n+1}\)
and \((0^n1)_K^+\), the latter corresponding to \(S_n\); the children
of \(r_n\) are \((1^n0)_K^+\), corresponding to \(T_n\), and
\(r_{n+1}\); and \(\widehat{wi}=\widehat w\,i\) for every word \(w\)
containing both digits.  In particular \(\Pi\) carries the state
reached by a word to the state reached by the same word.

We claim that \(\mathcal D(\mathcal A_I^0)=K\).  A tree-pair diagram
accepted by \(\mathcal A_I^0\) is accepted by \(\mathcal A(C(K))\)
after applying \(\Pi\), so
\(\mathcal D(\mathcal A_I^0)\subseteq\Cl(K)=K\).  Conversely, let
\(k\in K\) and expand its reduced tree pair, first so that both trees
have at least three leaves and then along the two extreme leaves,
until the leftmost branch pair is \(0^m\to0^n\) with \(m,n\geq N_L\)
and the rightmost is \(1^{m'}\to1^{n'}\) with \(m',n'\geq N_R\).
Every other branch pair \(v\to w\) has both words containing both
digits; lifting it to \(\operatorname{Stab}_H(I)\) gives the branch
\(\widehat v\to\widehat w\), so
\(\widehat v{}_H^+=\widehat w{}_H^+\) by
Lemma~\ref{lem:branch-criterion} and \(v,w\) reach the same state of
\(\mathcal A_I^0\).  For the leftmost pair, \(m-n\) is the first
endpoint coordinate of \(k\), so \(d_L\mid m-n\) and \(L_m=L_n\); the
rightmost pair is symmetric.  Thus the expanded diagram is accepted by
\(\mathcal A_I^0\), and \(k\in\mathcal D(\mathcal A_I^0)\).

Let \(\mathcal A_I\) be the folded quotient of \(\mathcal A_I^0\).  It
is a finite full folded tree automaton, and
\(\mathcal D(\mathcal A_I)=\mathcal D(\mathcal A_I^0)=K\).  Since
\(\mathcal A_I\) is the quotient of \(\mathcal A_I^0\) by the least
congruence with folded quotient, and the quotient of
\(\mathcal A_I^0\) by \(\ker\Pi\) is folded, being a subautomaton of
\(\mathcal A(C(K))\), the morphism \(\Pi\) factors through
\(\mathcal A_I\).  Let
\[
  \overline\Pi\colon\mathcal A_I\longrightarrow\mathcal A(C(K))
\]
be the induced morphism.  Since \(\mathcal A_I\) is folded and accepts
\(K\), \cite[Lemma~4.9]{GolanMaximal} provides a morphism
\[
  \varphi\colon\mathcal A(C(K))\longrightarrow\mathcal A_I .
\]
The composites \(\varphi\overline\Pi\) and \(\overline\Pi\varphi\) are
morphisms from \(\mathcal A(C(K))\) and from \(\mathcal A_I\) to
themselves, and so are the two identity maps; by the uniqueness of
morphisms of tree automata, both composites are the identity.  Hence
\[
  \mathcal A_I\cong\mathcal A(C(K)),
\]
and Proposition~\ref{prop:finite-full-core-automaton} recovers \(C(K)\)
from \(\mathcal A_I\) as a based directed \(2\)-complex.
\end{proof}

For every closed subgroup \(G\leq F\) with finite full core, define
its \emph{inner-edge complexity} by
\[
  \kappa(G)=
  \left|
    E_{\mathrm{in}}\bigl(C(G)\bigr)
  \right|.
\]

\begin{corollary}
\label{cor:strict-inner-edge-descent}
Let \(H\leq F\) be closed, suppose that \(C(H)\) is finite and full,
and let
\[
  I=(a,b)\neq(0,1)
\]
be a rational \(H\)-isolated interval.  If \(K=K_I(H)\), then
\[
  \kappa(K)<\kappa(H).
\]
\end{corollary}

\begin{proof}
By Proposition~\ref{prop:normalized-core},
\[
  \kappa(K)=|E_H(I)|.
\]
We prove that
\[
  E_H(I)\subsetneq E_{\mathrm{in}}(\mathcal L).
\]

Choose an endpoint \(c\in(0,1)\) of \(I\): take \(c=a\) if \(a>0\),
and take \(c=b\) otherwise.  If \(c\) is nondyadic, choose a
sufficiently long prefix \(w\) of its unique binary expansion.  If
\(c\) is dyadic, choose a sufficiently small standard dyadic interval
\([w]\) lying on the side of \(c\) opposite \(I\), with \(c\) as one
of its endpoints.  In either case, we may choose \(w\) so that it is a
finite binary word containing both digits \(0\) and \(1\),
\[
  c\in[w],
  \qquad
  [w]\not\subseteq I,
\]
and \([w]\) does not contain the other endpoint of \(I\).

By Lemma~\ref{lem:boundary-incidence}, \(w^+\) is an inner edge of
\(\mathcal L\).  We claim that
\[
  w^+\notin E_H(I).
\]
Otherwise, there would be a finite binary word \(v\) such that
\[
  [v]\subseteq I
  \qquad\text{and}\qquad
  v^+=w^+.
\]
By Lemma~\ref{lem:branch-criterion}, some \(h\in H\) would have the
branch
\[
  w\longrightarrow v.
\]
Since \(c\in[w]\), this gives
\[
  h(c)\in[v]\subseteq I.
\]
Continuity implies that \(h(I)\cap I\neq\varnothing\).  Isolation
therefore gives \(h(I)=I\), so \(h\) must fix the endpoint \(c\).
This contradicts \(h(c)\in I\).

Thus \(w^+\notin E_H(I)\), and therefore
\[
  E_H(I)\subsetneq E_{\mathrm{in}}(\mathcal L).
\]
It follows that
\[
  \kappa(K)<\kappa(H).
\]
\end{proof}

\subsection{Path classes inside an isolated interval}
\label{subsec:path-classes-isolated-interval}

Let \(H\leq F\) be closed, suppose that \(C(H)\) is finite and full,
and let \(I=(a,b)\) be a rational \(H\)-isolated interval.  Let
\[
  K=K_I(H)
\]
be the normalized restriction of \(\operatorname{Stab}_H(I)\) defined
in Definition~\ref{def:normalized-restriction}.

Since the two cores \(C(H)\) and \(C(K)\) occur simultaneously, we use
subscripts to distinguish the corresponding notation.  Thus
\(v_H(x)\) and \(v_K(x)\) are the types of the dyadic point \(x\) in
\(C(H)\) and in \(C(K)\), respectively, and \(\omega_H[x,y]\) and
\(\omega_K[x,y]\) denote the interval classes associated with
\([x,y]\) in the two cores.

We write
\[
  V_H(I)=\{v_H(x):x\in\D\cap I\}
\]
for the set of vertices of \(C(H)\) associated with dyadic points
lying in \(I\).  For \(u,v\in V_H(I)\), let
\[
  \Path_I(H)(u,v)
  =
  \left\{
    \omega_H[x,y]:
    x,y\in\D,\ x<y,\ [x,y]\subseteq I,\
    v_H(x)=u,\ v_H(y)=v
  \right\}.
\]
Thus \(\Path_I(H)(u,v)\) consists of the path classes in \(C(H)\) from
\(u\) to \(v\) which arise from intervals with dyadic endpoints
contained in \(I\).  We also write
\[
  \Path_I(H)
  =
  \left\{
    \omega_H[x,y]:
    x,y\in\D,\ x<y,\ [x,y]\subseteq I
  \right\}
\]
for the set of all such classes, without prescribing their initial and
terminal vertices.

\begin{proposition}[The isolated-interval dictionary]
\label{prop:isolated-interval-path-correspondence}
Let \(H\leq F\) be closed, suppose that \(C(H)\) is finite and full,
let \(I\) be a rational \(H\)-isolated interval, and let
\(K=K_I(H)\).  Then the following statements hold.
\begin{enumerate}[label=\textup{(\arabic*)}]
\item The rule
      \[
        \theta_I\bigl(v_K(c)\bigr)
        =
        v_H\bigl(\nu_I(c)\bigr)
      \]
      defines a bijection
      \[
        \theta_I\colon
        V_{\mathrm{in}}\bigl(C(K)\bigr)
        \longrightarrow
        V_H(I).
      \]
\item For all \(u',v'\in V_{\mathrm{in}}(C(K))\), the rule
      \[
        \Phi_I\bigl(\omega_K[c,d]\bigr)
        =
        \omega_H\bigl[\nu_I(c),\nu_I(d)\bigr]
      \]
      defines a bijection
      \[
        \Phi_I\colon
        \Path\bigl(C(K)\bigr)(u',v')
        \longrightarrow
        \Path_I(H)\bigl(\theta_I(u'),\theta_I(v')\bigr).
      \]
\end{enumerate}
\end{proposition}

\begin{proof}
As observed in the proof of Proposition~\ref{prop:normalized-core},
for all \(c,c'\in\D\),
\[
  c'\in c\cdot K
  \quad\Longleftrightarrow\quad
  \nu_I(c')\in \nu_I(c)\cdot H.
\]
Together with Proposition~\ref{prop:orbit-core}, this shows that
\[
  v_K(c)=v_K(c')
  \quad\Longleftrightarrow\quad
  v_H\bigl(\nu_I(c)\bigr)
  =
  v_H\bigl(\nu_I(c')\bigr).
\]
Therefore \(\theta_I\) is well defined and injective.  It is
surjective because \(\nu_I\) maps \(\D\) bijectively onto
\(\D\cap I\).

We now consider \(\Phi_I\).  By
Proposition~\ref{prop:interval-realization}, every class in
\(\Path(C(K))(u',v')\) is of the form \(\omega_K[c,d]\) for some
dyadic points \(c<d\).  The points \(\nu_I(c)\) and \(\nu_I(d)\) are
dyadic points of \(I\), the interval \([\nu_I(c),\nu_I(d)]\) is
contained in \(I\), and the associated vertices are
\(\theta_I(u')\) and \(\theta_I(v')\).  Hence
\[
  \omega_H\bigl[\nu_I(c),\nu_I(d)\bigr]
  \in
  \Path_I(H)\bigl(\theta_I(u'),\theta_I(v')\bigr).
\]

Suppose that
\[
  \omega_K[c,d]=\omega_K[c',d'].
\]
By Proposition~\ref{prop:exact-transport}, there exists \(k\in K\)
with \(k(c)=c'\) and \(k(d)=d'\).  Let
\(h\in\operatorname{Stab}_H(I)\) be the element whose normalized
restriction is \(k\).  Then
\[
  h\bigl(\nu_I(c)\bigr)=\nu_I(c'),
  \qquad
  h\bigl(\nu_I(d)\bigr)=\nu_I(d'),
\]
so Proposition~\ref{prop:exact-transport} gives
\[
  \omega_H\bigl[\nu_I(c),\nu_I(d)\bigr]
  =
  \omega_H\bigl[\nu_I(c'),\nu_I(d')\bigr].
\]
Thus \(\Phi_I\) is well defined.

Conversely, suppose that these two path classes in \(C(H)\) are
equal.  By Proposition~\ref{prop:exact-transport}, there exists
\(h\in H\) such that
\[
  h\bigl(\nu_I(c)\bigr)=\nu_I(c'),
  \qquad
  h\bigl(\nu_I(d)\bigr)=\nu_I(d').
\]
In particular, \(h(I)\cap I\neq\varnothing\).  Since \(I\) is
\(H\)-isolated, \(h(I)=I\).  The normalized restriction of \(h\)
therefore belongs to \(K\) and maps \(c\) to \(c'\) and \(d\) to
\(d'\).  Proposition~\ref{prop:exact-transport} now gives
\[
  \omega_K[c,d]=\omega_K[c',d'].
\]
Hence \(\Phi_I\) is injective.

Finally, let \(\omega_H[x,y]\) belong to
\(\Path_I(H)(\theta_I(u'),\theta_I(v'))\).  The points
\[
  c=\nu_I^{-1}(x),
  \qquad
  d=\nu_I^{-1}(y)
\]
are dyadic and satisfy \(c<d\).  Since \(\theta_I\) is injective,
\(v_K(c)=u'\) and \(v_K(d)=v'\).  Therefore
\[
  \omega_K[c,d]\in\Path\bigl(C(K)\bigr)(u',v'),
\]
and
\[
  \Phi_I\bigl(\omega_K[c,d]\bigr)=\omega_H[x,y].
\]
Thus \(\Phi_I\) is surjective.
\end{proof}

\begin{corollary}[Rank one inside an isolated interval]
\label{cor:rank-one-isolated-interval}
Let \(H\leq F\) be closed, suppose that \(C(H)\) is finite and full,
and let \(I\) be a rational \(H\)-isolated interval.  If
\(\pi_{\mathrm{ab}}(K_I(H))\) has rank one, then there exist
\(u,v\in V_H(I)\) such that
\[
  \Path_I(H)(u,v)
\]
is infinite.  In particular, \(\Path(C(H))\) is infinite.
\end{corollary}

\begin{proof}
Write \(K=K_I(H)\).  By Proposition~\ref{prop:normalized-core}, the
core \(C(K)\) is finite and full, and by
Lemma~\ref{lem:normalized-restriction} the group \(K\) is closed and
both coordinate projections of \(\pi_{\mathrm{ab}}(K)\) are nonzero.
Since \(\pi_{\mathrm{ab}}(K)\) has rank one, it is generated by a
single vector with both coordinates nonzero.
Corollary~\ref{cor:rank-one-infinite-paths} therefore gives inner
vertices \(u',v'\) of \(C(K)\) such that
\[
  \Path\bigl(C(K)\bigr)(u',v')
\]
is infinite.  By
Proposition~\ref{prop:isolated-interval-path-correspondence}, this set
is in bijection with
\(\Path_I(H)(\theta_I(u'),\theta_I(v'))\), which is therefore
infinite.  The latter set is contained in
\(\Path(C(H))(\theta_I(u'),\theta_I(v'))\).
\end{proof}

\begin{corollary}[Minimal rank two inside an isolated interval]
\label{cor:rank-two-isolated-interval}
Let \(H\leq F\) be closed, suppose that \(C(H)\) is finite and full,
and let \(I\) be a rational \(H\)-isolated interval.  Suppose that
\(K=K_I(H)\) acts minimally on \((0,1)\) and that
\[
  A_K=\pi_{\mathrm{ab}}(K)
\]
has rank two.  Let \(d_L^K\) and \(d_R^K\) be the positive generators
of the two coordinate projections of \(A_K\), and set
\[
  q_I=
  \bigl[
    d_L^K\mathbb Z\times d_R^K\mathbb Z:A_K
  \bigr].
\]
Then, for all \(u,v\in V_H(I)\),
\[
  \bigl|\Path_I(H)(u,v)\bigr|=q_I.
\]
\end{corollary}

\begin{proof}
By Proposition~\ref{prop:normalized-core}, the core \(C(K)\) is finite
and full, and by Lemma~\ref{lem:normalized-restriction} the group
\(K\) is closed.  Let \(u',v'\in V_{\mathrm{in}}(C(K))\) be the unique
vertices with
\[
  \theta_I(u')=u,
  \qquad
  \theta_I(v')=v.
\]
Proposition~\ref{prop:exact-inner-path-count}, applied to \(K\), gives
\[
  \bigl|\Path\bigl(C(K)\bigr)(u',v')\bigr|=q_I,
\]
and the result follows from
Proposition~\ref{prop:isolated-interval-path-correspondence}.
\end{proof}

\subsection{The rank-one obstruction inside an isolated interval}
\label{subsec:rank-one-obstruction-isolated-interval}

In this subsection we show that if the normalized restriction of \(H\)
to a rational \(H\)-isolated interval has rank-one endpoint image, then
\(H\) is not of type \(\mathrm{FP}_2\).  The argument exhibits that
restriction as a quasi-retract of \(H\).

Choose word metrics on the finitely generated groups under
consideration.  A map
\[
  f\colon G\longrightarrow Q
\]
is \emph{Lipschitz} if there is a constant \(C\geq0\) such that
\[
  d_Q\bigl(f(g_1),f(g_2)\bigr)
  \leq
  C\,d_G(g_1,g_2)
\]
for all \(g_1,g_2\in G\).  Two maps
\[
  f_1,f_2\colon G\longrightarrow Q
\]
are \emph{at bounded distance} if there is a constant \(C\geq0\) such
that
\[
  d_Q\bigl(f_1(g),f_2(g)\bigr)\leq C
\]
for every \(g\in G\).

A finitely generated group \(Q\) is a \emph{quasi-retract} of a
finitely generated group \(G\) if there are Lipschitz maps
\[
  \sigma\colon Q\longrightarrow G
  \qquad\text{and}\qquad
  \rho\colon G\longrightarrow Q
\]
such that \(\sigma\circ\rho\) is at bounded distance from the identity
map of \(Q\).  Since any two word metrics arising from finite
generating sets are bi-Lipschitz equivalent, this definition does not
depend on the chosen word metrics.

We shall use the following theorem of Alonso.

\begin{theorem}[Alonso]
\label{thm:alonso-quasi-retract}
Let \(G\) and \(Q\) be finitely generated groups, and suppose that
\(Q\) is a quasi-retract of \(G\).  If \(G\) is of type
\(\mathrm{FP}_n\) for some \(n\geq2\), then \(Q\) is of type
\(\mathrm{FP}_n\).
\end{theorem}

This is \cite[Theorem~8]{Alonso}.

\begin{lemma}
\label{lem:isolated-interval-lipschitz-retraction}
Let \(H\leq F\) be closed, suppose that \(C(H)\) is finite and full,
let \(I\) be a rational \(H\)-isolated interval, and put
\[
  K=K_I(H).
\]
Both \(H\) and \(K\) are finitely generated, by
Theorem~\ref{thm:core-closure}\textup{(ii)} and
Proposition~\ref{prop:normalized-core}.  There is a Lipschitz map
\[
  \rho\colon H\longrightarrow K
\]
such that, for every \(h\in\operatorname{Stab}_H(I)\),
\[
  \rho(h)=\nu_I(h|_I)\nu_I^{-1}.
\]
\end{lemma}

\begin{proof}
Put
\[
  \mathcal X=I\cdot H.
\]
The intervals in \(\mathcal X\) are pairwise disjoint.  We first choose
compatible coordinate maps
\[
  \zeta_J\colon(0,1)\longrightarrow J,
  \qquad J\in\mathcal X,
\]
such that:
\begin{enumerate}[label=\textup{(\arabic*)}]
\item \(\zeta_I=\nu_I\);
\item for every \(J\in\mathcal X\), there is an element \(g_J\in H\)
      satisfying
      \[
        g_J(I)=J
        \qquad\text{and}\qquad
        \zeta_J=\nu_Ig_J;
      \]
\item whenever \(h\in H\) is affine on \(J\),
      \[
        \zeta_{h(J)}=\zeta_Jh.
      \]
\end{enumerate}

To construct these maps, call two intervals in \(\mathcal X\) affinely
equivalent if an element of \(H\) maps one affinely onto the other.
Choose one representative \(J_0\) of each affine-equivalence class and
an element \(g_0\in H\) mapping \(I\) onto \(J_0\).  For the class
containing \(I\), take \(J_0=I\) and \(g_0=1\).  Define
\[
  \zeta_{J_0}=\nu_Ig_0.
\]

For every \(J\) affinely equivalent to \(J_0\), let
\[
  a_J\colon J_0\longrightarrow J
\]
be the unique increasing affine map.  By the definition of affine
equivalence, there is an element \(f_J\in H\) whose restriction to
\(J_0\) is \(a_J\).  Put
\[
  \zeta_J=\zeta_{J_0}a_J.
\]
Since \(\zeta_{J_0}=\nu_Ig_0\), we have
\[
  \zeta_J=\nu_I(g_0f_J),
\]
which proves property \textup{(2)}.

If \(h\in H\) is affine on \(J\), then \(a_Jh\) is the unique
increasing affine map from \(J_0\) onto \(h(J)\).  Hence
\[
  \zeta_{h(J)}=\zeta_Jh,
\]
proving property \textup{(3)}.

For \(J\in\mathcal X\) and \(h\in H\), define the increasing homeomorphism
\[
  A_J(h)\colon(0,1)\longrightarrow(0,1)
\]
by
\[
  A_J(h)=\zeta_Jh\zeta_{h(J)}^{-1}.
\]
By property \textup{(2)}, choose \(g_J,g_{h(J)}\in H\) such that
\[
  \zeta_J=\nu_Ig_J,
  \qquad
  \zeta_{h(J)}=\nu_Ig_{h(J)}.
\]
Then
\[
  A_J(h)
  =
  \nu_I
  \bigl((g_Jh g_{h(J)}^{-1})|_I\bigr)
  \nu_I^{-1}.
\]
The element \(g_Jh g_{h(J)}^{-1}\) stabilizes \(I\): it maps \(I\)
first onto \(J\), then onto \(h(J)\), and finally back onto \(I\).
Thus \(A_J(h)\) is its normalized restriction, and therefore
\[
  A_J(h)\in K.
\]

These maps satisfy the cocycle identity.  Indeed,
\[
\begin{aligned}
  A_J(h)A_{h(J)}(g)
  &=
  \zeta_Jh\zeta_{h(J)}^{-1}
  \zeta_{h(J)}g\zeta_{g(h(J))}^{-1}
  \\
  &=
  \zeta_Jhg\zeta_{(hg)(J)}^{-1}
  \\
  &=
  A_J(hg).
\end{aligned}
\]

If \(h\) is affine on \(J\), then property \textup{(3)} gives
\[
  A_J(h)=1.
\]
Consequently, \(A_J(h)\neq1\) only if the interior of \(J\) contains a
breakpoint of \(h\).  Since the intervals in \(\mathcal X\) are pairwise
disjoint and \(h\) has only finitely many breakpoints,
\[
  A_J(h)=1
\]
for all but finitely many \(J\in\mathcal X\).

Define
\[
  \rho(h)=A_I(h).
\]
If \(h\in\operatorname{Stab}_H(I)\), then \(h(I)=I\), and
therefore
\[
  \rho(h)
  =A_I(h)
  =\nu_I(h|_I)\nu_I^{-1}.
\]

We now prove that \(\rho\) is Lipschitz.  Let \(S\) be a finite
symmetric generating set of \(H\), let \(d_H\) be the corresponding
left-invariant word metric, and fix a left-invariant word metric \(d_K\)
on \(K\).  For each of these groups \(G\), we use the convention
\[
  d_G(g,h)=|g^{-1}h|_G.
\]
Write \(|\cdot|_K\) for the corresponding word length.  Fix \(s\in S\).  If
an interval \(J\in\mathcal X\) contains no breakpoint of \(s\), then \(s\) is
affine on \(J\).  The compatibility of the coordinate maps therefore
gives
\[
  A_J(s)=1.
\]
The intervals in \(\mathcal X\) are pairwise disjoint, and \(s\) has only
finitely many breakpoints.  Hence only finitely many intervals
\(J\in\mathcal X\) contain a breakpoint of \(s\), and therefore
\[
  \{A_J(s):J\in\mathcal X\}
\]
is finite.

Since \(S\) is finite, there is a constant \(M\) such that
\[
  |A_J(s)|_K\leq M
\]
for every \(s\in S\) and every \(J\in\mathcal X\).

Let \(h\in H\) and \(s\in S\).  Applying the cocycle identity with
initial interval \(I\) gives
\[
\begin{aligned}
  \rho(hs)
  &=A_I(hs)
  \\
  &=A_I(h)A_{h(I)}(s)
  \\
  &=\rho(h)A_{h(I)}(s).
\end{aligned}
\]
Consequently,
\[
\begin{aligned}
  d_K\bigl(\rho(h),\rho(hs)\bigr)
  &=|\rho(h)^{-1}\rho(hs)|_K
  \\
  &=|A_{h(I)}(s)|_K
  \\
  &\leq M.
\end{aligned}
\]
Applying this estimate successively along a shortest word from
\(h_1\) to \(h_2\) gives
\[
  d_K\bigl(\rho(h_1),\rho(h_2)\bigr)
  \leq
  M\,d_H(h_1,h_2).
\]
Thus \(\rho\) is Lipschitz.
\end{proof}

\begin{lemma}
\label{lem:rank-one-normalized-restriction-embedding}
Let \(H\leq F\) be closed, suppose that \(C(H)\) is finite and full,
and let \(I=(a,b)\) be a rational \(H\)-isolated interval.  Put
\[
  K=K_I(H),
\]
and suppose that \(\pi_{\mathrm{ab}}(K)\) has rank one.  Then there is
an injective homomorphism
\[
  \sigma\colon K\longrightarrow H
\]
such that every \(\sigma(k)\) stabilizes \(I\) and has normalized
restriction \(k\) on \(I\).
\end{lemma}

\begin{proof}
Write
\[
  \pi_{\mathrm{ab}}(K)=\mathbb Z(a_L,a_R)
\]
and define an epimorphism
\[
  \chi\colon K\longrightarrow\mathbb Z
\]
by
\[
  \pi_{\mathrm{ab}}(k)=\chi(k)(a_L,a_R).
\]
Put \(N=\ker\chi\).

Let \(u\in N\).  Then \(u\) has slope \(1\) at both endpoints of
\([0,1]\), and is therefore the identity on one-sided neighborhoods of
\(0\) and of \(1\).  Since \(\nu_I\) is an increasing homeomorphism of
\((0,1)\) onto \(I\), the map
\[
  u_I=\nu_I^{-1}u\nu_I\colon I\longrightarrow I
\]
is the identity on one-sided neighborhoods of \(a\) and of \(b\).

Let \(\overline u\) be the extension of \(u_I\) by the identity outside
\(I\).  By the definition of \(K\), there is
\(h\in\operatorname{Stab}_H(I)\) with \(h|_I=u_I\).  Choose dyadic
points \(x<y\) in \(I\) such that \(u_I\) is the identity on \((a,x]\)
and on \([y,b)\).  Then \(\overline u\) is the identity on \([0,x]\)
and on \([y,1]\) and agrees with \(h\) on \([x,y]\); in particular
\(\overline u\in F\), and \(\overline u\) is piecewise \(H\).  Since
\(H\) is closed,
\[
  \overline u\in H.
\]
The map \(u\mapsto\overline u\) is an injective homomorphism from \(N\)
to \(H\).

Choose \(t\in K\) satisfying \(\chi(t)=1\), and choose
\[
  g\in\operatorname{Stab}_H(I)
\]
whose normalized restriction is \(t\).  We have
\[
  K=N\rtimes\langle t\rangle,
\]
so every element of \(K\) has a unique expression \(ut^n\), where
\(u\in N\) and \(n\in\mathbb Z\).  Moreover,
\[
  g\,\overline u\,g^{-1}
  =
  \overline{tut^{-1}}
  \qquad(u\in N).
\]
Indeed, the two sides agree on \(I\) and are the identity outside
\(I\).  Therefore the formula
\[
  \sigma(ut^n)=\overline u\,g^n
\]
defines a homomorphism from \(K\) to \(H\).  Its normalized restriction
to \(I\) is \(ut^n\).  Thus \(\sigma\) is injective and has the required
properties.
\end{proof}

\begin{theorem}
\label{thm:rank-one-isolated-not-FP2}
Let \(H\leq F\) be closed, suppose that \(C(H)\) is finite and full,
and let \(I\) be a rational \(H\)-isolated interval.  If
\[
  \operatorname{rank}\pi_{\mathrm{ab}}\bigl(K_I(H)\bigr)=1,
\]
then \(H\) is not of type \(\mathrm{FP}_2\).
\end{theorem}

\begin{proof}
Put
\[
  K=K_I(H).
\]
By Lemma~\ref{lem:normalized-restriction}, \(K\) is closed, and by
Proposition~\ref{prop:normalized-core}, \(C(K)\) is finite and full.  In
particular, both \(H\) and \(K\) are finitely generated.

Let
\[
  \rho\colon H\longrightarrow K
\]
be the map supplied by
Lemma~\ref{lem:isolated-interval-lipschitz-retraction}, and let
\[
  \sigma\colon K\longrightarrow H
\]
be the homomorphism supplied by
Lemma~\ref{lem:rank-one-normalized-restriction-embedding}.  Since
\(\sigma\) is a homomorphism between finitely generated groups, it is
Lipschitz.

Moreover, \(\sigma(k)\in\operatorname{Stab}_H(I)\), so the defining
property of \(\rho\) in
Lemma~\ref{lem:isolated-interval-lipschitz-retraction} gives
\[
  \rho\bigl(\sigma(k)\bigr)
  =\nu_I\bigl(\sigma(k)|_I\bigr)\nu_I^{-1}
  =k.
\]
Therefore
\[
  \sigma\circ\rho=\operatorname{id}_K.
\]
Thus \(K\) is a quasi-retract of \(H\).

If \(H\) were of type \(\mathrm{FP}_2\),
Theorem~\ref{thm:alonso-quasi-retract} would imply that \(K\) is of
type \(\mathrm{FP}_2\).  This contradicts
Theorem~\ref{thm:rank-one-not-FP2}, because \(K\) is closed, its core is
finite and full, and its image in the abelianization of \(F\) has rank
one.
\end{proof}

\begin{corollary}
\label{cor:FP2-isolated-restriction-rank-two}
Let \(H\leq F\) be closed, suppose that \(C(H)\) is finite and full,
and suppose that \(H\) is of type \(\mathrm{FP}_2\).  Then, for every
rational \(H\)-isolated interval \(I\),
\[
  \pi_{\mathrm{ab}}\bigl(K_I(H)\bigr)
\]
has rank two.
\end{corollary}

\begin{proof}
By Lemma~\ref{lem:normalized-restriction}, both coordinate projections
of this subgroup are nonzero.  Its rank is therefore either one or two,
and Theorem~\ref{thm:rank-one-isolated-not-FP2} excludes rank one.
\end{proof}

\section{The nonminimal rank-two case}
\label{sec:nonminimal-rank-two}

Throughout this section, let \(H\leq F\) be closed, suppose that
\[
  \mathcal L=C(H)
\]
is finite and full, and suppose that \(\pi_{\mathrm{ab}}(H)\) has rank
two.  We assume that the action of \(H\) on \((0,1)\) is not minimal.
We shall prove that \(H\) has a proper rational \(H\)-isolated interval.

\subsection{Fixed points and exceptional minimal sets}
\label{subsec:fixed-points-exceptional-minimal-sets}

\begin{proposition}
\label{prop:nonminimal-rank-two-alternative}
Let \(H\leq F\) be closed.  Suppose that \(C(H)\) is finite and full,
that \(\pi_{\mathrm{ab}}(H)\) has rank two, and that the action of \(H\)
on \((0,1)\) is not minimal.  Then exactly one of the following holds.
\begin{enumerate}[label=\textup{(\arabic*)}]
\item The group \(H\) has a fixed point in \((0,1)\), and every orbital
      of \(H\) is a proper rational \(H\)-isolated interval.
\item The group \(H\) has no fixed point in \((0,1)\) and has a unique
      exceptional minimal set \(M\).  Every gap of \(M\) is
      \(H\)-isolated.
\end{enumerate}
\end{proposition}

\begin{proof}
Suppose first that \(H\) has a fixed point in \((0,1)\).  By
Lemma~\ref{lem:finite-fixed-set}, the set
\[
  \operatorname{Fix}_{(0,1)}(H)
\]
is finite and consists of rational points.  Consequently, every orbital
\(I\) of \(H\) is a proper interval with rational endpoints.  Every
element of \(H\) fixes the endpoints of \(I\), and therefore preserves
\(I\).  Hence \(I\) is a rational \(H\)-isolated interval.

Now suppose that \(H\) has no fixed point in \((0,1)\).  By
Theorem~\ref{thm:core-closure}\textup{(ii)}, the group \(H\) is finitely
generated, so Lemma~\ref{lem:minimal-set-alternative} applies.  The
action is not minimal by assumption.  We claim that it cannot have a
closed discrete orbit in \((0,1)\).

Indeed, let
\[
  \mathcal O=\{x_n:n\in\mathbb Z\},
  \qquad
  x_n<x_{n+1},
\]
be a closed discrete orbit in \((0,1)\).  Since \((0,1)\) is an orbital of \(H\),
\[
  \lim_{n\to-\infty}x_n=0,
  \qquad
  \lim_{n\to\infty}x_n=1.
\]
The action of \(H\) on the ordered set \(\mathcal O\) induces a
homomorphism
\[
  \chi\colon H\longrightarrow\mathbb Z
\]
characterized by
\[
  h(x_n)=x_{n+\chi(h)}.
\]
Every element of \(\ker\chi\) fixes every \(x_n\).  Since elements of
\(F\) are affine on sufficiently small one-sided neighborhoods of \(0\)
and \(1\), and the sequence \((x_n)\) accumulates at both endpoints,
every element of \(\ker\chi\) has slope \(1\) at both endpoints.  Thus
\[
  \ker\chi\subseteq\ker\pi_{\mathrm{ab}}.
\]
Conversely, every element of \(\ker\pi_{\mathrm{ab}}\) fixes
neighborhoods of both \(0\) and \(1\).  It therefore fixes \(x_n\) for
all sufficiently negative \(n\).  Since \(h(x_n)=x_{n+\chi(h)}\),
this forces \(\chi(h)=0\).  Hence
\[
  \ker\chi=\ker\pi_{\mathrm{ab}}.
\]
Consequently,
\[
  \pi_{\mathrm{ab}}(H)\cong\chi(H)\leq\mathbb Z,
\]
contradicting the assumption that \(\pi_{\mathrm{ab}}(H)\) has rank two.

Lemma~\ref{lem:minimal-set-alternative} now implies that \(H\) has a
unique exceptional minimal set \(M\), and that \(H\) permutes the gaps
of \(M\).  Since distinct gaps are disjoint, if \(I\) is a gap and
\[
  h(I)\cap I\neq\varnothing,
\]
then \(h(I)=I\).  Hence every gap of \(M\) is \(H\)-isolated.
\end{proof}

In the first case, Proposition~\ref{prop:nonminimal-rank-two-alternative}
already provides a proper rational \(H\)-isolated interval.  In the
second case, the next subsection will use the finite child automaton of
\(C(H)\) to construct finitely many rational gaps meeting every
\(H\)-orbit of gaps.

\subsection{Orbit closures in the child automaton}
\label{subsec:orbit-closures-child-automaton}

When discussing reachability and strongly connected components, we
regard the child automaton \(\mathcal A(\mathcal L)\) as a labelled
directed graph.  Denote its underlying graph by
\[
  \Gamma(\mathcal L).
\]
Its vertices are all the edges of \(\mathcal L\), and its labelled arrows
are
\[
  e\xrightarrow{0}e_0,
  \qquad
  e\xrightarrow{1}e_1,
\]
whenever \(e_0\) and \(e_1\) are the two children of \(e\).

The subgraph induced by the inner edges of \(\mathcal L\) is the
\emph{inner child graph}, denoted by
\[
  \Gamma_{\mathrm{in}}(\mathcal L).
\]
By Lemma~\ref{lem:boundary-incidence}, the children of an inner edge are
again inner edges.  Thus a directed path which begins at an inner edge
remains in \(\Gamma_{\mathrm{in}}(\mathcal L)\), whereas a path beginning
at a non-inner edge may later enter the inner child graph.

A strongly connected component of
\(\Gamma_{\mathrm{in}}(\mathcal L)\) is called \emph{cyclic} if it
contains a directed cycle; we abbreviate ``cyclic strongly connected
component'' to \emph{cyclic component}.  For vertices or subsets of vertices \(A\)
and \(B\) of \(\Gamma(\mathcal L)\), write
\[
  A\leadsto B
\]
if there is a directed path in the full graph \(\Gamma(\mathcal L)\)
from a vertex of \(A\) to a vertex of \(B\).  Paths of length zero are
allowed.

Every infinite directed path in a finite graph eventually remains in a
cyclic component.  Indeed, once such a path leaves
a strongly connected component, it cannot return to it, and the graph
has only finitely many strongly connected components.  Consequently,
if \(x\in(0,1)\) and
\[
  x=.\!x_1x_2x_3\cdots
\]
is a binary expansion of \(x\), then its state path
\[
  p,
  \quad x_1^+,
  \quad (x_1x_2)^+,
  \quad (x_1x_2x_3)^+,
  \ldots
\]
enters \(\Gamma_{\mathrm{in}}(\mathcal L)\) and eventually remains in a
cyclic component of it.  Indeed, since \(x\in(0,1)\), the expansion
contains both digits, so all but finitely many of its prefixes contain
both digits, and Lemma~\ref{lem:boundary-incidence} shows that the
states they reach are inner edges.

\begin{definition}
\label{def:cyclic-component-limit-set}
Let \(H\leq F\) be closed, suppose that \(\mathcal L=C(H)\) is finite
and full, and let \(\mathcal C\) be a cyclic component of
\(\Gamma_{\mathrm{in}}(\mathcal L)\).  Its \emph{predecessor set} is
\[
  \operatorname{Pred}(\mathcal C)=\{e\in E(\mathcal L):e\leadsto\mathcal C\}.
\]
Thus \(\operatorname{Pred}(\mathcal C)\) may contain non-inner edges.

For \(n\geq0\), let
\[
  X_n(\mathcal C)
  =
  \bigcup_{\substack{|u|=n\\u^+\in \operatorname{Pred}(\mathcal C)}}[u],
\]
and put
\[
  X(\mathcal C)=\bigcap_{n\geq0}X_n(\mathcal C).
\]
The sets \(X_n(\mathcal C)\) are finite unions of closed standard dyadic
intervals, and
\[
  X_{n+1}(\mathcal C)\subseteq X_n(\mathcal C).
\]
Hence \(X(\mathcal C)\) is compact.
\end{definition}

\begin{lemma}
\label{lem:cyclic-component-membership}
Let \(H\), \(\mathcal L\) and \(\mathcal C\) be as in
Definition~\ref{def:cyclic-component-limit-set}.  A point
\(x\in[0,1]\) belongs to \(X(\mathcal C)\) if and only if it has a
binary expansion
\[
  x=.\!x_1x_2x_3\cdots
\]
such that
\[
  (x_1\cdots x_n)^+\in \operatorname{Pred}(\mathcal C)
  \qquad\text{for every }n\geq0.
\]
\end{lemma}

\begin{proof}
One implication follows immediately from the definition of
\(X(\mathcal C)\).  Conversely, suppose that \(x\in X(\mathcal C)\).  If \(x\) is
nondyadic or \(x\in\{0,1\}\), it has a unique binary expansion, and membership in
\(X_n(\mathcal C)\) shows that the state reached by its prefix of length \(n\)
belongs to \(\operatorname{Pred}(\mathcal C)\) for every \(n\).

If \(x\in(0,1)\) is dyadic, it has two binary expansions.  For every \(n\), at
least one of their prefixes of length \(n\) reaches a state in
\(\operatorname{Pred}(\mathcal C)\).  One of the two expansions has this property for arbitrarily
large values of \(n\).  Since every prefix of a word whose state belongs
to \(\operatorname{Pred}(\mathcal C)\) also reaches a state in \(\operatorname{Pred}(\mathcal C)\), all the prefixes of that
expansion have the required property.
\end{proof}

\begin{proposition}
\label{prop:orbit-closure-cyclic-component}
Let \(H\leq F\) be closed, suppose that \(\mathcal L=C(H)\) is finite
and full, and let \(\mathcal C\) be a cyclic component of
\(\Gamma_{\mathrm{in}}(\mathcal L)\).  Suppose that a binary expansion
of \(x\in(0,1)\) has a state path which eventually remains in \(\mathcal C\).
Then the closure in \([0,1]\) of the orbit of \(x\) under \(H\) is
\(X(\mathcal C)\); that is,
\[
  \overline{Hx}^{\,[0,1]}=X(\mathcal C).
\]
\end{proposition}

\begin{proof}
Let \(h\in H\).  The domain branches of a tree-pair diagram for \(h\)
form a finite complete prefix code.  Hence there is a unique branch pair
\[
  u\longrightarrow v
\]
for which \(u\) is a prefix of the chosen binary expansion of \(x\).
Writing
\[
  x=.\!u\xi,
\]
we have
\[
  h(x)=.\!v\xi.
\]
By Lemma~\ref{lem:branch-criterion},
\[
  u^+=v^+.
\]
The two state paths therefore coincide after the prefixes \(u\) and
\(v\).  In particular, the state path of the displayed expansion of
\(h(x)\) eventually remains in \(\mathcal C\).  Every state on that path can
consequently reach \(\mathcal C\), and hence
\[
  h(x)\in X(\mathcal C).
\]
Since \(X(\mathcal C)\) is closed,
\[
  \overline{Hx}^{\,[0,1]}\subseteq X(\mathcal C).
\]

Conversely, let \(z\in X(\mathcal C)\), and choose a binary expansion
\[
  z=.\!z_1z_2z_3\cdots
\]
such that every prefix state belongs to \(\operatorname{Pred}(\mathcal C)\).  For \(n\geq1\), put
\[
  u=z_1\cdots z_n.
\]
Since \(u^+\in \operatorname{Pred}(\mathcal C)\), there is a finite word \(r\) such that
\[
  (ur)^+=c
\]
for some state \(c\in\mathcal C\).

Choose a sufficiently long prefix \(v\) of the selected expansion of
\(x\) such that
\[
  v^+=d\in\mathcal C.
\]
Since \(\mathcal C\) is strongly connected, there is a finite word \(s\)
satisfying
\[
  \delta(c,s)=d.
\]
It follows that
\[
  (urs)^+=v^+.
\]
By Lemma~\ref{lem:branch-criterion}, some \(h_n\in H\) has the branch
pair
\[
  v\longrightarrow urs.
\]
Therefore
\[
  h_n(x)\in[urs]\subseteq[u].
\]
Both \(z\) and \(h_n(x)\) belong to the standard dyadic interval
\([u]\), whose length is \(2^{-n}\).  Hence
\[
  |h_n(x)-z|\leq2^{-n}.
\]
Thus \(h_n(x)\to z\), and so
\[
  X(\mathcal C)\subseteq\overline{Hx}^{\,[0,1]}.
\]

The two inclusions prove the proposition.
\end{proof}

\subsection{Top components and minimal sets}
\label{subsec:top-components-minimal-sets}

Throughout this subsection, let \(H\leq F\) be closed, and suppose that
its core \(\mathcal L=C(H)\) is finite and full.

A cyclic component \(\mathcal C\) of \(\Gamma_{\mathrm{in}}(\mathcal L)\) is
called \emph{top} if there is no distinct cyclic component \(\mathcal C'\) such
that \(\mathcal C'\leadsto\mathcal C\).

For every top cyclic component \(\mathcal C\), let
\[
  M_{\mathcal C}=X(\mathcal C)\cap(0,1).
\]

\begin{proposition}
\label{prop:top-component-minimal-set}
Let \(H\leq F\) be closed, suppose that \(\mathcal L=C(H)\) is finite
and full, and let \(\mathcal C\) be a top cyclic component of
\(\Gamma_{\mathrm{in}}(\mathcal L)\).  Then \(M_{\mathcal C}\) is a nonempty
minimal closed \(H\)-invariant subset of \((0,1)\).  Moreover,
\[
  X(\mathcal C)=\overline{M_{\mathcal C}}^{\,[0,1]}.
\]
\end{proposition}

\begin{proof}
Choose a state \(c\in\mathcal C\), a finite binary word \(u\) such that
\(u^+=c\), and a nonempty word \(w\) labeling a directed loop at
\(c\).  Since \(c\) is an inner edge, \(u\) contains both digits by
Lemma~\ref{lem:boundary-incidence}.  Hence the point with binary
expansion
\[
  .\!uw^\infty
\]
belongs to \((0,1)\).  Every state encountered while reading this
expansion belongs to \(\operatorname{Pred}(\mathcal C)\), so this point lies in \(M_{\mathcal C}\).  Thus
\(M_{\mathcal C}\neq\varnothing\).

Let \(x\in M_{\mathcal C}\), and choose a binary expansion of \(x\) all of whose
prefix states belong to \(\operatorname{Pred}(\mathcal C)\).  Since the child automaton is finite,
its state path eventually remains in a cyclic component \(\mathcal C'\) of
\(\Gamma_{\mathrm{in}}(\mathcal L)\).  Every state of \(\mathcal C'\) reaches
\(\mathcal C\), so \(\mathcal C'\leadsto\mathcal C\).  Since \(\mathcal C\) is top, \(\mathcal C'=\mathcal C\).
Proposition~\ref{prop:orbit-closure-cyclic-component} therefore gives
\[
  \overline{Hx}^{\,[0,1]}=X(\mathcal C).
\]
Consequently,
\[
  \overline{Hx}^{\,(0,1)}
  =X(\mathcal C)\cap(0,1)
  =M_{\mathcal C}.
\]
Thus every \(H\)-orbit in \(M_{\mathcal C}\) is dense in \(M_{\mathcal C}\), and \(M_{\mathcal C}\)
is minimal.  Since \(Hx\) is dense in \(M_{\mathcal C}\), its closure in
\([0,1]\) is \(\overline{M_{\mathcal C}}^{\,[0,1]}\).  Comparing this with the
preceding displayed equality gives
\[
  X(\mathcal C)=\overline{M_{\mathcal C}}^{\,[0,1]}.
\]
\end{proof}

\begin{proposition}
\label{prop:minimal-set-top-component}
Let \(H\leq F\) be closed, and suppose that \(\mathcal L=C(H)\) is
finite and full.  Every minimal closed \(H\)-invariant subset \(M\) of
\((0,1)\) is equal to \(M_{\mathcal C}\) for some top cyclic component \(\mathcal C\) of
\(\Gamma_{\mathrm{in}}(\mathcal L)\).  If \(M\) contains a nondyadic
point, then \(\mathcal C\) is unique.
\end{proposition}

\begin{proof}
Choose \(x\in M\) and a binary expansion of \(x\).  The corresponding
state path eventually remains in a cyclic component \(\mathcal C_0\) of
\(\Gamma_{\mathrm{in}}(\mathcal L)\).
Proposition~\ref{prop:orbit-closure-cyclic-component} gives
\[
  X(\mathcal C_0)=\overline{Hx}^{\,[0,1]}
        =\overline{M}^{\,[0,1]}.
\]
Among the cyclic components which reach \(\mathcal C_0\), choose one, say
\(\mathcal C\), which has no distinct cyclic predecessor.  Then \(\mathcal C\) is top
and \(\mathcal C\leadsto\mathcal C_0\).  Hence
\[
  \operatorname{Pred}(\mathcal C)\subseteq \operatorname{Pred}(\mathcal C_0),
  \qquad\text{and therefore}\qquad
  X(\mathcal C)\subseteq X(\mathcal C_0).
\]
It follows that \(M_{\mathcal C}\subseteq M\).  By
Proposition~\ref{prop:top-component-minimal-set}, \(M_{\mathcal C}\) is a nonempty
minimal closed \(H\)-invariant set.  The minimality of \(M\) therefore
gives \(M_{\mathcal C}=M\).

Now suppose that \(M=M_{\mathcal C}=M_{\mathcal C'}\), where \(\mathcal C\) and \(\mathcal C'\) are top cyclic
components, and let \(x\in M\) be nondyadic.  Since
\[
  x\in M_{\mathcal C}=X(\mathcal C)\cap(0,1)
  \quad\text{and}\quad
  x\in M_{\mathcal C'}=X(\mathcal C')\cap(0,1),
\]
the unique binary expansion of \(x\) has all its prefix states in both
\(\operatorname{Pred}(\mathcal C)\) and \(\operatorname{Pred}(\mathcal C')\).  Let \(\mathcal C''\) be the cyclic component in which its
state path eventually remains.  Every state of \(\mathcal C''\) belongs to
\(\operatorname{Pred}(\mathcal C)\), and hence reaches \(\mathcal C\); therefore \(\mathcal C''\leadsto\mathcal C\).
Similarly, \(\mathcal C''\leadsto\mathcal C'\).  Since \(\mathcal C\) and \(\mathcal C'\) are top, it follows
that \(\mathcal C''=\mathcal C=\mathcal C'\).
\end{proof}

We next give a finite-state criterion for deciding which top components
correspond to fixed points.  Because \(\mathcal C\) is cyclic, for every
\(e\in \operatorname{Pred}(\mathcal C)\) there is at least one digit \(i\in\{0,1\}\) such that
\(e_i\in \operatorname{Pred}(\mathcal C)\).  We say that \(\operatorname{Pred}(\mathcal C)\) is \emph{forced} if
\[
  \left|
    \left\{i\in\{0,1\}:e_i\in \operatorname{Pred}(\mathcal C)\right\}
  \right|=1
\]
for every \(e\in \operatorname{Pred}(\mathcal C)\).

\begin{lemma}
\label{lem:forced-predecessor-singleton}
Let \(\mathcal C\) be a cyclic component of
\(\Gamma_{\mathrm{in}}(\mathcal L)\).  Then \(X(\mathcal C)\) is a singleton if
and only if \(\operatorname{Pred}(\mathcal C)\) is forced.  If these conditions hold, then \(\mathcal C\)
is top and the unique point of \(X(\mathcal C)\) is rational.
\end{lemma}

\begin{proof}
Suppose first that \(\operatorname{Pred}(\mathcal C)\) is forced.  Beginning at the root edge,
there is a unique choice of a digit which keeps the state in \(\operatorname{Pred}(\mathcal C)\)
at every step.  Hence there is a unique infinite binary word all of
whose prefix states belong to \(\operatorname{Pred}(\mathcal C)\), and therefore \(X(\mathcal C)\) consists
of a single point.  Since the state sequence lies in a finite graph, it
is eventually periodic.  The corresponding binary expansion is
therefore eventually periodic, so this point is rational.

We claim that \(\mathcal C\) is top.  Otherwise, let \(\mathcal C'\neq\mathcal C\) be a cyclic
component such that \(\mathcal C'\leadsto\mathcal C\).  Choose a directed path from \(\mathcal C'\)
to \(\mathcal C\), and let
\[
  d\longrightarrow d_i
\]
be the first edge of this path which leaves \(\mathcal C'\).  Since \(d\) belongs
to the cyclic component \(\mathcal C'\), there is a digit \(j\neq i\) such that
\(d_j\in\mathcal C'\).  Now \(d_i\in \operatorname{Pred}(\mathcal C)\) because it lies on a path to \(\mathcal C\),
while \(d_j\in \operatorname{Pred}(\mathcal C)\) because \(d_j\in\mathcal C'\) and \(\mathcal C'\leadsto\mathcal C\).
Thus both digits \(i\) and \(j\) keep \(d\) inside \(\operatorname{Pred}(\mathcal C)\),
contradicting the assumption that \(\operatorname{Pred}(\mathcal C)\) is forced.  Hence \(\mathcal C\) is
top.

Conversely, suppose that \(\operatorname{Pred}(\mathcal C)\) is not forced.  Then there is some
\(e\in \operatorname{Pred}(\mathcal C)\) such that
\[
  e_0,e_1\in \operatorname{Pred}(\mathcal C).
\]
Choose a word \(u\) with \(u^+=e\).  For \(i=0,1\), choose a word
\(r_i\) taking \(e_i\) to a fixed state \(c\in\mathcal C\), extending the paths
inside \(\mathcal C\) if necessary.  Finally, choose a nonempty word \(w\)
labeling a directed loop at \(c\).  The two infinite words
\[
  u0r_0w^\infty
  \qquad\text{and}\qquad
  u1r_1w^\infty
\]
remain in \(\operatorname{Pred}(\mathcal C)\), and hence represent points of \(X(\mathcal C)\).  These
points are distinct.  Indeed, two binary expansions which first differ
by the digits \(0\) and \(1\) can represent the same point only when
their remaining tails are \(1^\infty\) and \(0^\infty\), respectively,
whereas the two words above eventually have the same periodic tail
\(w^\infty\).  Thus \(X(\mathcal C)\) is not a singleton.
\end{proof}

\begin{proposition}
\label{prop:fixed-points-cyclic-components}
Let \(H\leq F\) be closed, and suppose that \(\mathcal L=C(H)\) is
finite and full.  The fixed points of \(H\) in \((0,1)\) are precisely
the unique points of the sets \(X(\mathcal C)\), where \(\mathcal C\) ranges over the
cyclic components for which \(\operatorname{Pred}(\mathcal C)\) is forced.
\end{proposition}

\begin{proof}
Let \(\mathcal C\) be a cyclic component for which \(\operatorname{Pred}(\mathcal C)\) is forced.  By
Lemma~\ref{lem:forced-predecessor-singleton}, \(\mathcal C\) is top and
\[
  X(\mathcal C)=\{x_{\mathcal C}\}
\]
for some rational point \(x_{\mathcal C}\).
Proposition~\ref{prop:top-component-minimal-set} gives
\[
  \overline{Hx_{\mathcal C}}^{\,[0,1]}=X(\mathcal C)=\{x_{\mathcal C}\},
\]
so \(x_{\mathcal C}\) is fixed by \(H\).

Conversely, let \(x\in(0,1)\) be fixed by \(H\).  Then \(\{x\}\) is a
minimal closed \(H\)-invariant set.  By
Proposition~\ref{prop:minimal-set-top-component}, there is a top cyclic
component \(\mathcal C\) such that \(M_{\mathcal C}=\{x\}\).
Proposition~\ref{prop:top-component-minimal-set} gives
\[
  X(\mathcal C)=\overline{M_{\mathcal C}}^{\,[0,1]}=\{x\}.
\]
Hence \(\operatorname{Pred}(\mathcal C)\) is forced by
Lemma~\ref{lem:forced-predecessor-singleton}.  The same lemma also shows
directly that \(x\) is rational.
\end{proof}

The fixed points of \(H\) can therefore be found effectively from the
child automaton.  One computes its cyclic components and tests which of
them have forced predecessor sets.  By
Lemma~\ref{lem:forced-predecessor-singleton}, every such component is
automatically top.  Following the unique permitted digit eventually
repeats a state and gives the corresponding fixed point as an
eventually periodic binary expansion.

The same finite graph also detects whether the action is minimal.

\begin{theorem}
\label{thm:minimality-child-automaton}
Let \(H\leq F\) be closed, and suppose that \(\mathcal L=C(H)\) is
finite and full.  The following conditions are equivalent:
\begin{enumerate}[label=\textup{(\arabic*)}]
\item The action of \(H\) on \((0,1)\) is minimal.
\item There is a top cyclic component \(\mathcal C\) such that
      \[
        X(\mathcal C)=[0,1].
      \]
\item There is a top cyclic component \(\mathcal C\) such that every edge of
      \(\mathcal L\) belongs to \(\operatorname{Pred}(\mathcal C)\).
\item The graph \(\Gamma_{\mathrm{in}}(\mathcal L)\) has exactly one
      cyclic component.
\end{enumerate}
\end{theorem}

\begin{proof}
Suppose first that the action of \(H\) is minimal, and let \(\mathcal C\) be any
top cyclic component.  Proposition~\ref{prop:top-component-minimal-set}
shows that \(M_{\mathcal C}\) is a nonempty minimal closed \(H\)-invariant subset
of \((0,1)\).  Hence \(M_{\mathcal C}=(0,1)\), and therefore
\[
  X(\mathcal C)=\overline{M_{\mathcal C}}^{\,[0,1]}=[0,1].
\]
Thus \textup{(1)} implies \textup{(2)}.

Conversely, suppose that \textup{(2)} holds.  Then
\[
  M_{\mathcal C}=X(\mathcal C)\cap(0,1)=(0,1).
\]
By Proposition~\ref{prop:top-component-minimal-set}, \(M_{\mathcal C}\) is
minimal, so the action of \(H\) on \((0,1)\) is minimal.  Thus
\textup{(2)} implies \textup{(1)}.

Suppose that \(X(\mathcal C)=[0,1]\), and let \(e\) be an edge of
\(\mathcal L\).  Choose a finite binary word \(u\) such that
\(u^+=e\), and choose a nondyadic point in the interior of \([u]\).
Its unique binary expansion begins with \(u\).  Since the point belongs
to \(X(\mathcal C)\), every prefix state of this expansion belongs to \(\operatorname{Pred}(\mathcal C)\).
In particular, \(e\in \operatorname{Pred}(\mathcal C)\).  Thus \textup{(2)} implies
\textup{(3)}.

Conversely, if every edge of \(\mathcal L\) belongs to \(\operatorname{Pred}(\mathcal C)\), then
every binary expansion has all its prefix states in \(\operatorname{Pred}(\mathcal C)\).  Hence
every point of \([0,1]\) belongs to \(X(\mathcal C)\), and \textup{(3)} implies
\textup{(2)}.

Now suppose that \textup{(3)} holds, and let \(\mathcal C'\) be any cyclic
component of \(\Gamma_{\mathrm{in}}(\mathcal L)\).  Every state of
\(\mathcal C'\) belongs to \(\operatorname{Pred}(\mathcal C)\), so \(\mathcal C'\leadsto\mathcal C\).  Since \(\mathcal C\) is top,
\(\mathcal C'=\mathcal C\).  Thus \(\mathcal C\) is the unique cyclic component of
\(\Gamma_{\mathrm{in}}(\mathcal L)\), proving \textup{(4)}.

Finally, suppose that \(\Gamma_{\mathrm{in}}(\mathcal L)\) has a unique
cyclic component \(\mathcal C\).  Then \(\mathcal C\) is top.  Starting at any inner edge
and following child transitions indefinitely produces an infinite path
in the finite graph \(\Gamma_{\mathrm{in}}(\mathcal L)\).  Some state
repeats, so the initial edge reaches a cyclic component, which must be
\(\mathcal C\).

Every remaining edge has an inner descendant.  Indeed, if \(u^+=e\),
one may append digits so that the resulting word contains both \(0\)
and \(1\); its state is then inner by
Lemma~\ref{lem:boundary-incidence}.  Thus every edge of \(\mathcal L\)
reaches an inner edge and hence reaches \(\mathcal C\).  Therefore every edge
belongs to \(\operatorname{Pred}(\mathcal C)\), proving \textup{(3)}.
\end{proof}

Thus minimality can be decided directly from the finite child automaton
by computing the cyclic components of its inner transition graph.

We now return to the rank-two case.  Combining
Proposition~\ref{prop:top-component-minimal-set} with the rank-two
alternative established in Subsection~\ref{subsec:fixed-points-exceptional-minimal-sets}
gives the following description.

\begin{corollary}
\label{cor:unique-top-component-exceptional-set}
Suppose, in addition, that the image of \(H\) in the abelianization of
\(F\) has rank two, that \(H\) has no fixed point in \((0,1)\), and
that its action is not minimal.  Let \(M\) be its unique minimal
exceptional set.  Then \(\Gamma_{\mathrm{in}}(\mathcal L)\) has a
unique top cyclic component \(\mathcal C_*\), and
\[
  M=X(\mathcal C_*)\cap(0,1),
  \qquad
  X(\mathcal C_*)=\overline{M}^{\,[0,1]}.
\]
\end{corollary}

\begin{proof}
Every top cyclic component \(\mathcal C\) gives a minimal set
\(M_{\mathcal C}\) by
Proposition~\ref{prop:top-component-minimal-set}.  Since \(H\) has no
fixed point, \(M_{\mathcal C}\) is not a singleton, and since the
action is not minimal, \(M_{\mathcal C}\neq(0,1)\); a closed discrete orbit in \((0,1)\) is
excluded in the proof of
Proposition~\ref{prop:nonminimal-rank-two-alternative}.  Hence
\(M_{\mathcal C}\) is exceptional, and
Proposition~\ref{prop:nonminimal-rank-two-alternative} gives
\(M_{\mathcal C}=M\).  Since the exceptional set \(M\) contains nondyadic points,
Proposition~\ref{prop:minimal-set-top-component} shows that the
corresponding top cyclic component is unique.  Finally,
Proposition~\ref{prop:top-component-minimal-set} gives
\[
  X(\mathcal C_*)=\overline{M_{\mathcal C_*}}^{\,[0,1]}
        =\overline{M}^{\,[0,1]}.
\]
\end{proof}

The next subsection treats the case in which \(H\) has fixed points and
reduces the finiteness of the path sets to the corresponding questions
for its finitely many orbitals.  We shall then turn to the
fixed-point-free nonminimal case and the gaps of its exceptional
minimal set.

\subsection{Fixed points and reduction to the orbitals}
\label{subsec:fixed-points-orbital-reduction}

Throughout this subsection, let \(H\leq F\) be closed and suppose that
its core \(\mathcal L=C(H)\) is finite and full.  Assume that \(H\) has
a fixed point in \((0,1)\).  By
Proposition~\ref{prop:fixed-points-cyclic-components}, the fixed points
of \(H\) in \((0,1)\) form a finite set of rational points.  Write them
as
\[
  a_1<\cdots<a_m,
\]
and put \(a_0=0\) and \(a_{m+1}=1\).  The orbitals of \(H\) are
precisely the intervals
\[
  I_i=(a_i,a_{i+1}),
  \qquad 0\leq i\leq m.
\]
Indeed, every element of \(H\) fixes the endpoints of each \(I_i\),
while \(H\) has no fixed point in \(I_i\).  In particular, every
\(I_i\) is a proper rational \(H\)-isolated interval.

For \(0\leq i\leq m\), let
\[
  K_i=K_{I_i}(H)
  \qquad\text{and}\qquad
  \mathcal L_i=C(K_i).
\]
The action of \(K_i\) on \((0,1)\) has no fixed point.  Indeed, if
\(t\in(0,1)\) were fixed by every element of \(K_i\), then
\(\nu_{I_i}(t)\) would be fixed by every element of
\(\operatorname{Stab}_H(I_i)=H\), contrary to the fact that \(I_i\)
contains no fixed point of \(H\).  By
Proposition~\ref{prop:normalized-core}, every \(\mathcal L_i\) is
finite and full, and by
Proposition~\ref{prop:normalized-core-effective} it can be constructed
effectively from \(\mathcal L\).
Moreover, Corollary~\ref{cor:strict-inner-edge-descent} shows that
\(\mathcal L_i\) has fewer inner edges than \(\mathcal L\).

\begin{proposition}
\label{prop:fixed-points-orbital-reduction-inner}
Let \(H\leq F\) be closed, suppose that \(\mathcal L=C(H)\) is finite
and full, and suppose that \(H\) has a fixed point in \((0,1)\).  With
the notation above,
\[
  \Path_{\mathrm{in}}(\mathcal L)
\]
is finite if and only if
\[
  \Path_{\mathrm{in}}(\mathcal L_i)
\]
is finite for every \(0\leq i\leq m\).
\end{proposition}

\begin{proof}
Suppose first that \(\Path_{\mathrm{in}}(\mathcal L)\) is finite.  For
every pair \(r,s\) of inner vertices of \(\mathcal L_i\),
Proposition~\ref{prop:isolated-interval-path-correspondence} gives a
bijection
\[
  \Phi_{I_i}\colon
  \Path(\mathcal L_i)(r,s)
  \longrightarrow
  \Path_{I_i}(H)
  \bigl(\theta_{I_i}(r),\theta_{I_i}(s)\bigr).
\]
The set on the right consists of inner path classes in \(\mathcal L\),
and is therefore finite.  Since \(\mathcal L_i\) has only finitely many
inner vertices, it follows that
\(\Path_{\mathrm{in}}(\mathcal L_i)\) is finite.

Conversely, suppose that \(\Path_{\mathrm{in}}(\mathcal L_i)\) is
finite for every \(i\), and let
\[
  \mathfrak p\in\Path_{\mathrm{in}}(\mathcal L).
\]
By Proposition~\ref{prop:interval-realization}, there are dyadic points
\(x<y\) in \((0,1)\) such that
\[
  \mathfrak p=\omega_H[x,y].
\]
Choose a dyadic subdivision of \([x,y]\) fine enough that each closed
subdivision interval meets at most one fixed point of \(H\).  Mark every
subdivision interval which contains one of the points
\(a_1,\ldots,a_m\).  Each fixed point belongs to at most two marked
intervals: it either lies in the interior of one subdivision interval
or is a common endpoint of two consecutive subdivision intervals.
Hence there are at most \(2m\) marked intervals.

Expand the path defining \(\omega_H[x,y]\) according to this
subdivision.  Every marked interval contributes a single edge to the
expanded path.  Group the unmarked intervals into maximal consecutive
blocks.  The union of each such block is a closed interval with dyadic
endpoints contained in a unique orbital \(I_i\).  Consequently, the
expanded path is a concatenation of:
\begin{enumerate}[label=\textup{(\arabic*)}]
\item at most \(m+1\) path classes belonging to sets of the form
      \[
        \Path_{I_i}(H)(u,v);
      \]
\item at most \(2m\) one-edge paths in \(\mathcal L\).
\end{enumerate}

By Proposition~\ref{prop:isolated-interval-path-correspondence}, the
possible classes of the first kind correspond to path classes in
\(\Path_{\mathrm{in}}(\mathcal L_i)\).  There are only finitely many
such possibilities because the sets
\(\Path_{\mathrm{in}}(\mathcal L_i)\) are finite.  There are also only
finitely many possibilities for the one-edge factors, since
\(\mathcal L\) has finitely many edges.  The number of factors is
bounded independently of \(\mathfrak p\).  It follows that only
finitely many concatenations can occur, and hence
\(\Path_{\mathrm{in}}(\mathcal L)\) is finite.
\end{proof}

Combining Proposition~\ref{prop:fixed-points-orbital-reduction-inner}
with the boundary-path results of Subsection~\ref{subsec:inner-paths}
gives the formulation for all nonempty path classes which will be used
below.

\begin{corollary}
\label{cor:fixed-points-orbital-reduction}
Let \(H\leq F\) be closed, suppose that \(\mathcal L=C(H)\) is finite
and full, and suppose that \(H\) has a fixed point in \((0,1)\).  Let
\[
  a_1<\cdots<a_m
\]
be the fixed points of \(H\) in \((0,1)\), put \(a_0=0\) and
\(a_{m+1}=1\), and, for \(0\leq i\leq m\), let
\[
  I_i=(a_i,a_{i+1}),
  \qquad
  \mathcal L_i=C\bigl(K_{I_i}(H)\bigr).
\]
Then \(\Path(\mathcal L)\) is finite if and only if
\(\Path(\mathcal L_i)\) is finite for every \(0\leq i\leq m\).
\end{corollary}

\begin{proof}
Since \(\Path_{\mathrm{in}}(\mathcal L)\subseteq\Path(\mathcal L)\),
finiteness of \(\Path(\mathcal L)\) implies that of
\(\Path_{\mathrm{in}}(\mathcal L)\), and the converse holds by
Lemma~\ref{lem:inner-reduction}\textup{(2)}, since \(\mathcal L\) is
finite.  The same applies to each \(\mathcal L_i\), which is finite by
Proposition~\ref{prop:normalized-core}.  The result therefore follows
from Proposition~\ref{prop:fixed-points-orbital-reduction-inner}.
\end{proof}

Thus, when \(H\) has fixed points, the finiteness problem reduces to the
finitely many normalized groups \(K_i\), whose cores have fewer inner
edges.  Each \(K_i\) acts without fixed points.  If its image in the
abelianization has rank one, the results of Section~\ref{sec:rational-isolated-intervals}
apply; if its image has rank two, its action is either minimal or has an
exceptional minimal set.  The latter case is considered in the next
subsection.

\subsection{The exceptional minimal set and its gaps}
\label{subsec:exceptional-minimal-set-gaps}

Throughout this subsection, let \(H\leq F\) be closed, suppose that
its core \(\mathcal L=C(H)\) is finite and full, and suppose that
\(\pi_{\mathrm{ab}}(H)\) has rank two.  Assume that \(H\) has no fixed
point in \((0,1)\) and that its action is not minimal.  Let \(M\) be
its unique exceptional minimal set.

By Corollary~\ref{cor:unique-top-component-exceptional-set}, the inner
child graph has a unique top cyclic component \(\mathcal C_*\), and
\[
  M=X(\mathcal C_*)\cap(0,1),
  \qquad
  X(\mathcal C_*)=\overline{M}^{\,[0,1]}.
\]
Put
\[
  \operatorname{Pred}_*=\operatorname{Pred}(\mathcal C_*).
\]
Since \(\inf M=0\) and \(\sup M=1\), every gap \(I=(a,b)\) of \(M\)
satisfies
\[
  0<a<b<1,
  \qquad
  a,b\in M.
\]
We first construct a finite collection of gaps meeting every
\(H\)-orbit of gaps.

For \(e\in \operatorname{Pred}_*\), a \emph{continuation from \(e\) in \(\operatorname{Pred}_*\)} is an
infinite binary word such that, when the word is read starting at
\(e\), every state reached belongs to \(\operatorname{Pred}_*\).  Every state in \(\operatorname{Pred}_*\)
has at least one child in \(\operatorname{Pred}_*\), so such continuations exist.

Define \(\lambda(e)\) as follows.  Starting at \(e\), choose \(0\) if
the left child belongs to \(\operatorname{Pred}_*\), and choose \(1\) otherwise.  Move
to the chosen child and repeat the same rule.  Define \(\rho(e)\)
symmetrically, choosing \(1\) whenever the right child belongs to
\(\operatorname{Pred}_*\), and choosing \(0\) otherwise.

By construction, \(\lambda(e)\) is the lexicographically least
continuation from \(e\) in \(\operatorname{Pred}_*\), and \(\rho(e)\) is the greatest.
Both words are eventually periodic.  Indeed, in each construction the
next digit and state are determined by the current state.  Since there
are finitely many states, a state eventually repeats, and all
subsequent choices then repeat periodically.

For every \(e\in \operatorname{Pred}_*\), choose a finite binary word \(w_e\) such that
\[
  w_e^+=e.
\]
Whenever both children \(e_0,e_1\) belong to \(\operatorname{Pred}_*\), put
\[
  a_e=.\!w_e0\,\rho(e_0),
  \qquad
  b_e=.\!w_e1\,\lambda(e_1).
\]
These are rational numbers, and \(a_e\leq b_e\).  The following
proposition describes when they determine a gap.

\begin{proposition}
\label{prop:gap-orbit-representatives}
Let \(e\in \operatorname{Pred}_*\), and suppose that \(e_0,e_1\in \operatorname{Pred}_*\).  If
\(a_e=b_e\), the two expansions defining these numbers are the two
binary expansions of one dyadic point, and they do not determine a
gap.  If \(a_e<b_e\), then
\[
  I_e=(a_e,b_e)
\]
is a gap of \(M\).

Moreover, the finite collection
\[
  \mathcal I
  =
  \left\{
    I_e:
    e\in \operatorname{Pred}_*,\ e_0,e_1\in \operatorname{Pred}_*,\ a_e<b_e
  \right\}
\]
meets every \(H\)-orbit of gaps and can be constructed effectively
from the child automaton of \(\mathcal L\).  Every gap of \(M\) has
rational endpoints.
\end{proposition}

\begin{proof}
The two expansions defining \(a_e\) and \(b_e\), when read from the
root, stay in \(\operatorname{Pred}_*\).  Hence
Lemma~\ref{lem:cyclic-component-membership} gives
\[
  a_e,b_e\in X(\mathcal C_*).
\]
If \(a_e=b_e\), their distinct binary expansions represent the same
number, which is therefore dyadic.

Suppose that \(a_e<b_e\).  There is no infinite binary word strictly
between
\[
  w_e0\rho(e_0)
  \quad\text{and}\quad
  w_e1\lambda(e_1)
\]
in lexicographic order which, when read from the root, stays in
\(\operatorname{Pred}_*\).  Indeed, a word between them must begin with \(w_e\).  If its
next digit is \(0\), its remaining suffix would be a continuation from
\(e_0\) greater than \(\rho(e_0)\).  If its next digit is \(1\), its
remaining suffix would be a continuation from \(e_1\) smaller than
\(\lambda(e_1)\).  Both are impossible.

By Lemma~\ref{lem:cyclic-component-membership}, it follows that
\[
  X(\mathcal C_*)\cap(a_e,b_e)=\varnothing.
\]
Moreover \(a_e>0\): otherwise \(M\cap(0,b_e)=\varnothing\),
contradicting \(\inf M=0\).  Symmetrically \(b_e<1\).  Hence
\(a_e,b_e\in M\), and \(I_e=(a_e,b_e)\) is a gap of \(M\).

Conversely, let \(I=(a,b)\) be a gap of \(M\).  Let \(\xi\) be the
lexicographically greatest of the at most two binary expansions of
\(a\) which stay in \(\operatorname{Pred}_*\) when read from the root, and let \(\eta\)
be the least such expansion of \(b\).  These expansions exist by
Lemma~\ref{lem:cyclic-component-membership}.

There is no infinite binary word strictly between \(\xi\) and \(\eta\)
which stays in \(\operatorname{Pred}_*\) when read from the root.  Indeed, its value
would belong to
\[
  X(\mathcal C_*)\cap(a,b),
\]
contradicting the fact that \((a,b)\) is a gap of \(M\).

Let \(u\) be the longest common prefix of \(\xi\) and \(\eta\).  Then
\[
  \xi=u0\alpha,
  \qquad
  \eta=u1\beta.
\]
Put \(e=u^+\).  Both children of \(e\) belong to \(\operatorname{Pred}_*\).  The absence
of an intermediate word implies that
\[
  \alpha=\rho(e_0),
  \qquad
  \beta=\lambda(e_1).
\]
Otherwise, replacing \(\alpha\) by \(\rho(e_0)\), or replacing
\(\beta\) by \(\lambda(e_1)\), would produce such a word.  Therefore
\[
  a=.\!u0\,\rho(e_0),
  \qquad
  b=.\!u1\,\lambda(e_1).
\]
Since \(\rho(e_0)\) and \(\lambda(e_1)\) are eventually periodic,
both \(a\) and \(b\) are rational.

Since \(u^+=w_e^+\), Lemma~\ref{lem:branch-criterion} gives an element
\(h\in H\) with the branch pair
\[
  u\longrightarrow w_e.
\]
This element replaces the prefix \(u\) by \(w_e\) and preserves the
remaining suffix.  Hence
\[
  h(a)=a_e,
  \qquad
  h(b)=b_e,
\]
so
\[
  h(I)=I_e\in\mathcal I.
\]
Thus \(\mathcal I\) meets every \(H\)-orbit of gaps.

The representative words and the eventual periods of \(\lambda(e)\)
and \(\rho(e)\) can be found in the finite child automaton.  Comparing
the resulting rational endpoints therefore constructs \(\mathcal I\)
effectively.
\end{proof}

The collection \(\mathcal I\) may contain several gaps from the same
orbit; this will not affect the arguments below.  By
Proposition~\ref{prop:nonminimal-rank-two-alternative}, every gap is
\(H\)-isolated.  Proposition~\ref{prop:gap-orbit-representatives}
therefore shows that every gap is a proper rational \(H\)-isolated
interval.

We shall show that \(\Path(\mathcal L)\) is finite if and only if the
path semigroupoids of the normalized cores associated with the gaps in
\(\mathcal I\) are finite.  For this, we separate the inner path
classes according to whether their representing intervals meet \(M\)
in their interiors.  An inner path class \(W\) is called
\emph{essential} if
\[
  W=\omega_H[a,b]
\]
for dyadic points \(a<b\) such that
\[
  M\cap(a,b)\neq\varnothing.
\]
This is independent of the representing interval: by
Proposition~\ref{prop:exact-transport}, two intervals representing the
same class are carried onto one another by an element of \(H\), and
every element of \(H\) preserves \(M\).

A nonessential class is represented by an interval \([a,b]\) such that
\((a,b)\) lies in a single gap of \(M\).  We first show that there are
only finitely many nonessential classes if and only if every normalized
core associated with a gap in \(\mathcal I\) has finitely many path
classes.

Write
\[
  \mathcal I=\{I_1,\ldots,I_k\},
\]
and put
\[
  K_i=K_{I_i}(H),
  \qquad
  \mathcal L_i=C(K_i).
\]
By Proposition~\ref{prop:normalized-core}, each \(\mathcal L_i\) is
finite and full, and by
Proposition~\ref{prop:normalized-core-effective} it can be constructed
effectively from \(\mathcal L\).  By Corollary~\ref{cor:strict-inner-edge-descent}, it
has fewer inner edges than \(\mathcal L\).

\begin{proposition}
\label{prop:nonessential-gap-reduction}
The set of nonessential inner path classes differs from
\[
  \bigcup_{i=1}^{k}\Path_{I_i}(H)
\]
by a finite set.  Consequently, there are only finitely many
nonessential inner path classes if and only if \(\Path(\mathcal L_i)\)
is finite for every \(1\leq i\leq k\).
\end{proposition}

\begin{proof}
Every class in the displayed union is nonessential, since it is
represented by an interval contained in a gap.

Conversely, let
\[
  W=\omega_H[a,b]
\]
be a nonessential inner path class.  Then
\[
  M\cap(a,b)=\varnothing,
\]
so \((a,b)\) lies in a unique gap \(I\), and
\[
  [a,b]\subseteq\overline I.
\]
By Proposition~\ref{prop:gap-orbit-representatives}, some element of
\(H\) carries \(I\) onto one of the representative gaps \(I_i\).
Applying that element does not change \(W\).  Thus it is enough to
consider intervals with dyadic endpoints contained in the closure of
a fixed representative gap
\[
  I_i=(\alpha_i,\beta_i).
\]
If both endpoints lie in \(I_i\), the interval class belongs to
\(\Path_{I_i}(H)\).

Suppose that the left endpoint is \(\alpha_i\), which must then be
dyadic, and the right endpoint lies in \(I_i\).  If
\(b,b'\in\D\cap I_i\) have the same type in \(\mathcal L\),
Proposition~\ref{prop:orbit-core} gives \(h\in H\) with
\[
  h(b)=b'.
\]
Since \(h(I_i)\cap I_i\neq\varnothing\) and \(I_i\) is
\(H\)-isolated, this element stabilizes \(I_i\).  It therefore fixes
\(\alpha_i\), and Proposition~\ref{prop:exact-transport} gives
\[
  \omega_H[\alpha_i,b]
  =
  \omega_H[\alpha_i,b'].
\]
There are only finitely many dyadic types, so only finitely many
classes of this form occur.  The case in which the right endpoint is
\(\beta_i\) and the left endpoint lies in \(I_i\) is symmetric.

If both endpoints are endpoints of \(I_i\), there is at most one
interval to consider, namely
\[
  [\alpha_i,\beta_i],
\]
and it occurs only when both endpoints are dyadic.

Thus each representative gap contributes only finitely many classes
beyond those represented by intervals contained in the gap.  Since
there are finitely many representative gaps, the first assertion
follows.

By Proposition~\ref{prop:isolated-interval-path-correspondence}, the
isolated-interval bijections combine to give a bijection
\[
  \Path_{\mathrm{in}}(\mathcal L_i)
  \longrightarrow
  \Path_{I_i}(H)
\]
for every \(i\).  Thus
\[
  \bigcup_{i=1}^{k}\Path_{I_i}(H)
\]
is finite if and only if \(\Path_{\mathrm{in}}(\mathcal L_i)\) is
finite for every \(i\).  Lemma~\ref{lem:inner-reduction}\textup{(2)},
applied to each \(\mathcal L_i\), gives the final assertion.
\end{proof}

Next we prove that the essential inner path classes form a finite set.
We follow the argument of Subsection~\ref{subsec:stabilizer-orbitals},
with the necessary adaptations.

Recall that
\[
  \inf M=0,
  \qquad
  \sup M=1.
\]
Thus, for every \(x\in(0,1)\), the sets \(M\cap(0,x)\) and
\(M\cap(x,1)\) are nonempty.  Put
\[
  x_-=\sup\bigl(M\cap(0,x)\bigr),
  \qquad
  x_+=\inf\bigl(M\cap(x,1)\bigr).
\]
These points belong to \(M\) and satisfy
\[
  0<x_-\leq x\leq x_+<1.
\]
By their definitions,
\[
  M\cap(x_-,x)=M\cap(x,x_+)=\varnothing.
\]
We use the pointwise stabilizers \(H_E\) introduced in
Subsection~\ref{subsec:stabilizer-orbitals}.

\begin{lemma}
\label{lem:exceptional-slopes-fixed-point-transfer}
The following statements hold.
\begin{enumerate}[label=\textup{(\arabic*)}]
\item For every \(a\in(0,1)\), the subgroup \(H_{[a,1]}\) contains an
      element \(g\) with \(g'(0^+)\neq1\), and \(H_{[0,a]}\) contains
      an element \(g\) with \(g'(1^-)\neq1\).
\item If \(H_{[a,1]}\) fixes a point in \((0,a)\), then \(H_a\) fixes
      a point in \((0,a)\).  Similarly, for \(b\in(0,1)\), if
      \(H_{[0,b]}\) fixes a point in \((b,1)\), then \(H_b\) fixes a
      point in \((b,1)\).
\end{enumerate}
\end{lemma}

\begin{proof}
The proof of Lemma~\ref{lem:one-sided-slopes} applies to the present
action, since \((0,1)\) is an orbital of \(H\), and gives
\textup{(1)}.  The proof of Lemma~\ref{lem:fixed-point-transfer}, using
\textup{(1)} in place of Lemma~\ref{lem:one-sided-slopes}, gives
\textup{(2)}.
\end{proof}

\begin{proposition}[Stabilizer orbitals at points of \(M\)]
\label{prop:exceptional-stabilizer-orbitals}
If \(s\in M\) is accumulated by \(M\) from the left, then \((0,s)\)
is an orbital of \(H_{[s,1]}\).  If \(s\) is accumulated by \(M\)
from the right, then \((s,1)\) is an orbital of \(H_{[0,s]}\).
\end{proposition}

\begin{proof}
We prove the first assertion.  Fix \(a_0\in(0,1)\) and put
\[
  B=H_{[a_0,1]}.
\]
The group \(B\) fixes \(a_0\), and
Lemma~\ref{lem:exceptional-slopes-fixed-point-transfer}\textup{(1)}
shows that there is a right neighborhood of \(0\) containing no fixed
point of \(B\) other than \(0\).  Its fixed-point set in \((0,1)\)
therefore has a least element \(c\), with
\[
  0<c\leq a_0.
\]
Since \(B\leq H_c\) and \(B\) has no fixed point in \((0,c)\), neither
does \(H_c\).  The contrapositive of
Lemma~\ref{lem:exceptional-slopes-fixed-point-transfer}\textup{(2)},
applied at \(c\), shows that \(H_{[c,1]}\) has no fixed point in
\((0,c)\).  Hence \((0,c)\) is an orbital of \(H_{[c,1]}\).

We claim that \(c\in M\) and that \(M\) accumulates at \(c\) from the
left.  If \(c\) belonged to a gap, every element of \(H_{[c,1]}\)
would stabilize that gap and fix its left endpoint, which lies in
\((0,c)\).  If \(c\in M\) were not accumulated from the left, it would
be the right endpoint of a gap, and every element of \(H_{[c,1]}\)
would again fix the left endpoint of that gap.  Both possibilities
contradict the fact that \((0,c)\) is an orbital.

Now let \(s\in M\) be accumulated from the left, and let
\(y\in(0,s)\).  Since \(Hc\) is dense in \(M\), there is \(h\in H\)
such that
\[
  y<h(c)<s.
\]
Conjugating the orbital obtained above, we find that \((0,h(c))\) is
an orbital of
\[
  h^{-1}H_{[c,1]}h=H_{[h(c),1]}.
\]
Since
\[
  H_{[h(c),1]}\leq H_{[s,1]},
\]
some element of \(H_{[s,1]}\) moves \(y\).  Thus \(H_{[s,1]}\) has no
fixed point in \((0,s)\), proving that this interval is an orbital.

The second assertion follows by the symmetric argument.
\end{proof}

\begin{corollary}[Endpoint cofinality]
\label{cor:exceptional-endpoint-cofinality}
For every \(x\in(0,1)\), the interval \((0,x_-)\) is an orbital of
\(H_{[x,1]}\), and \((x_+,1)\) is an orbital of \(H_{[0,x]}\).
Consequently, if \(0<a<b<1\) and \(M\cap(a,b)\neq\varnothing\), then
\[
\begin{aligned}
  \inf H_{[0,a]}b&=a_+,
  &
  \sup H_{[0,a]}b&=1,
  \\
  \inf H_{[b,1]}a&=0,
  &
  \sup H_{[b,1]}a&=b_-.
\end{aligned}
\]
\end{corollary}

\begin{proof}
The point \(x_-\) is accumulated by \(M\) from the left.  If
\(x_-=x\), this follows from its definition.  If \(x_-<x\), there are
no points of \(M\) immediately to its right, so perfectness gives
accumulation from the left.
Proposition~\ref{prop:exceptional-stabilizer-orbitals} therefore shows
that \((0,x_-)\) is an orbital of
\[
  H_{[x_-,1]}\leq H_{[x,1]}.
\]
Every element of \(H_{[x,1]}\) fixes \(x_-\), since it preserves
\(M\cap(0,x)\) and its supremum.  Thus \((0,x_-)\) is also an orbital
of \(H_{[x,1]}\).  The argument for \((x_+,1)\) is symmetric.

If \(M\cap(a,b)\neq\varnothing\), then
\[
  a_+<b,
  \qquad
  a<b_-.
\]
Hence \(b\) lies in the orbital \((a_+,1)\) of \(H_{[0,a]}\), and
\(a\) lies in the orbital \((0,b_-)\) of \(H_{[b,1]}\).  The displayed
equalities follow from the infimum and supremum property of orbitals.
\end{proof}

For inner vertices \(u,v\), let
\[
  \mathcal E(u,v)
  =
  \left\{
    W\in\Path(\mathcal L)(u,v):
    W\text{ is essential}
  \right\}.
\]

\begin{proposition}[The essential groupoid]
\label{prop:essential-groupoid}
Let
\[
  U\in\Path(\mathcal L)(u,v),
  \qquad
  V\in\Path(\mathcal L)(v,w)
\]
be inner path classes.  If either \(U\) or \(V\) is essential, then
\(UV\) is essential.  Moreover, the essential inner path classes form
a groupoid whose objects are the inner vertices of \(\mathcal L\), and
\[
  \mathcal E(u,v)\neq\varnothing
\]
for every pair of inner vertices \(u,v\).
\end{proposition}

\begin{proof}
We first describe how essentiality behaves under concatenation.  Let
\[
  U\in\Path(\mathcal L)(u,v),
  \qquad
  V\in\Path(\mathcal L)(v,w),
\]
and choose representing intervals
\[
  U=\omega_H[x_0,x_1],
  \qquad
  V=\omega_H[y_0,y_1].
\]
The points \(x_1\) and \(y_0\) have the same type \(v\), so
Proposition~\ref{prop:orbit-core} gives \(h\in H\) with
\[
  h(y_0)=x_1.
\]
By Proposition~\ref{prop:transport} and Lemma~\ref{lem:cutting},
\[
\begin{aligned}
  UV
  &=
  \omega_H[x_0,x_1]\,
  \omega_H[x_1,h(y_1)]
  \\
  &=
  \omega_H[x_0,h(y_1)].
\end{aligned}
\]
If \(U\) is essential, then \(M\) meets \((x_0,x_1)\).  If \(V\) is
essential, then \(M\) meets \((x_1,h(y_1))\), since \(h\) preserves
\(M\).  In either case, \(UV\) is essential.  Thus a product of two
composable inner classes is essential whenever either factor is
essential.  In particular, the essential classes form a semigroupoid.

We next show that every \(\mathcal E(u,v)\) is nonempty.  Choose dyadic
points \(x,y\) of types \(u,v\), respectively, and choose \(z\in M\)
with \(z>x\).  Since \(\sup Hy=1\), there is \(h\in H\) with
\(h(y)>z\).  Then
\[
  \omega_H[x,h(y)]\in\mathcal E(u,v).
\]

Next we prove that the semigroupoid of essential inner path classes is
regular.  Take an arbitrary
\[
  W\in\mathcal E(u,v),
\]
and choose a dyadic interval \([a,b]\) representing \(W\):
\[
  W=\omega_H[a,b].
\]
Since \(M\) has no isolated points, the nonempty set \(M\cap(a,b)\)
contains three distinct points
\[
  z_1<z<z_2.
\]
By the definitions of \(a_+\) and \(b_-\),
\[
  a_+\leq z_1<z<z_2\leq b_-.
\]
In particular,
\[
  a_+<z<b_-.
\]
Corollary~\ref{cor:exceptional-endpoint-cofinality} gives
\[
  \inf H_{[0,a]}b=a_+,
  \qquad
  \sup H_{[b,1]}a=b_-.
\]
Choose
\[
  h_1\in H_{[0,a]},
  \qquad
  h_2\in H_{[b,1]}
\]
so that, on putting
\[
  c=h_1(b),
  \qquad
  d=h_2(a),
\]
we have
\[
  a\leq a_+<c<z<d<b_-\leq b.
\]
The points \(c,d\) are dyadic.  Since \(h_1\) fixes \(a\) and
\(h_2\) fixes \(b\), Proposition~\ref{prop:transport} gives
\[
  W=\omega_H[a,c]=\omega_H[d,b].
\]
Put
\[
  X=\omega_H[c,d].
\]
The preceding equalities show that \(c\) has type \(v\) and \(d\) has
type \(u\), and \(X\) is essential because \(z\in M\cap(c,d)\).  Thus
\[
  X\in\mathcal E(v,u).
\]
Cutting at \(c\) and \(d\), we obtain
\[
  W
  =
  \omega_H[a,c]\omega_H[c,d]\omega_H[d,b]
  =
  WXW.
\]
Hence the semigroupoid of essential classes is regular.

Finally, let
\[
  U,V\in\mathcal E(u,u).
\]
We show that \(U\) is both a left divisor and a right divisor of \(V\)
within \(\mathcal E(u,u)\).  Write
\[
  V=\omega_H[a,b].
\]
The left endpoint of a representative of \(U\) has the same type as
\(a\).  Using Propositions~\ref{prop:orbit-core}
and~\ref{prop:transport}, we may therefore represent \(U\) as
\[
  U=\omega_H[a,t].
\]
Since \(U\) and \(V\) are essential,
\[
  t>a_+,
  \qquad
  b>a_+.
\]
As above, perfectness of \(M\) allows us to choose
\[
  z\in M\cap\bigl(a_+,\min\{b,t\}\bigr).
\]
Corollary~\ref{cor:exceptional-endpoint-cofinality} gives
\[
  \inf H_{[0,a]}t=a_+.
\]
Choose \(h\in H_{[0,a]}\) such that
\[
  a_+<h(t)<z<\min\{b,t\},
\]
and put \(t'=h(t)\).  Then
\[
  U=\omega_H[a,t']
\]
and
\[
  V=U\,\omega_H[t',b].
\]
The second factor is an essential loop at \(u\), since its interval
contains \(z\) in its interior.  Thus \(U\) is a left divisor of \(V\)
in \(\mathcal E(u,u)\).

The symmetric argument, aligning the right endpoints and using the
pointwise stabilizer of the interval to their right, shows that \(U\)
is also a right divisor of \(V\).

All the hypotheses of Lemma~\ref{lem:local-cancellation-criterion} are
now satisfied.  Hence the essential classes form a groupoid.
\end{proof}

For every inner vertex \(u\), denote the identity of this groupoid at
\(u\) by
\[
  e_u\in\mathcal E(u,u).
\]
Retain the notation
\[
  d_L,d_R,
  \quad N_L,N_R,
  \quad \ell,r,
  \quad v_L,v_R,
  \quad \Lambda_L
\]
from Subsection~\ref{subsec:boundary-periods}, and put
\[
  A=\pi_{\mathrm{ab}}(H),
  \qquad
  q=\bigl[d_L\mathbb Z\times d_R\mathbb Z:A\bigr].
\]
As in Subsection~\ref{subsec:exact-inner-path-count}, \(q\) is the
least positive integer \(t\) such that
\[
  (td_L,0)\in A.
\]

\begin{proposition}[Exact number of essential path classes]
\label{prop:essential-exact-count}
For every pair of inner vertices \(u,v\) of \(\mathcal L\),
\[
  |\mathcal E(u,v)|=q.
\]
In particular, there are only finitely many essential inner path
classes.
\end{proposition}

\begin{proof}
Since \((qd_L,0)\in A\), Proposition~\ref{prop:endpoint-relations}
gives a nonempty inner \(1\)-path
\[
  Z\colon v_L\longrightarrow v_R
\]
such that
\[
  \Lambda_L^q[Z]=[Z].
\]
By Proposition~\ref{prop:essential-groupoid}, the class
\[
  Q=[Z]e_{v_R}
\]
is essential, and it satisfies
\[
  \Lambda_L^qQ=Q.
\]
Multiplying on the right by the inverse of \(Q\) in the essential
groupoid gives
\[
  \Lambda_L^q e_{v_L}=e_{v_L}.
\]
Consequently, for every \(W\in\mathcal E(v_L,v)\),
\[
\begin{aligned}
  \Lambda_L^qW
  &=\Lambda_L^q e_{v_L}W
  \\
  &=e_{v_L}W
  \\
  &=W.
\end{aligned}
\]

To apply the converse part of Proposition~\ref{prop:endpoint-relations},
we next choose a nonempty inner \(1\)-path
\[
  Y\colon v_L\longrightarrow v_R
\]
such that \([Y]\) is essential and there is a positive diagram
\[
  p_H\longrightarrow_{\mathcal L}\ell Yr.
\]
Choose sufficiently large integers
\[
  n\equiv N_L\pmod{d_L},
  \qquad
  m\equiv N_R\pmod{d_R},
\]
such that
\[
  M\cap(2^{-n},1-2^{-m})\neq\varnothing.
\]
A binary tree whose extreme leaves are \(0^n\) and \(1^m\) gives a
positive diagram
\[
  p_H\longrightarrow_{\mathcal L}\ell Yr,
\]
where
\[
  [Y]=\omega_H[2^{-n},1-2^{-m}]
  \in\mathcal E(v_L,v_R).
\]
Here the labels of the extreme edges are \(\ell,r\) by
Lemma~\ref{lem:boundary-periodicity}.

Fix an inner vertex \(v\) and choose
\[
  W_0\in\mathcal E(v_L,v).
\]
We claim that the \(q\) classes
\[
  W_0,\ \Lambda_LW_0,\ \ldots,\ \Lambda_L^{q-1}W_0
\]
are distinct.  Throughout this argument, \(\Lambda_L^0W\) means
\(W\).  Suppose that
\[
  \Lambda_L^iW_0=\Lambda_L^jW_0
\]
for some \(0\leq i<j<q\).  Multiplying on the right by
\(W_0^{-1}[Y]\), we obtain
\[
  \Lambda_L^i[Y]=\Lambda_L^j[Y].
\]
The converse part of Proposition~\ref{prop:endpoint-relations} now
gives
\[
  ((i-j)d_L,0)\in A.
\]
Hence \(((j-i)d_L,0)\in A\), contradicting the choice of \(q\).  This
proves the claim.

Now let \(W\in\mathcal E(v_L,v)\).  By
Lemma~\ref{lem:boundary-comparison}, there are positive integers \(i,j\)
such that
\[
  \Lambda_L^iW=\Lambda_L^jW_0.
\]
Choose a positive integer \(k\) such that \(k+i\) is divisible by
\(q\).  Multiplying on the left by \(\Lambda_L^k\) and using the
periodicity proved above gives
\[
  W=\Lambda_L^tW_0
\]
for some \(t\in\{0,\ldots,q-1\}\).  Therefore
\[
  |\mathcal E(v_L,v)|=q.
\]

Finally, let \(u,v\) be arbitrary inner vertices.  Choose
\[
  U\in\mathcal E(v_L,u).
\]
Left multiplication by \(U\) defines a bijection
\[
  \mathcal E(u,v)
  \longrightarrow
  \mathcal E(v_L,v),
  \qquad
  W\longmapsto UW,
\]
whose inverse is left multiplication by \(U^{-1}\).  Hence
\[
  |\mathcal E(u,v)|=q.
\]
Since \(\mathcal L\) has finitely many inner vertices, the set of
essential inner path classes is finite.
\end{proof}

We can now characterize finiteness of \(\Path(\mathcal L)\) in terms
of the normalized cores associated with the finitely many gaps in
\(\mathcal I\).

\begin{corollary}
\label{cor:exceptional-gap-reduction}
Let \(H\leq F\) be closed, suppose that \(\mathcal L=C(H)\) is finite
and full and that \(\pi_{\mathrm{ab}}(H)\) has rank two, and assume
that \(H\) has no fixed point in \((0,1)\) and does not act minimally.
Let
\[
  I_1,\ldots,I_k
\]
be the gaps in the finite collection of
Proposition~\ref{prop:gap-orbit-representatives}, and put
\[
  \mathcal L_i=C\bigl(K_{I_i}(H)\bigr).
\]
Then
\[
  \Path(\mathcal L)\text{ is finite}
  \quad\Longleftrightarrow\quad
  \Path(\mathcal L_i)\text{ is finite for every }1\leq i\leq k.
\]
Moreover, the intervals \(I_i\) and the cores \(\mathcal L_i\) can be
constructed effectively from \(\mathcal L\).  Each \(I_i\) is a
proper rational \(H\)-isolated interval, and each \(\mathcal L_i\) is
finite and full and has fewer inner edges than \(\mathcal L\).
\end{corollary}

\begin{proof}
By Lemma~\ref{lem:inner-reduction}\textup{(2)}, finiteness of
\(\Path(\mathcal L)\) is equivalent to finiteness of
\(\Path_{\mathrm{in}}(\mathcal L)\).  The essential inner path classes
form a finite set by Proposition~\ref{prop:essential-exact-count}.
Hence \(\Path_{\mathrm{in}}(\mathcal L)\) is finite if and only if the
nonessential classes form a finite set.  By
Proposition~\ref{prop:nonessential-gap-reduction}, this is equivalent
to finiteness of every \(\Path(\mathcal L_i)\).

The remaining assertions follow from
Proposition~\ref{prop:gap-orbit-representatives},
Proposition~\ref{prop:normalized-core},
Proposition~\ref{prop:isolated-interval-path-correspondence}, and
Corollary~\ref{cor:strict-inner-edge-descent}.
\end{proof}

Thus, for a nonminimal rank-two action, determining whether \(C(H)\)
has finitely many path classes reduces to the same question for
finitely many effectively constructed finite full cores with fewer
inner edges.  If \(H\) has fixed points, these are the normalized cores
associated with its orbitals, as in
Corollary~\ref{cor:fixed-points-orbital-reduction}.  Otherwise, they are
the normalized cores associated with the gaps in \(\mathcal I\), as in
Corollary~\ref{cor:exceptional-gap-reduction}.

\section{Finiteness properties of groups with finite full core}
\label{sec:finiteness-properties-finite-full-core}

\subsection{Iterated restrictions and rank-one obstructions}
\label{subsec:iterated-restrictions-rank-one-obstructions}

Throughout this subsection, let \(H\leq F\) be closed and suppose that
its core \(\mathcal L=C(H)\) is finite and full.

The reductions in Section~\ref{sec:nonminimal-rank-two} replace \(H\)
by normalized restrictions to proper rational \(H\)-isolated intervals.
Repeating these reductions may produce a group with rank-one image in
the abelianization of \(F\).  To apply the obstruction from
Section~\ref{sec:rational-isolated-intervals} to \(H\), we first show
that such a group corresponds to a normalized restriction of \(H\) to
a single rational \(H\)-isolated interval.

\begin{lemma}[Iterated normalized restrictions]
\label{lem:iterated-normalized-restrictions}
Let \(I\) be a rational \(H\)-isolated interval, put
\[
  G=K_I(H),
\]
and let \(J\) be a rational \(G\)-isolated interval.  Then
\[
  I'=\nu_I(J)
\]
is a rational \(H\)-isolated interval.

Moreover, the increasing homeomorphism
\[
  \phi=\nu_J(\nu_I|_J)\nu_{I'}^{-1}
  \colon(0,1)\longrightarrow(0,1),
\]
extended continuously to \([0,1]\), satisfies
\[
  \phi^{-1}K_J(G)\phi=K_{I'}(H).
\]
In particular,
\[
  \operatorname{rank}\pi_{\mathrm{ab}}\bigl(K_J(G)\bigr)
  =
  \operatorname{rank}\pi_{\mathrm{ab}}\bigl(K_{I'}(H)\bigr).
\]
\end{lemma}

\begin{proof}
Write
\[
  I=(a,b),
  \qquad
  J=(c,d),
\]
where
\[
  0\leq c<d\leq1.
\]
Using the continuous extension of \(\nu_I\) to the endpoints, we have
\[
  I'=\bigl(\nu_I(c),\nu_I(d)\bigr).
\]
The canonical chart maps rational points of \((0,1)\) to rational
points of \(I\), and
\[
  \nu_I(0)=a,
  \qquad
  \nu_I(1)=b.
\]
Thus \(I'\) has rational endpoints.

We next prove that \(I'\) is \(H\)-isolated.  Suppose that \(h\in H\)
satisfies
\[
  h(I')\cap I'\neq\varnothing.
\]
Since \(I'\subseteq I\), it follows that
\[
  h(I)\cap I\neq\varnothing.
\]
The interval \(I\) is \(H\)-isolated, so \(h(I)=I\).  Hence its
normalized restriction
\[
  g=\nu_I(h|_I)\nu_I^{-1}
\]
belongs to \(G\).  Since \(\nu_I(J)=I'\), the assumed intersection
gives
\[
  g(J)\cap J\neq\varnothing.
\]
The interval \(J\) is \(G\)-isolated, so \(g(J)=J\), and therefore
\[
  h(I')=I'.
\]
This proves that \(I'\) is \(H\)-isolated.

We now compare the two normalized restrictions.  Since \(J\) is an
interval for \(G\leq F\), it is contained in \((0,1)\), the domain of
\(\nu_I\).  Thus the composition \(\nu_J(\nu_I|_J)\) is well defined
and maps \((0,1)\) onto \(I'\):
\[
  (0,1)\xrightarrow{\ \nu_J\ }J
  \xrightarrow{\ \nu_I|_J\ }I'.
\]
Every element of \(\operatorname{Stab}_H(I')\) also stabilizes \(I\),
since its image of \(I\) meets \(I\).  For
\(h\in\operatorname{Stab}_H(I)\) and
\[
  g=\nu_I(h|_I)\nu_I^{-1},
\]
we have
\[
  h(I')=I'
  \quad\Longleftrightarrow\quad
  g(J)=J.
\]
Conversely, every \(g\in\operatorname{Stab}_G(J)\) is the normalized
restriction of some \(h\in\operatorname{Stab}_H(I)\), and the displayed
equivalence shows that this \(h\) stabilizes \(I'\).  Consequently,
\[
  K_J(G)
  =
  \bigl(\nu_J(\nu_I|_J)\bigr)
  \bigl(\operatorname{Stab}_H(I')|_{I'}\bigr)
  \bigl(\nu_J(\nu_I|_J)\bigr)^{-1}.
\]
By definition,
\[
  K_{I'}(H)
  =
  \nu_{I'}
  \bigl(\operatorname{Stab}_H(I')|_{I'}\bigr)
  \nu_{I'}^{-1}.
\]
Both \(\nu_J(\nu_I|_J)\) and \(\nu_{I'}\) are increasing
homeomorphisms from \((0,1)\) onto \(I'\).  Thus
\[
  \phi=\nu_J(\nu_I|_J)\nu_{I'}^{-1}
\]
is an increasing homeomorphism of \((0,1)\), extending continuously to
an increasing homeomorphism of \([0,1]\).  The preceding equalities
give
\[
  \phi^{-1}K_J(G)\phi=K_{I'}(H).
\]

Finally, an element \(f\in F\) satisfies
\[
  \pi_{\mathrm{ab}}(f)=(0,0)
\]
if and only if it is the identity on a neighborhood of \(0\) and a
neighborhood of \(1\).  This property is preserved by conjugation by
\(\phi\).  Hence the conjugacy carries
\[
  \ker\bigl(\pi_{\mathrm{ab}}|_{K_J(G)}\bigr)
\]
onto
\[
  \ker\bigl(\pi_{\mathrm{ab}}|_{K_{I'}(H)}\bigr).
\]
It therefore induces an isomorphism
\[
  \pi_{\mathrm{ab}}\bigl(K_J(G)\bigr)
  \cong
  \pi_{\mathrm{ab}}\bigl(K_{I'}(H)\bigr),
\]
so these two subgroups of \(\mathbb Z^2\) have the same rank.
\end{proof}

Recall that the interval \((0,1)\) is \(H\)-isolated.  Its canonical
chart is the identity, so
\[
  K_{(0,1)}(H)=H.
\]
We allow this interval in the following characterization.

\begin{proposition}[Rank-one characterization of infinite path semigroupoids]
\label{prop:rank-one-isolated-characterization}
Let \(H\leq F\) be closed and suppose that \(\mathcal L=C(H)\) is
finite and full.  Then \(\Path(\mathcal L)\) is infinite if and only if
there is a rational \(H\)-isolated interval \(I\) such that
\[
  \operatorname{rank}\pi_{\mathrm{ab}}\bigl(K_I(H)\bigr)=1.
\]
\end{proposition}

\begin{proof}
If such an interval exists,
Corollary~\ref{cor:rank-one-isolated-interval} shows that
\(\Path(\mathcal L)\) is infinite.

Conversely, suppose that \(\Path(\mathcal L)\) is infinite.  We prove
the existence of \(I\) by induction on
\[
  \kappa(H)=\bigl|E_{\mathrm{in}}(C(H))\bigr|.
\]
Assume that the assertion holds for every closed subgroup of \(F\)
with finite full core and smaller inner-edge complexity.

By Lemma~\ref{lem:normalized-restriction}, applied to \((0,1)\), both
coordinate projections of \(\pi_{\mathrm{ab}}(H)\) are nonzero.  Its
rank is therefore either one or two.  If it has rank one, take
\[
  I=(0,1).
\]

Suppose that \(\pi_{\mathrm{ab}}(H)\) has rank two.  The group \(H\) is
finitely generated by Theorem~\ref{thm:core-closure}\textup{(ii)}.  If
its action on \((0,1)\) were minimal,
Corollary~\ref{cor:minimal-rank-two-finfty} would imply that
\(\Path(\mathcal L)\) is finite.  Thus the action is not minimal.

If \(H\) has a fixed point in \((0,1)\), apply
Corollary~\ref{cor:fixed-points-orbital-reduction}.  Otherwise, apply
Corollary~\ref{cor:exceptional-gap-reduction}.  In either case, we
obtain finitely many proper rational \(H\)-isolated intervals
\[
  I_1,\ldots,I_r
\]
such that, on putting
\[
  G_i=K_{I_i}(H),
\]
we have
\[
  \Path(\mathcal L)\text{ is finite}
  \quad\Longleftrightarrow\quad
  \Path(C(G_i))\text{ is finite for every }i.
\]
Since \(\Path(\mathcal L)\) is infinite, there is an index \(j\) for
which
\[
  \Path(C(G_j))
\]
is infinite.  Put
\[
  G=G_j.
\]
By Lemma~\ref{lem:normalized-restriction} and
Proposition~\ref{prop:normalized-core}, \(G\) is closed and \(C(G)\) is
finite and full.  Since \(I_j\) is proper,
Corollary~\ref{cor:strict-inner-edge-descent} gives
\[
  \kappa(G)<\kappa(H).
\]
The induction hypothesis therefore provides a rational \(G\)-isolated
interval \(J\) such that
\[
  \operatorname{rank}\pi_{\mathrm{ab}}\bigl(K_J(G)\bigr)=1.
\]
Apply Lemma~\ref{lem:iterated-normalized-restrictions} to \(I_j\) and
\(J\).  The interval
\[
  I=\nu_{I_j}(J)
\]
is rational and \(H\)-isolated, and
\[
  \operatorname{rank}\pi_{\mathrm{ab}}\bigl(K_I(H)\bigr)
  =
  \operatorname{rank}\pi_{\mathrm{ab}}\bigl(K_J(G)\bigr)
  =
  1.
\]
This completes the induction.
\end{proof}

Proposition~\ref{prop:rank-one-isolated-characterization} identifies a
rank-one normalized restriction of the original group \(H\) whenever
\(\Path(C(H))\) is infinite.
Theorem~\ref{thm:rank-one-isolated-not-FP2} then implies that \(H\) is
not of type \(\mathrm{FP}_2\).  In the next subsection, we combine this
implication with the finite-path theorem for diagram groups to show
that \(H\) is of type \(\mathrm{FP}_2\) if and only if it is of type
\(F_\infty\), and that both conditions are equivalent to finiteness of
\(\Path(C(H))\).

\subsection{Finiteness properties and decidability}
\label{subsec:finiteness-properties-decidability}

We now combine
Proposition~\ref{prop:rank-one-isolated-characterization} with the
rank-one obstruction from
Section~\ref{sec:rational-isolated-intervals} and the finite-path
theorem for diagram groups.

\begin{theorem}
\label{thm:finite-full-core-equivalence}
Let \(H\leq F\) be closed and suppose that its core
\(\mathcal L=C(H)\) is finite and full.  The following conditions are
equivalent.
\begin{enumerate}[label=\textup{(\arabic*)}]
\item \(H\) is of type \(\mathrm{FP}_2\).
\item \(\Path(\mathcal L)\) is finite.
\item \(H\) is of type \(F_\infty\).
\end{enumerate}
\end{theorem}

\begin{proof}
Suppose first that \(H\) is of type \(\mathrm{FP}_2\).  If
\(\Path(\mathcal L)\) were infinite,
Proposition~\ref{prop:rank-one-isolated-characterization} would provide
a rational \(H\)-isolated interval \(I\) such that
\[
  \operatorname{rank}\pi_{\mathrm{ab}}\bigl(K_I(H)\bigr)=1.
\]
Theorem~\ref{thm:rank-one-isolated-not-FP2} would then imply that \(H\)
is not of type \(\mathrm{FP}_2\), a contradiction.  Thus
\(\Path(\mathcal L)\) is finite, proving that \textup{(1)} implies
\textup{(2)}.

Now suppose that \(\Path(\mathcal L)\) is finite.  Since \(H\) is
closed, Theorem~\ref{thm:core-closure}\textup{(i)} gives an isomorphism
\[
  H\cong D(\mathcal L,p_H).
\]
The directed \(2\)-complex \(\mathcal L\) is finite and has only
finitely many homotopy classes of nonempty \(1\)-paths.
Theorem~\ref{thm:finite-path-semigroupoid} therefore shows that
\(D(\mathcal L,p_H)\), and hence \(H\), is of type \(F_\infty\).  Thus
\textup{(2)} implies \textup{(3)}.

Finally, every group of type \(F_\infty\) is of type
\(\mathrm{FP}_2\), so \textup{(3)} implies \textup{(1)}.
\end{proof}

Since every group of type \(F_\infty\) is finitely presented and every
finitely presented group is of type \(\mathrm{FP}_2\), the conditions
of Theorem~\ref{thm:finite-full-core-equivalence} are also equivalent
to \(H\) being finitely presented.

The same finite-path theorem applies to every nonempty base path in
\(\mathcal L\), not only to the distinguished edge \(p_H\).

\begin{corollary}
\label{cor:all-base-paths-finfty}
Let \(H\leq F\) be closed and of type \(\mathrm{FP}_2\), and suppose
that its core \(\mathcal L=C(H)\) is finite and full.  Then, for every
nonempty \(1\)-path \(w\) in \(\mathcal L\), the diagram group
\[
  D(\mathcal L,w)
\]
is of type \(F_\infty\).
\end{corollary}

\begin{proof}
By Theorem~\ref{thm:finite-full-core-equivalence},
\(\Path(\mathcal L)\) is finite.  Since \(\mathcal L\) is a finite
directed \(2\)-complex, Theorem~\ref{thm:finite-path-semigroupoid}
applies to every nonempty \(1\)-path \(w\) in \(\mathcal L\).
\end{proof}

We next show that the equivalent conditions of
Theorem~\ref{thm:finite-full-core-equivalence} can be decided from the
finite full core \(\mathcal L=C(H)\).

We first explain how to compute the rank of the endpoint image from a
finite core.  Let \(G\leq F\) be closed and suppose that
\(\mathcal M=C(G)\) is finite.  By
\cite[Remark~5.13]{GPSclosed}, a finite generating set
\[
  g_1,\ldots,g_s
\]
of \(G\), represented by tree-pair diagrams, can be computed from
\(\mathcal M\).  The extreme branches of these diagrams determine the
vectors \(\pi_{\mathrm{ab}}(g_j)\), and
\[
  \pi_{\mathrm{ab}}(G)
  =
  \left\langle
    \pi_{\mathrm{ab}}(g_1),\ldots,\pi_{\mathrm{ab}}(g_s)
  \right\rangle.
\]
Thus the rank of \(\pi_{\mathrm{ab}}(G)\) is the dimension over
\(\mathbb Q\) of the span of these finitely many vectors, and can be
computed from \(\mathcal M\).

\begin{theorem}
\label{thm:finite-full-core-decision}
There is an algorithm which takes as input a finite full core
\[
  \mathcal L=C(H)
\]
of a closed subgroup \(H\leq F\) and decides whether the equivalent
conditions of Theorem~\ref{thm:finite-full-core-equivalence} hold.

If these conditions do not hold, the algorithm also finds a rational
\(H\)-isolated interval \(I\) such that
\[
  \operatorname{rank}\pi_{\mathrm{ab}}\bigl(K_I(H)\bigr)=1.
\]
\end{theorem}

\begin{proof}
We describe a recursive procedure whose input at each stage is a finite
full core
\[
  \mathcal M=C(G)
\]
of a closed subgroup \(G\leq F\).  The procedure returns either a
positive answer or a negative answer together with a rational
\(G\)-isolated interval witnessing rank one.

Compute the rank of \(\pi_{\mathrm{ab}}(G)\) as explained above.  By
Lemma~\ref{lem:normalized-restriction}, applied to \((0,1)\), both
coordinate projections are nonzero, so this rank is either one or two.

\medskip\noindent\textup{(1) Rank one.}
If
\[
  \operatorname{rank}\pi_{\mathrm{ab}}(G)=1,
\]
return a negative answer together with the interval \((0,1)\).

\medskip\noindent\textup{(2) Rank two and minimal action.}
Suppose that
\[
  \operatorname{rank}\pi_{\mathrm{ab}}(G)=2.
\]
Use Theorem~\ref{thm:minimality-child-automaton} to decide whether the
action of \(G\) on \((0,1)\) is minimal.  If it is, return a positive
answer.

\medskip\noindent\textup{(3) Rank two and nonminimal action.}
Use the procedure in
Subsection~\ref{subsec:top-components-minimal-sets} to find the fixed
points of \(G\) in \((0,1)\).

If there are fixed points, let
\[
  I_1,\ldots,I_r
\]
be the orbitals of \(G\), as in
Subsection~\ref{subsec:fixed-points-orbital-reduction}.  If there are no
fixed points, construct the finite collection of gaps of the exceptional
minimal set given by Proposition~\ref{prop:gap-orbit-representatives},
and denote its members by
\[
  I_1,\ldots,I_r.
\]
In either case, these are proper rational \(G\)-isolated intervals.

For every \(i\), put
\[
  G_i=K_{I_i}(G).
\]
By Lemma~\ref{lem:normalized-restriction} and
Proposition~\ref{prop:normalized-core}, each \(G_i\) is closed and has
finite full core.  Construct these cores using
Proposition~\ref{prop:normalized-core-effective} and apply the
procedure recursively to each of them.

If every recursive application returns a positive answer, return a
positive answer.  Otherwise, choose an index \(j\) for which the
recursive application returns a negative answer, together with a
rational \(G_j\)-isolated interval \(J\).  Return a negative answer
together with
\[
  I=\nu_{I_j}(J).
\]

We first prove that the procedure terminates.  Recall that
\[
  \kappa(G)=|E_{\mathrm{in}}(C(G))|.
\]
Every recursive application in \textup{(3)} is made to a group \(G_i\)
for which
\[
  \kappa(G_i)<\kappa(G)
\]
by Corollary~\ref{cor:strict-inner-edge-descent}.  Moreover, only
finitely many groups \(G_i\) occur at each stage.  Termination therefore
follows by induction on \(\kappa(G)\).

We next prove, by the same induction, that the procedure returns a
positive answer exactly when
\[
  \Path(C(G))
\]
is finite.  In \textup{(1)}, the interval \((0,1)\) has normalized
restriction \(G\), so
Proposition~\ref{prop:rank-one-isolated-characterization} shows that
\(\Path(C(G))\) is infinite.  In \textup{(2)},
Corollary~\ref{cor:minimal-rank-two-finfty} shows that
\(\Path(C(G))\) is finite.

In \textup{(3)}, Corollary~\ref{cor:fixed-points-orbital-reduction} or
Corollary~\ref{cor:exceptional-gap-reduction} gives
\[
  \Path(C(G))\text{ is finite}
  \quad\Longleftrightarrow\quad
  \Path(C(G_i))\text{ is finite for every }i.
\]
By the induction hypothesis, the condition on the right holds exactly
when all the recursive applications return positive answers.  This
proves correctness.  Applying
Theorem~\ref{thm:finite-full-core-equivalence} at the initial input
\(G=H\) gives the required decision.

Finally, we verify the additional conclusion in the negative case.  In
\textup{(1)}, the returned interval \((0,1)\) has the required property.
In \textup{(3)}, the induction hypothesis gives
\[
  \operatorname{rank}\pi_{\mathrm{ab}}\bigl(K_J(G_j)\bigr)=1.
\]
By Lemma~\ref{lem:iterated-normalized-restrictions}, the interval
\[
  I=\nu_{I_j}(J)
\]
is rational and \(G\)-isolated, and
\[
  \operatorname{rank}\pi_{\mathrm{ab}}\bigl(K_I(G)\bigr)
  =
  \operatorname{rank}\pi_{\mathrm{ab}}\bigl(K_J(G_j)\bigr)
  =1.
\]
Its endpoints can be computed using the explicit formulas for the
canonical chart in Section~\ref{sec:rational-isolated-intervals}.  Thus
every negative answer is accompanied by the required interval.
\end{proof}

We now describe additional constructions available when the algorithm
of Theorem~\ref{thm:finite-full-core-decision} returns a positive
answer.

We use completeness of directed \(2\)-complexes in the sense of
\cite[Section~6]{GS06}.  Completeness means that the chosen rewriting
orientation is terminating and confluent, so every \(1\)-path has a
unique irreducible representative.  A \emph{finite completion} of \(\mathcal L\)
is a finite directed \(2\)-complex \(\widehat{\mathcal L}\) obtained
by adjoining \(2\)-cells \(f\) whose top and bottom paths are nonempty
and satisfy
\[
  \topPath(f)\simeq\botPath(f)
  \quad\text{in }\mathcal L,
\]
together with rewriting orientations for which
\(\widehat{\mathcal L}\) is complete.  In the construction used below,
the original core cells are oriented for rewriting as
\[
  e_0e_1\longrightarrow e.
\]

\begin{proposition}[Effective completion]
\label{prop:effective-completion}
Let \(H\leq F\) be closed and of type \(\mathrm{FP}_2\), and suppose
that its core \(\mathcal L=C(H)\) is finite and full.  From
\(\mathcal L\), one can compute representatives of all classes in
\(\Path(\mathcal L)\), their partially defined multiplication table,
and a finite completion \(\widehat{\mathcal L}\).

Moreover, for every nonempty \(1\)-path \(w\) in \(\mathcal L\), the
inclusion induces an injective homomorphism
\[
  j_w\colon D(\mathcal L,w)\longrightarrow
  D(\widehat{\mathcal L},w)
\]
with an effectively computable retraction
\[
  r_w\colon D(\widehat{\mathcal L},w)\longrightarrow
  D(\mathcal L,w).
\]
\end{proposition}

\begin{proof}
By Theorem~\ref{thm:finite-full-core-equivalence},
\(\Path(\mathcal L)\) is finite.  The category of homotopy classes of
\(1\)-paths in \(\mathcal L\), including the empty paths, is therefore
finite.  It is presented by the edges of \(\mathcal L\), subject to the
relations \(\topPath(f)=\botPath(f)\) for the cells \(f\) of
\(\mathcal L\).  The categorical enumeration procedure of
\cite[Theorem~6.1 and Section~7]{BushLeemingWalters} computes
representatives and its composition table.  Restricting to the
nonempty paths gives the required multiplication table.

With this table, the finite-completion construction in the proof of
\cite[Lemma~6.1]{GS06} is effective.  Their cell-replacement
construction in the proof of \cite[Lemma~5.1\textup{(4)}]{GS06} gives
the effective retractions \(r_w\).
\end{proof}

The presentation and homology constructions of Guba and Sapir can now
be applied to the completion, and the effective retractions give the
corresponding information for the original diagram groups.

\begin{corollary}[Effective presentations and homology]
\label{cor:effective-presentations-homology}
Let \(H\leq F\) be closed and of type \(\mathrm{FP}_2\), and suppose
that its core \(\mathcal L=C(H)\) is finite and full.  Given
\(\mathcal L\) and a nonempty \(1\)-path \(w\) in \(\mathcal L\), one
can construct a finite presentation of
\[
  D(\mathcal L,w),
\]
with its generators represented by diagrams over \(\mathcal L\).

Moreover, for every specified integer \(n\geq0\), one can compute the
rank of the finite-rank free abelian group
\[
  H_n\bigl(D(\mathcal L,w);\mathbb Z\bigr).
\]
In particular, a finite presentation of \(H\), with generators
represented by tree-pair diagrams, can be computed from \(\mathcal L\).
\end{corollary}

\begin{proof}
Construct \(\widehat{\mathcal L}\) and the retractions of
Proposition~\ref{prop:effective-completion}.  Since the completion is
finite and has finitely many irreducible \(1\)-paths, the construction
of \cite[Theorem~6.6]{GS06} gives a finite presentation of
\(D(\widehat{\mathcal L},w)\).
The effective retraction supplies the images of the generators under
the idempotent endomorphism \(r_wj_w\). Applying
\cite[Corollary~2.2]{GrovesWilton09} gives a finite presentation of
\(D(\mathcal L,w)\), with generators represented by their image diagrams.

For homology, apply the cell replacements in the proof of
\cite[Lemma~5.1\textup{(4)}]{GS06} coordinatewise to Squier cubes,
and combine this map with the chain reduction described immediately
before \cite[Lemma~9.6]{GS06}. Using cycle representatives for the basis of
\cite[Theorem~9.7]{GS06}, these constructions compute the matrix of
the idempotent induced by \(r_wj_w\) in each specified degree.
Its image is isomorphic to
\(H_n(D(\mathcal L,w);\mathbb Z)\), by the retract argument preceding
\cite[Theorem~9.9]{GS06}; its rank is then computable by integer
linear algebra.
Taking
\(w=p_H\) gives the final assertion.
\end{proof}

\section{Applications}
\label{sec:applications}

We apply the finiteness criterion of
Theorem~\ref{thm:finite-full-core-equivalence} to maximal subgroups,
actions on dyadic tuples, supported subgroups, and closed overgroups.

\begin{theorem}[Finitely generated maximal subgroups]
\label{thm:finitely-generated-maximal-subgroups-finfty}
Every finitely generated maximal subgroup of Thompson's group \(F\) is
of type \(F_\infty\).
\end{theorem}

\begin{proof}
Let \(H\) be a finitely generated maximal subgroup of \(F\).  If \(H\)
has finite index, the conclusion follows from the fact that \(F\) is of
type \(F_\infty\) \cite{BG84}, since this property passes to
finite-index subgroups.

Suppose that \(H\) has infinite index.  By
\cite[Theorem~1.1 and Lemma~5.6]{GolanMaximal}, \(H\) is closed and its
core
\[
  \mathcal L=C(H)
\]
is full.  Since \(H\) is finitely generated, \(\mathcal L\) is finite.
Moreover,
\[
  \pi_{\mathrm{ab}}(H)=\mathbb Z^2.
\]
Indeed, otherwise \(\pi_{\mathrm{ab}}(H)\) would be contained in a
proper finite-index subgroup of \(\mathbb Z^2\).  Its inverse image
would be a proper finite-index subgroup of \(F\) containing \(H\),
contrary to maximality and the assumption that \(H\) has infinite
index.

If \(H\) acts minimally on \((0,1)\),
Corollary~\ref{cor:minimal-rank-two-finfty} gives the result.  Assume,
therefore, that the action is not minimal.

Suppose first that \(H\) fixes \(c\in(0,1)\).  Then
\[
  H\leq\operatorname{Stab}_F(c)<F,
\]
so maximality gives
\[
  H=\operatorname{Stab}_F(c).
\]
Its two orbitals are \((0,c)\) and \((c,1)\).  On the dyadic points in
either orbital, its action is order-\(n\)-transitive for every \(n\):
any two increasing \(n\)-tuples in that orbital lie in the interior of
a closed dyadic interval \(J\) contained in the orbital, and
\[
  F[J]\leq H.
\]
The actions on the two sides can also be performed independently, using
elements supported on the respective sides.  Consequently, the orbit
of an increasing dyadic pair is determined by the positions of its
entries relative to \(c\), including possible equality when \(c\) is
dyadic.  There are only finitely many possibilities.

By Corollary~\ref{cor:interval-pair-dictionary} and
Lemma~\ref{lem:inner-reduction}\textup{(2)}, it follows that
\(\Path(\mathcal L)\) is finite.
Theorem~\ref{thm:finite-full-core-equivalence} therefore gives
that \(H\) is of type \(F_\infty\).  Since \(c\) is rational by
Lemma~\ref{lem:finite-fixed-set}, this case also follows from Farley
\cite[Theorem~1.2]{FarleyFixed}.  It also follows from the description
in \cite[Lemma~4.11]{GSstabilizers}, together with
\cite[Section~7.2, Exercise~3]{Geoghegan08}: the stabilizer of a
nondyadic rational point is an ascending HNN extension of
\(F\times F\), whereas the stabilizer of a dyadic point is
\(F\times F\).

Now suppose that \(H\) has no fixed point in \((0,1)\).  By
Proposition~\ref{prop:nonminimal-rank-two-alternative}, it has a unique
exceptional minimal set \(M\).  Let
\[
  S=\{f\in F:f(M)=M\}.
\]
Then \(H\leq S\).  The subgroup \(S\) is proper: otherwise \(M\) would
be a nonempty proper closed \(F\)-invariant subset of \((0,1)\),
contradicting minimality of the action of \(F\).  Hence maximality gives
\[
  H=S.
\]
Let \(I\) be a gap of \(M\).  Every subgroup \(F[J]\), where
\(J\subseteq I\) is a closed dyadic interval, fixes \(M\) pointwise.
Thus
\[
  F[J]\leq H.
\]
Given two closed dyadic intervals contained in \(I\), choose such a
\(J\) containing both in its interior.  Order-\(2\)-transitivity of
\(F[J]\) gives an element of \(H\) carrying the endpoints of one
interval to those of the other.  Proposition~\ref{prop:transport}
therefore implies that
\[
  |\Path_I(H)|=1.
\]
Every gap is a rational \(H\)-isolated interval by
Propositions~\ref{prop:nonminimal-rank-two-alternative} and
\ref{prop:gap-orbit-representatives}.
Proposition~\ref{prop:isolated-interval-path-correspondence} identifies
\(\Path_I(H)\) with
\[
  \Path_{\mathrm{in}}\bigl(C(K_I(H))\bigr).
\]
Since the normalized core is finite and full,
Lemma~\ref{lem:inner-reduction}\textup{(2)} shows that
\[
  \Path\bigl(C(K_I(H))\bigr)
\]
is finite.

In particular, this holds for the finitely many representative gaps of
Proposition~\ref{prop:gap-orbit-representatives}.
Corollary~\ref{cor:exceptional-gap-reduction} gives finiteness of
\(\Path(\mathcal L)\), and
Theorem~\ref{thm:finite-full-core-equivalence} yields
that \(H\) is of type \(F_\infty\).
\end{proof}

In particular, this answers affirmatively
\cite[Problem~8.8]{GolanMaximal}, which asks whether every finitely
generated maximal subgroup of \(F\) is finitely presented.

We next characterize the groups covered by
Theorem~\ref{thm:finite-full-core-equivalence} in terms of their actions
on dyadic tuples.  Recall that an action on a set \(X\) is
\emph{oligomorphic} if its diagonal action on \(X^n\) has finitely many
orbits for every \(n\geq1\).

\begin{proposition}[Dyadic tuple orbits]
\label{prop:dyadic-tuple-orbits}
Let \(H\leq F\) be closed, and suppose that \(\mathcal L=C(H)\) is
finite and full.  The following conditions are equivalent.
\begin{enumerate}[label=\textup{(\arabic*)}]
\item \(H\) is of type \(\mathrm{FP}_2\).
\item \(H\) has finitely many orbits on
      \[
        \D_{<}^{\,2}
        =
        \{(a,b)\in\D^2:a<b\}.
      \]
\item The action of \(H\) on \(\D\) is oligomorphic.
\end{enumerate}
\end{proposition}

\begin{proof}
By Corollary~\ref{cor:interval-pair-dictionary}, the \(H\)-orbits on
\(\D_{<}^{\,2}\) are in bijection with
\(\Path_{\mathrm{in}}(\mathcal L)\).  Since \(\mathcal L\) is finite
and full, Lemma~\ref{lem:inner-reduction}\textup{(2)} gives
\[
  \Path_{\mathrm{in}}(\mathcal L)\text{ is finite}
  \quad\Longleftrightarrow\quad
  \Path(\mathcal L)\text{ is finite}.
\]
Theorem~\ref{thm:finite-full-core-equivalence} therefore proves the
equivalence of \textup{(1)} and \textup{(2)}.

Assume \textup{(2)}.  Then \(\Path_{\mathrm{in}}(\mathcal L)\) is
finite.  For an increasing dyadic tuple
\[
  a_1<\cdots<a_n,
\]
Proposition~\ref{prop:tuple-transport} shows that its orbit is
determined by the types
\[
  v_H(a_1),\ldots,v_H(a_n)
\]
and, when \(n\geq2\), the consecutive interval classes
\[
  \omega_H[a_1,a_2],\ldots,\omega_H[a_{n-1},a_n].
\]
There are finitely many inner vertices and finitely many inner path
classes, so there are only finitely many orbits on increasing
\(n\)-tuples for every \(n\).  An arbitrary tuple in \(\D^n\) is
determined by its equality-and-order pattern together with the
increasing tuple of its distinct entries.  For fixed \(n\), there are
only finitely many such patterns.  Hence \(H\) has finitely many orbits
on \(\D^n\), proving \textup{(3)}.

Finally, \textup{(3)} implies \textup{(2)}, since \(\D_{<}^{\,2}\) is
an \(H\)-invariant subset of \(\D^2\).
\end{proof}

We now consider subgroups supported in dyadic intervals.  The intervals
in the next proposition need not be \(H\)-isolated.

\begin{proposition}[Supported subgroups]
\label{prop:supported-subgroups}
Let \(H\leq F\) be closed, suppose that \(\mathcal L=C(H)\) is finite
and full, and let
\[
  J=[a,b]\subseteq[0,1]
\]
be a closed interval with dyadic endpoints.  Let \(W_J\)
be the \(1\)-path associated with a dyadic subdivision of \(J\).  Then
\[
  H\cap F[J]\cong D(\mathcal L,W_J).
\]
If \(H\) is of type \(\mathrm{FP}_2\), then \(H\cap F[J]\) is of type
\(F_\infty\), and a finite presentation of this subgroup can be
constructed effectively from \(\mathcal L\) and \(J\).
\end{proposition}

\begin{proof}
Extend the chosen subdivision of \(J\) to a dyadic subdivision of
\([0,1]\), associated with a finite binary tree \(T\).  Its positive
diagram over \(\mathcal L\) has the form
\[
  \Psi_T\colon p_H\longrightarrow_{\mathcal L}W_0 W_J W_1,
\]
where \(W_0\) and \(W_1\) are the paths associated with the
subdivisions of \([0,a]\) and \([b,1]\), respectively.

For a spherical \((W_J,W_J)\)-diagram \(\Delta\), form
\[
  \Psi_T\circ
  \bigl(\varepsilon(W_0)+\Delta+\varepsilon(W_1)\bigr)
  \circ\Psi_T^{-1}.
\]
If \(W_0\) or \(W_1\) is empty, the corresponding summand is omitted.  This
gives a spherical diagram based at \(p_H\), and hence an element of
\(H\).  Its action is the identity outside \(J\).

Adding trivial exterior diagrams is injective on diagram groups: if
\(\Delta\) is reduced and nontrivial, then the padded diagram is also
reduced and nontrivial.  Conjugation by \(\Psi_T\) is an isomorphism
between the corresponding base-path diagram groups.  Thus the
construction gives an injective homomorphism
\[
  D(\mathcal L,W_J)\longrightarrow H\cap F[J].
\]

To prove surjectivity, let \(h\in H\cap F[J]\).  Choose an accepted tree
pair for \(h\) whose two trees refine \(T\).  Since \(h\) is the
identity outside \(J\), the two trees agree there.  Their positive
diagrams therefore decompose over \(\Psi_T\) into matching exterior
forests and two forests based at \(W_J\).  The matching exterior forests
cancel, leaving a spherical diagram over \(W_J\).  Hence \(h\) belongs
to the image of the displayed homomorphism.

If \(H\) is of type \(\mathrm{FP}_2\),
Corollary~\ref{cor:all-base-paths-finfty} gives type \(F_\infty\) for
\(D(\mathcal L,W_J)\).  The effective presentation follows from
Corollary~\ref{cor:effective-presentations-homology}, since \(W_J\) can
be computed from the chosen subdivision of \(J\).
\end{proof}

\begin{corollary}[Stabilizers of finite dyadic sets]
\label{cor:finite-dyadic-set-stabilizers}
Let \(H\leq F\) be closed and of type \(\mathrm{FP}_2\), and suppose
that \(C(H)\) is finite and full.  For every finite set \(S\subseteq\D\),
its stabilizer in \(H\) is of type \(F_\infty\).  A finite presentation
of the stabilizer can be constructed effectively from \(C(H)\) and
\(S\).
\end{corollary}

\begin{proof}
Write \(S=\{x_1<\cdots<x_k\}\), put \(x_0=0\) and \(x_{k+1}=1\), and
let \(J_i=[x_i,x_{i+1}]\).  If \(h\in H\) fixes \(S\) pointwise, then
for every \(i\) the map which agrees with \(h\) on \(J_i\) and is the
identity elsewhere is piecewise \(H\), and hence lies in \(H\) because
\(H\) is closed.  Therefore
\[
  H_S\cong\prod_{i=0}^{k}\bigl(H\cap F[J_i]\bigr).
\]
Each factor is of type \(F_\infty\) and has an effectively computable
finite presentation by Proposition~\ref{prop:supported-subgroups}.  The
same conclusions hold for their finite direct product.
\end{proof}

The finite path semigroupoid also forces the existence of copies of
\(F\) supported in arbitrarily small intervals.

\begin{proposition}[Locally supported copies of \(F\)]
\label{prop:locally-supported-copies-of-F}
Let \(H\leq F\) be closed and of type \(\mathrm{FP}_2\), and suppose
that its core \(\mathcal L=C(H)\) is finite and full.  For every
nonempty open interval \(U\subseteq(0,1)\), there is a subgroup of \(H\)
isomorphic to \(F\) whose elements are the identity outside \(U\).
\end{proposition}

\begin{proof}
Choose a standard dyadic interval
\[
  J=[u]\subseteq U,
\]
and put \(e=u^+\).  By Proposition~\ref{prop:supported-subgroups},
\[
  D(\mathcal L,e)\cong H\cap F[J].
\]
Since \(\mathcal L\) is full, \(e\) has positive expansions of
arbitrarily large length.  Since \(\Path(\mathcal L)\) is finite by
Theorem~\ref{thm:finite-full-core-equivalence}, two distinct nonempty
prefixes of one such expansion represent the same path class.  Write
these prefixes as \(P\) and \(PB\), and the whole expansion as \(PBA\).
Then \(B\) is nonempty,
\[
  P\simeq PB,
\]
and hence
\[
  e\simeq PB^nA
  \qquad(n\geq1).
\]
The finite semigroup \(\{[B^n]:n\geq1\}\) contains an idempotent.  Thus
an idempotent represented by a nonempty path divides \([e]\), and
\cite[Theorem~4]{GubaSapirRigidity} gives a copy of \(F\) in
\(D(\mathcal L,e)\).  Under the displayed isomorphism, this copy lies in
\(H\cap F[J]\), so its elements are the identity outside \(U\).
\end{proof}

The quotient-core argument in
\cite[Corollary~4.10 and the proof of Theorem~5.4]{GolanMaximal} shows
that a subgroup with finite full core has only finitely many closed
overgroups in \(F\), all of which have finite full core.  We now show
that if the original subgroup is closed and of type \(\mathrm{FP}_2\),
then all these overgroups are of type \(F_\infty\).

\begin{proposition}[Closed overgroups]
\label{prop:closed-overgroups-finfty}
Let \(H\leq F\) be closed and of type \(\mathrm{FP}_2\), and suppose
that \(C(H)\) is finite and full.  Then every closed subgroup
\(G\leq F\) containing \(H\) is of type \(F_\infty\).
\end{proposition}

\begin{proof}
By the result recalled above, \(C(G)\) is finite and full.
Proposition~\ref{prop:dyadic-tuple-orbits} shows that \(H\) has finitely
many orbits on increasing dyadic pairs.  Every \(G\)-orbit is a union
of \(H\)-orbits, so \(G\) also has finitely many such orbits.  Applying
Proposition~\ref{prop:dyadic-tuple-orbits} to \(G\), followed by
Theorem~\ref{thm:finite-full-core-equivalence}, gives
that \(G\) is of type \(F_\infty\).
\end{proof}

\bigskip
\noindent
\textsc{Department of Mathematics, Ben-Gurion University of the Negev,
Be'er Sheva, Israel}

\noindent
\textit{Email address}: \texttt{golangi@bgu.ac.il}

\end{document}